\documentclass[a4paper,12pt]{article}

\usepackage[utf8]{inputenc}
\usepackage{amsmath}
\usepackage{amsfonts}
\usepackage{amssymb}
\usepackage{amsthm}
\usepackage{graphicx,caption}
\usepackage{tikz-cd}
\usepackage{mathtools}
\usepackage{pgfplots}
\usepackage{enumitem}
\usepackage{faktor}
\usepackage{tabularx}
\usepackage{hyperref}
\usepackage{placeins}
\usepackage{enumitem}

\usepackage{vmargin}
\setmarginsrb{2.5 cm}{2 cm}{2.5 cm}{2 cm}{0.5 cm}{1 cm}{0.5 cm}{1 cm}

\hypersetup{colorlinks = true, linkbordercolor = {white}, linkcolor = {blue}, citecolor = {blue}}
\usetikzlibrary{matrix}

\numberwithin{subsection}{section}
\newtheoremstyle{1}
{6pt} 
{0pt} 
{\itshape} 
{} 
{\bfseries} 
{.} 
{.5em} 
{} 

\newtheoremstyle{2}
{6pt} 
{0pt} 
{} 
{} 
{\bfseries} 
{.} 
{.5em} 
{} 

\theoremstyle{1}

\newtheorem{theorem}{Theorem}[section]

\newtheorem{lemma}[theorem]{Lemma}
\newtheorem{proposition}[theorem]{Proposition}
\newtheorem{corollary}[theorem]{Corollary}

\theoremstyle{2}

\newtheorem{rmrk}[theorem]{Remark}

\DeclareMathOperator{\id}{id}
\DeclareMathOperator{\Id}{Id}
\DeclareMathOperator{\End}{End}
\DeclareMathOperator{\Ortho}{O}
	
\DeclareMathOperator{\Sp}{Sp}

\DeclareMathOperator{\U}{U}

\DeclareMathOperator{\SU}{SU}
\DeclareMathOperator{\SO}{SO}
\DeclareMathOperator{\PGL}{PGL}

\DeclareMathOperator{\GL}{GL}

\DeclareMathOperator{\cusp}{cusp}

\DeclareMathOperator{\vol}{vol}

\DeclareMathOperator{\Irr}{Irr}

\DeclareMathOperator{\stab}{stab}

\DeclareMathOperator{\gen}{gen}

\DeclareMathOperator{\im}{im}
\DeclareMathOperator{\re}{re}

\DeclareMathOperator{\supp}{supp}
\DeclareMathOperator{\Symm}{Sym}

\DeclareMathOperator{\rank}{rank}
\DeclareMathOperator{\defin}{def}
\DeclareMathOperator{\Ad}{Ad}
\DeclareMathOperator{\ad}{ad}
\DeclareMathOperator{\tr}{tr}
\DeclareMathOperator{\Hom}{Hom}

\begin{document}

\title{Large values of Maass forms on hyperbolic Grassmannians in the volume aspect}
\author{Thibaut Ménès}
\date{ }
\maketitle

\abstract
Let $n > m \geq 1$ be integers such that $n + m \geq 4$ is even. We prove the existence, in the volume aspect, of exceptional Maass forms on compact quotients of the hyperbolic Grassmannian of signature $(n,m)$. The method builds upon the work of Rudnick and Sarnak \cite{RudnickSarnak}, extended by Donnelly \cite{Donnelly} and then generalized by Brumley and Marshall \cite{BM2} to higher rank manifolds. It combines a counting argument with a period relation, showing that a certain period distinguishes theta lifts from an auxiliary group. The congruence structure is defined with respect to this period and the auxiliary group is either $\U(m,m)$ or $\Sp_{2m}(\mathbb{R})$, making $(\U(n,m), \U(m,m))$ or $(\Ortho(n,m), \Sp_{2m}(\mathbb{R}))$ a type $1$ dual reductive pair. The lower bound is naturally expressed, up to a logarithmic factor, as the ratio of the volumes, with the principal congruence structure on the auxiliary group.

\tableofcontents

\section{Introduction}
\subsection{The sup norm problem}
\subsubsection{Spectral aspect}
Let $X$ be a compact Riemannian manifold of dimension $n$. The eigenfunctions of the Laplace operator $\Delta$, acting on $L^2(X)$, are of importance in various fields such as quantum chaos and harmonic analysis. Typical questions ask about asymptotics, distribution and multiplicities of the eigenvalues but also about point-wise bounds of the eigenfunctions in terms of their eigenvalues. Regarding the latter, the basic upper bound proved by Levitan \cite{Lev} and Avacumomi\'{c} \cite{Ava} is
\begin{equation} \label{eq1}
	\Vert f \Vert_{\infty} \ll \lambda^{\frac{n - 1}{2}} \quad (\Delta f = -\lambda^2 f).
\end{equation}
This bound is sharp on the $n$-sphere but is not optimal on negatively curved manifolds as Bérard \cite{Berard} proved that
\begin{equation*}
	\Vert f \Vert_{\infty} \ll \frac{\lambda^{\frac{n - 1}{2}}}{(\log \lambda)^{\frac{1}{2}}}  \quad (\Delta f = -\lambda^2 f).
\end{equation*}
In fact, when $X$ is negatively curved, its geodesic flow is ergodic and quantum ergodicity holds. More precisely, if $\lbrace f_i \rbrace_i$ is an orthonormal basis of $L^2(X)$ of Laplace eigenfunctions satisfying $(\Delta + \lambda_i^2) f_i = f_i$ and if this basis is ordered by eigenvalues so that $0 = \lambda_1 \leq \lambda_2 \leq \ldots$ then a density one subsequence of eigenfunctions equidistributes \textit{i.e.} there is a density one subsequence $(f_{i_j})_j$ such that \cite{Shnirelman, ColindeVerdiere, Zelditch}
\begin{equation*}
\lim_{j \rightarrow \infty} \int_X \phi(x) \vert f_{i_j} (x) \vert^2 dx = \int_X \phi(x) dx \quad (\phi \in C^{\infty}(X)).
\end{equation*}
In particular, the strong delocalization bound
\begin{equation*}
	\Vert f_{i} \Vert_{\infty} \ll_{\epsilon} \lambda_i^{\epsilon}
\end{equation*}
can reasonably be expected to hold with density $1$ \cite{BM1}. A sequence of functions violating the above strong delocalization bound is said to be \textit{exceptional}. In dimension $2$, Iwaniec and Sarnak \cite{IwaniecSarnak} conjecture that exceptional sequences of eigenfunctions do not exist on hyperbolic surfaces. However, Rudnick and Sarnak \cite{RudnickSarnak} (see also \cite{MilicevicLarge}) proved that certain arithmetic hyperbolic $3$-manifolds do support exceptional sequences.

These arithmetic manifolds arise as locally symmetric spaces of non-compact type, that is, the quotient of a globally symmetric space of non-compact type by a lattice. More precisely, locally symmetric spaces of non-compact type are constructed by taking a non-compact semi-simple real Lie group $G$, a maximal compact subgroup $K \subset G$ and by considering the quotient of the globally symmetric space $S = G/K$ by a lattice $\Gamma \subset G$. As we will focus our attention on arithmetic manifolds, we shall assume that $\Gamma$ enjoys some additional arithmetic properties.\\
For example, let $V$ be a non-degenerate $\mathbb{Q}$-anisotropic quadratic space of signature $(3,1)$, let $G = \SO(V)^0$ (the identity component of the orthogonal group of $V$) and let $K \subset G$ be a maximal compact subgroup of $G$. Fixing a $\mathbb{Z}$-lattice $L \subset V$ and considering the subgroup $\Gamma$ of $G$ preserving $L$, one obtain an arithmetic $3$-dimensional compact locally symmetric space of non-compact type and of rank one by setting $X = \Gamma \backslash G / K$.

In the case described above, Rudnick and Sarnak \cite[Theorem 1.3]{RudnickSarnak} proved that there is a subsequence $(f_{i_j})_j$ of $L^2$-normalized eigenfunctions in $L^2(X)$ such that
\begin{equation} \label{RSBOUND}
\Vert f_{i_j} \Vert_{\infty} \gg \lambda_{i_j}^{\frac{1}{2}}.
\end{equation}
However, this does not contradicts quantum ergodicity because this subsequence is of density zero. By considering a totally real quadratic extension $K$ of $\mathbb{Q}$ and a non-degenerate $K$-anisotropic quadratic form of signature $(n,1)$, one may construct a $n$-dimensional compact locally symmetric space of rank one using the same mechanism described above. Donnelly \cite{Donnelly} proved, for $n \geq 5$, that these arithmetic manifolds support exceptional sequences, more precisely that there is a subsequence $(f_{i_j})_j$ of $L^2$-normalized eigenfunctions such that
\begin{equation} \label{DONBOUND}
\Vert f_{i_j} \Vert_{\infty} \gg \lambda_{i_j}^{\frac{(n-4)}{2}}.
\end{equation}
Brumley and Marshall \cite{BM2} (see also \cite{BM1, LapidOffen}) generalized the above results to the case where the rank of $X$ is greater than one, with $X$ defined similarly with respect to an anisotropic quadratic form. The form produced are not only Laplace eigenfunctions but joint eigenfunctions of $\mathcal{D}(S)$, the algebra of $G$-invariant differential operators on ${S = G/K}$, the globally symmetric space associated to $G$.\\
This algebra is (isomorphic to) a finitely generated polynomial ring in $\rank(G)$ variables with $\mathcal{D}(S) = \mathbb{C}[\Delta]$ when the rank of $G$ is $1$. In particular, $\mathcal{D}(S)$ is commutative so that, when the rank is greater than $1$, $L^2(X)$ can be further diagonalized: a \textit{Maass form} is a joint eigenfunction $f \in L^2(X)$ of $\mathcal{D}(S)$. Since $\Delta \in \mathcal{D}(S)$, Maass forms are also Laplace eigenfunctions.

While Laplace eigenfunctions are parametrized by scalars (namely, their eigenvalue) Maass forms are parametrized by vectors in a space of dimension the rank of $X$. Let $\mathfrak{g} = \mathfrak{k} \oplus \mathfrak{p}$ be the Cartan decomposition of the Lie algebra $\mathfrak{g}$ of $G$, where $\mathfrak{k}$ is the Lie algebra of $K$, and let $\mathfrak{a} \subseteq \mathfrak{p}$ be a maximal abelian subspace. Given a Maass form $f \in L^2(X)$, there is a character $\chi : \mathcal{D}(S) \rightarrow \mathbb{C}$ such that
\begin{equation*}
	D \cdot f = \chi (D) f \quad (D \in \mathcal{D}(S)).
\end{equation*}
Via the Harish-Chandra isomorphism (see Section \ref{sec3.2}) this character $\chi$ corresponds to a unique (up to the Weyl group) element $\mu \in \mathfrak{a}^*_{\mathbb{C}}$. We call $\mu$ the \textit{spectral parameter} of $f$.

If $f$ is a Maass form with \textit{sufficiently regular} spectral parameter $\nu \in \mathfrak{a}^*_{\mathbb{C}}$ (see Section \ref{sec3.5}), then Sarnak \cite{SarnakLetter} proved the bound
\begin{equation*}
\Vert f \Vert_{\infty} \ll \beta_S (\nu)^{\frac{1}{2}}
\end{equation*}
where $\beta$ is the \textit{spectral density function} of $S$ (see Section \ref{sec3.5}). This bound may be stated only in term of the Laplace eigenvalue $\lambda$ of $f$ as
\begin{equation*}
\Vert f \Vert_{\infty} \ll \lambda^{\frac{n-r}{2}}
\end{equation*}
where $r$ is the rank of $X$; recovering \eqref{eq1} when $r = 1$.

In \cite{BM1}, Brumley and Marshall introduced a broad class of higher-rank manifolds that support exceptional sequences of Laplace eigenfunctions. More precisely, they proved that there is $\delta > 0$ (logically ineffective) and a subsequence $(f_{i_j})_j$ of Laplace eigenfunctions such that
\begin{equation} \label{eqqqq}
	\Vert f_{i_j} \Vert_{\infty} \gg \lambda_{i_j}^{\delta}.
\end{equation}
In a subsequent work \cite{BM2}, the same authors established a similar result for congruence quotients of the hyperbolic Grassmannian $\mathbb{H}^{n,m}$ of signature $(n,m)$, which parametrizes negative definite $m$-dimensional subspaces of $\mathbb{R}^{n+m}$. More precisely, if $n > m \geq 1$ and $n+m \geq 4$ is even, they proved that, for sufficiently regular $\nu \in i \mathfrak{a}^*$, there is a Maass form $f$ of spectral parameter $\nu + O(1)$ such that (up to a logarithmic factor)
\begin{equation*}
	\Vert f \Vert_{\infty} \gg \dfrac{\beta_{\mathbb{H}^{n,m}}(\nu)^{\frac{1}{2}}}{\beta_{\mathcal{H}^{m}}(\nu)^{\frac{1}{2}}}
\end{equation*}
where $\mathcal{H}^{m}$ is the Siegel upper half-space of rank $m$. Note that, for well-balanced\footnote[1]{$\nu$ satisfies $1 + \vert \langle \alpha, \nu \rangle \vert \asymp 1 + \Vert \nu \Vert$ for all roots $\alpha$.} $\nu$, the above lower bound may be stated in term of the Laplace eigenvalue $\lambda$ of $f$ as
\begin{equation*}
	\Vert f \Vert_{\infty} \gg \lambda^{\frac{d-r-r^2}{2}}
\end{equation*}
where $d = nm$ is the dimension of $\mathbb{H}^{n,m}$ and $r = m$ its rank. Since $\mathbb{H}^{n,m}$ enters the class of manifolds described in \cite{BM1}, the interest in this lower bound lies in the quality of the exponent rather than in the existence of the latter (the exponent $\delta$ in \eqref{eqqqq} being logically ineffective).\\
Note that, if $r = 1$, the lower bound reads
\begin{equation*}
	\Vert f \Vert_{\infty} \gg \lambda^{\frac{d-2}{2}}
\end{equation*}
which improves upon the bound of Donnelly \eqref{DONBOUND} for $d = n \geq 5$ odd. 
\subsubsection{Volume aspect} \label{secVolasp}
In the above discussion, often referred to as the spectral aspect of the sup-norm problem, the estimates depend, in an unspecified manner, on the underlying manifold or, equivalently, on the lattice $\Gamma$. The volume\footnote[2]{Although the term "level aspect" is commonly used in the literature, we prefer "volume aspect" because, even though our lattices are congruence lattices, the volume is the natural quantity for our method.} aspect of the sup-norm problem asks for sup-norm bounds on Laplace eigenfunctions in term of the volume of the underlying manifold. In other words, we fix the Laplace eigenvalue $\lambda$ while allowing $\Gamma$ to vary among lattices in $G$. Before stating our main result, we provide benchmarks on the behavior of Maass forms in the volume aspect and define exceptional forms in contrast to their average behavior.

Let $G$ be a non-compact real reductive Lie group, let $K \subset G$ be a maximal compact subgroup and set $S = G/K$ for the associated globally symmetric space. Given a lattice $\Gamma \subset G$, we denote the corresponding locally symmetric space by $X_{\Gamma} = \Gamma \backslash S$. We fix a Haar measure $dg$ on $G$, a Haar measure $dk$ on $K$ and let $d\mu$ be the quotient measure on $S$ defined by $dg$ and $dk$. For a lattice $\Gamma \subset G$, we denote by $d\mu_{\Gamma}$ the quotient measure on $X_{\Gamma}$ defined by $d\mu$ and the counting measure on $\Gamma$, that is, such that
\begin{equation*}
\int_{X_{\Gamma}} \sum_{\gamma \in \Gamma} \phi(\gamma \cdot x) d\mu_{\Gamma}(x) = \int_{S} \phi (x) d\mu(x) \quad (\phi \in C_c^{\infty}(S)).
\end{equation*}
We set $d\overline{\mu_{\Gamma}}$ for the probability measure on $X_{\Gamma}$, that is, $d\overline{\mu_{\Gamma}} = \frac{d\mu_{\Gamma}}{\vol(X_{\Gamma},d\mu_{\Gamma})}$.

Assume there is a connected simple Lie group $G_s$ with maximal compact subgroup $K_s$ such that $G/K = G_{s}/K_{s}$. In particular, the center of $G$ is compact and $S = G/K$ is irreducible. Moreover, assume $\mathcal{D}(G/K) = \mathcal{D}(G_{s}/K_{s})$: for example, if $n > m \geq 1$ are integers, we show in Section \ref{chap6} (see the paragraph following Proposition \ref{prop6.0.1}) that $G = \Ortho(n,m)$ and $G = \U(n,m)$ satisfy these hypothesis, with $G_s = \SO^0(n,m)$ and $G_s = \SU(n,m)$ respectively.

Given a lattice $\Gamma \subset G$, we let $\lbrace f^{\Gamma}_{\mu} \rbrace_{\mu}$ denote an orthonormal basis of $L^2(X_{\Gamma}, d\overline{\mu_{\Gamma}})$ consisting of Maass forms of spectral parameter $\mu \in \mathfrak{a}_{\mathbb{C}}^*$. In Section \ref{chap6} we prove that, for uniform $\Gamma \subset G$, there is $Q \geq 1$ such that, for any point $x \in X_{\Gamma}$ and any $\nu \in i \mathfrak{a}^*$, we have\footnote[1]{We actually prove a stronger result with explicit dependence on $\nu$ but, for expository purposes, we emphasize only the asymptotic behavior in terms of the volume.}
\begin{equation} \label{eq1.1}
\sum_{\Vert \im(\mu) - \nu \Vert \leq Q} \vert f^{\Gamma}_{\mu} (x) \vert^2 \asymp_{\nu} \vol (X_{\Gamma}, d\mu_{\Gamma}).
\end{equation}
Since the number of terms in this sum is asymptotically the volume of $X_{\Gamma}$, one deduces that the average size of point evaluation of Maass forms (in a spectral ball of radius $Q$ about $\nu$) is $\asymp 1$. Given this average behavior, we say that a Maass form takes \textit{large values} if its sup norm grows by at least a positive power of the volume. Dropping all but one term in \eqref{eq1.1}, one can derive the upper bound
\begin{equation} \label{locbo}
\Vert f^{\Gamma}_{\nu} \Vert_{\infty} \ll \vol(X_{\Gamma}, d\mu_{\Gamma})^{\frac{1}{2}}.
\end{equation}
We call this the \textit{local bound} as its proof only takes into account local informations about a given point.

In this context, Brumley and Marshall \cite{BM1} proved that a wide range of manifolds support exceptional forms. In this text, we achieve the same result for certain congruence manifolds, namely hyperbolic Grassmannians as in \cite{BM2}. Since these manifolds fall within the framework of \cite{BM1}, the existence of exceptional forms is already established. However, the exponent provided in \textit{loc. cit.} is (logically) ineffective: the interest of Theorem \ref{theo1.2.1} given below lies in the quality of the exponent.
\subsection{Main result} \label{sec1.2}
For integers $n \geq m \geq 1$, let $\mathbb{H}^{n,m}_{\mathbb{C}}$ (resp. $\mathbb{H}^{n,m}_{\mathbb{R}}$) denote the complex (resp. real) hyperbolic Grassmannian of signature $(n,m)$, parametrizing negative definite $m$-dimensional subspaces in $\mathbb{C}^{n+m}$ (resp. $\mathbb{R}^{n+m}$). This is a simply connected symmetric space of non-compact type, of rank $m$ and dimension $2nm$ (resp. $nm$). The group $\U(n,m)$ (resp. $\Ortho(n,m)$) acts transitively on $\mathbb{H}^{n,m}_{\mathbb{C}}$ (resp. $\mathbb{H}^{n,m}_{\mathbb{R}}$) by isometries and the stabilizer of $\mathbb{C}^n \times \lbrace 0 \rbrace$ (resp. $\mathbb{R}^n \times \lbrace 0 \rbrace$) is $\U(n) \times \U(m)$ (resp. $\Ortho(n) \times \Ortho(m)$).\\
Note that the connected simple Lie group $\SU(n,m)$ (resp. $\SO^0(n,m)$) also acts transitively on $\mathbb{H}^{n,m}_{\mathbb{C}}$ (resp. $\mathbb{H}^{n,m}_{\mathbb{R}}$) and the stabilizer of $\mathbb{C}^n \times \lbrace 0 \rbrace$ (resp. $\mathbb{R}^n \times \lbrace 0 \rbrace$) is $\text{S}(\U(n) \times \U(m))$ (resp. $\SO(n) \times \SO(m)$). Hence
\begin{equation*}
\mathbb{H}^{n,m}_{\mathbb{C}} = \U(n,m) / \U(n) \times \U(m) = \SU(n,m) / \text{S}(\U(n) \times \U(m))
\end{equation*}
and
\begin{equation*}
\mathbb{H}^{n,m}_{\mathbb{R}} = \Ortho(n,m) / \Ortho(n) \times \Ortho(m) = \SO^0(n,m) / \SO(n) \times \SO(m).
\end{equation*}
The main result discusses the existence of Maass forms with large values on certain congruence quotients of $\mathbb{H}^{n,m}_{\mathbb{C}}$ and $\mathbb{H}^{n,m}_{\mathbb{R}}$. We prove this in the setting of disjoint union of such locally symmetric spaces: these disconnected manifolds arise as adelic double quotients; more precisely, adelic double quotients associated to the isometry group $G$ of a rational Hermitian or quadratic form.

As explained in \cite[§1]{BM2}, the method is based on the article of Rudnick and Sarnak \cite{RudnickSarnak} and the main idea may be expressed very simply as follows. Given a subgroup $H$ of $G$, one can show the existence of Maass forms with large $H$-periods (\textit{i.e.} larger than average) provided
\begin{enumerate}
\item one knows the average size of $H$-periods; \label{step1}
\item one can show that the non-vanishing of $H$-periods characterizes the image of a functorial lift from an auxiliary group $G'$;\label{step2}
\item one can count the number of corresponding forms on $G'$.\label{step3}
\end{enumerate}
In some sense, the non-vanishing $H$-periods must compensate for the vanishing ones by being unusually large. If $H$ is compact at infinity, these periods are discrete (\textit{i.e.} sums of point evaluations) and one would expect to obtain large point evaluations from large $H$-periods.

Step \ref{step2} is an important theme in Langlands program: the non-vanishing of certain periods is known to characterize the image of certain Langlands functorial lifting. In this text, we work with the theta correspondence. Specifically, we choose $G'$ such that $(G,G')$ form a dual reductive pair.

Let $n \geq m \geq 1$ be integers. Let $F$ be a totally real number field and let $E$ be either a totally imaginary quadratic extension of $F$ or $F$ itself. Let $V$ be a non-degenerate anisotropic Hermitian (or quadratic) $E$-space of dimension $n + m$ and let $W$ be a non-degenerate skew-Hermitian (or symplectic) $E$-space of dimension $2m$. We set $G = \U(V)$ and $G' = \U(W)$ for the associated $F$-groups.

We say that $V$ (resp. $W$) is \textit{unramified} at a finite place $v$ if the Hermitian space $V_v = V \otimes_F F_v$ (resp. $W_v = V \otimes_F F_v$), over the algebra $E \otimes_F F_v$, contains a self-dual lattice. Let $\mathcal{R}_0$ be the finite set of finite places of $F$ at which either $V$, $W$ or $E/F$ ramifies and let $\mathcal{R}$ denote the union of $\mathcal{R}_0$ with the places above $2$. 
Let $\mathcal{O}$ be the ring of integers of $F$ and let $L_V \subset V$ be a $\mathcal{O}$-lattice which is self-dual at each finite place $v \notin \mathcal{R}$, that is, $L_{V,v} = L_V \otimes \mathcal{O}_v$ is a self-dual lattice in $V_v$ for all $v \notin \mathcal{R}$ (with $ \mathcal{O}_v$ the ring of integers of $F_v$). Similarly, let $L_W \subset W$ be a $\mathcal{O}$-lattice which is self-dual at each finite place $v \notin \mathcal{R}$.

If $v \notin \mathcal{R}$ is a finite place, we define $K_v$ (resp. $K'_v$) as the stabilizer of $L_{V,v}$ (resp. $L_{W,v}$) while, if $v \in \mathcal{R}$, we fix an open compact subgroup $K_v$ (resp. $K'_v$\footnote[1]{Some extra care is required to define $K_v'$ (see Section \ref{secc8.2})}) of $G(F_v)$ (resp. $G'(F_v)$). Then, form the products
\begin{equation*}
K_f = \prod_{v \nmid \infty} K_v \quad \text{,} \quad K'_f = \prod_{v \nmid \infty} K'_v.
\end{equation*}
Given an ideal $\mathfrak{n} \subset \mathcal{O}$ prime to $\mathcal{R}$, we define the principal congruence subgroups ${K(\mathfrak{n}) \subset K_f}$ and $K'(\mathfrak{n}) \subset K_f'$ of level $\mathfrak{n}$ as the products
\begin{equation*}
K(\mathfrak{n}) = \prod_{v \mid \mathcal{R}} K_v \prod_{v \nmid \infty \mathcal{R}} K_v(\mathfrak{n}) \quad \text{,} \quad K'(\mathfrak{n}) = \prod_{v \mid \mathcal{R}} K'_v \prod_{v \nmid \infty \mathcal{R}} K'_v(\mathfrak{n})
\end{equation*}
where, for $v \nmid \infty \mathcal{R}$, we set
\begin{equation*}
K_v(\mathfrak{n}) = \lbrace g_v \in G(F_v) : (g_v - 1) L_{V,v} \subset \mathfrak{n} L_{V,v} \rbrace
\end{equation*}
and
\begin{equation*}
K'_v(\mathfrak{n}) = \lbrace g'_v \in G'(F_v) : (g'_v - 1) L_{W,v} \subset \mathfrak{n} L_{W,v} \rbrace.
\end{equation*}
Now, let $X \subset V$ be a $m$-dimensional Hermitian (or quadratic) sub-$E$-space and set ${H = \U(X) \times \U(X^{\perp})}$, a $F$-subgroup of $G$. We let $K_f^H = K_f \cap H(\mathbb{A}_f)$ and choose, for $\mathfrak{n} \subset \mathcal{O}$ prime to $\mathcal{R}$, a congruence subgroup $K_H(\mathfrak{n})$ of level $\mathfrak{n}$\footnote[2]{$K_H(\mathfrak{n})$ contains $K(\mathfrak{n})$ but contains no $K_H(\mathfrak{d})$ ($\mathfrak{d} \mid \mathfrak{n}$).} satisfying
\begin{equation*}
K_H(\mathfrak{n}) \cap K_f^H = K_f^H.
\end{equation*}
For example, one can consider the thickening $K_f^H \cdot K(\mathfrak{n})$ of the principal congruence subgroup as in \cite{BM1}.\\
We fix maximal compact subgroups ${K_v \subset G(F_v)}$ (resp. $K_v' \subset G'(F_v)$) at archimedean places and set
\begin{equation*}
K_{\infty} = \prod_{v \mid \infty} K_v \quad \text{,} \quad K'_{\infty} = \prod_{v \mid \infty} K'_v.
\end{equation*}
Finally, for $\mathfrak{n} \subset \mathcal{O}$ prime to $\mathcal{R}$, we define
\begin{equation*}
X_H(\mathfrak{n}) = G(F) \backslash G(\mathbb{A}) / K_{H}(\mathfrak{n}) K_{\infty} \quad \text{,} \quad X'(\mathfrak{n}) = G'(F) \backslash G'(\mathbb{A}) / K'(\mathfrak{n}) K'_{\infty}.
\end{equation*}
\begin{theorem} \label{theo1.2.1}
Fix integers $n > m \geq 1$ with $n + m \geq 4$ even and assume $W$ is split. Fix an archimedean place $v_0$ of $F$ and assume $V$ has signature $(n,m)$ at $v_0$, $V$ is positive definite at every other real places and $X$ is negative definite at $v_0$. Let $\nu \in i \mathfrak{a}^*$ be sufficiently regular with sufficiently large norm. For $\mathfrak{n} \subset \mathcal{O}$ an ideal prime to $\mathcal{R}$, there is a $L^2$-normalized (with respect to the probability measure) Maass form $f \in L^2(X_{H}(\mathfrak{n}))$ of spectral parameter $\nu + O(1)$ such that
\begin{equation*}
	\Vert f \Vert_{\infty} \gg \dfrac{\vol (X_H(\mathfrak{n}))^{\frac{1}{2}}}{\vol (X'(\mathfrak{n}))^{\frac{1}{2}}} \log ( \mathcal{N}(\mathfrak{n}) )^{-\frac{m [F : \mathbb{Q}]}{2}}
\end{equation*}
where $\mathcal{N}(\mathfrak{n}) = \vert \mathcal{O} / \mathfrak{n} \vert$ is the norm of $\mathfrak{n}$.
\end{theorem}
Before explaining the conditions on the congruence structure, we briefly comment on the Theorem.

The condition that $n + m$ is even can be relaxed, with minimal effort, when $E \neq F$. However, for $E = F$, one must replace $G' = \Sp_{2m}$ by its metaplectic (double) cover in order to ensure that $(G,G')$ form a dual reductive pair. In particular, the counting problem (in the sense of step \ref{step3}) requires additional care in this case.

We may compare the setting of Theorem \ref{theo1.2.1} to the general context provided by \cite{BM1}: applying \cite[Theorem 1.2]{BM1} we obtain, for $K_H(\mathfrak{n}) = K(\mathfrak{n}) K_f^H$, the existence of exceptional forms. This follows from the conditions on the signature of $V$ and $X$ at infinity: in particular, $G(F_{v_0})$ is non-split because $n \neq m$ and $H(F_{v_0}) = K_{v_0}$ (see \cite[§2.3]{BM2} for more details).

The logarithmic loss is due to the use of a partial trace formula when counting automorphic forms on $G'$, that is, when dealing step \ref{step3} (see Section \ref{chap7}). A precise analysis of the geometric side of the trace formula should eliminate this loss.

For Theorem \ref{theo1.2.1} to be non-trivial, certain additional conditions on the signature (as well as on the choice of $K_H(\mathfrak{n})$) must be satisfied. We perform explicit computations with ${K_H(\mathfrak{n}) = K(\mathfrak{n}) K_f^H}$, which is the smallest congruence subgroup of level $\mathfrak{n}$ satisfying the condition ${K_H(\mathfrak{n}) \cap K_f^H = K_f^H}$. By definition, we have
\begin{equation*}
	\vol (X_H(\mathfrak{n})) = \vol (X_H(1)) [K_f : K_H(\mathfrak{n}) ] \quad \text{,} \quad \vol (X'(\mathfrak{n})) =  \vol (X'(1)) [K'_f : K'(\mathfrak{n})].
\end{equation*}
We emphasize that, unlike the usual dichotomy in the level aspect, the lower bound in Theorem \ref{theo1.2.1} does not depend on the decomposition of $\mathfrak{n}$. However, to avoid technicalities due to the lack of strong approximation for orthogonal and unitary groups, we shall take $\mathfrak{n} = \mathfrak{p}^k$, where $\mathfrak{p}$ is a prime ideal, prime to $\mathcal{R}$, and $k \geq 1$ is an integer. We put aside the logarithmic loss, we let $\mathfrak{p} \subset \mathcal{O}$ be a prime ideal prime to $\mathcal{R}$, let $k \geq 1$ be an integer and set $E(\mathfrak{p}^k)$ for the exponent such that (up to a constant)
\begin{equation*}
\dfrac{\vol (X_{H}(\mathfrak{p}^k))^{\frac{1}{2}}}{\vol (X'(\mathfrak{p}^k))^{\frac{1}{2}}} = \vol (X_{H}(\mathfrak{p}^k))^{\frac{1}{2} -  E(\mathfrak{p}^k) + o(1)}.
\end{equation*}
Hence, $E(\mathfrak{p}^k)$ quantifies how far our lower bound is from the local bound \eqref{locbo}. In other words, Theorem \ref{theo1.2.1} is non-trivial if $E(\mathfrak{p}^k) < 1/2$. We have (see appendix \ref{appendixA})
\begin{equation*}
E(\mathfrak{p}^k) = \begin{cases}
      \frac{2m + 1}{2n}  \hspace*{5mm} \text{if $E = F$} \\
	\frac{m}{n}  \hspace*{10mm} \text{if $E \neq F$} \\
    \end{cases}\,.
\end{equation*}
We see that, in the case $E \neq F$, Theorem \ref{theo1.2.1} is non-trivial if $2m < n$ while, in the case $E = F$, the lower bound is non-trivial if $2m < n - 1$.\\
In view of Theorem \ref{theo1.2.2}, we refer to \cite[§2.4]{BM2} for the analogous discussion in the spectral aspect.

The forms produced in Theorem \ref{theo1.2.1} are theta lifts. The fact that these forms exhibit singular behaviors is not new. For instance, Howe and Piatetski-Shapiro \cite{HowePS} used them to prove the failure of the (naive) Ramanujan conjecture.
\subsubsection{The choice of $K_H(\mathfrak{n})$}\label{secccc1.2.1}
Recall there are two conditions on $K_H(\mathfrak{n})$
\begin{enumerate}
\item $K_H(\mathfrak{n})$ is a congruence subgroup of level $\mathfrak{n}$; \label{sstteepp1}
\item $K_H(\mathfrak{n}) \cap K_f^H = K_f^H$. \label{sstteepp2}
\end{enumerate}
The first condition determines the level of the lifted forms and, consequently, the level of the congruence manifold associated to $G'$. Roughly speaking, the theta lift of a level $K(\mathfrak{n})$ automorphic forms on $G$ is a level $K'(\mathfrak{n})$ automorphic form on $G'$. Thus, since $K_H(\mathfrak{n})$ contains $K(\mathfrak{n})$, the same applies to level $K_H(\mathfrak{n})$ automorphic forms on $G$, that is, their theta lift is of level $K'(\mathfrak{n})$. 

The second condition ensures that the $H$-period of a level $K_H(\mathfrak{n})$ automorphic form is independent of $\mathfrak{n}$. More precisely, let $[G] = G(F) \backslash G(\mathbb{A})$ and recall $V$ is anisotropic (hence $[G]$ is compact). By definition, the $H$-period of $f \in C^{\infty} ([G])$ is
\begin{equation*}
	\mathcal{P}_H(f) = \int_{[H]} f (h) dh
\end{equation*}
where $[H] = H(F) \backslash H(\mathbb{A})$ is endowed with the $H(\mathbb{A})$-invariant probability measure; since $[H]$ is compact, this integral converges absolutely.

Recalling the maximal compact subgroup $K_{\infty} \subset G(\mathbb{A}_{\infty})$ and the compactness of $H$ at infinity (\textit{cf}. the conditions on $V$ and $X$ in Theorem \ref{theo1.2.1}), we have
\begin{equation*}
K_{\infty} \cap H(\mathbb{A}_{\infty}) = H(\mathbb{A}_{\infty}).
\end{equation*}
Hence, if $K \subset G(\mathbb{A}_f)$ is an open compact subgroup and if $f \in C^{\infty}([G])$ is right-$KK_{\infty}$-invariant, we have
 \begin{equation*}
	\mathcal{P}_H(f) = \sum_{h \in \gen_H(K)} \vol \big( H(F) h (K \cap H(\mathbb{A}_f)) \big) f(h)
\end{equation*}
where $\gen_H(K)$ is the finite set $H(F) \backslash H(\mathbb{A}_f) / K \cap H(\mathbb{A}_f)$ (see \cite[Theorem 5.1]{BorelFiniteness}) and $\vol \big( H(F) h (K \cap H(\mathbb{A}_f)) \big)$ is the volume of $H(F) \cdot h \cdot K \cap H(\mathbb{A}_f)$ considered as a subset of $H(F) \backslash H(\mathbb{A}_f)$. Thus, the $H$-period of $f$ is a discrete period.\\
In particular, if we suppose $K$ satisfies condition \ref{sstteepp2}, then the $H$-period of $f$ is a discrete period whose number of points and weights are independent of $K$.

This way, if $K_{H}(\mathfrak{n})$ satisfies both conditions \ref{sstteepp1} and \ref{sstteepp2}, the $H$-period of a right-$K_H(\mathfrak{n})K_{\infty}$-invariant function $f \in C^{\infty}([G])$ is a finite weighted sum of point evaluations with both the number of points and the weights independent of $\mathfrak{n}$. This independence with respect to the level is crucial: it allows us to derive a lower bound on the sup-norm from a lower bound on the $H$-period, without any loss in the number of terms nor the weights.
\subsubsection{A hybrid result in the spectral-volume aspect} \label{secsecsecty}
We actually prove a stronger result than the one stated in Theorem \ref{theo1.2.1}, establishing a lower bound, with an additional multiplicity term, in both the volume and the spectral aspects simultaneously.

Indeed, following \cite{SarnakLetter}, if $\mu_H(\mathfrak{n}, \nu)$ denote the dimension of the space $V_H(\mathfrak{n}, \nu) \subset L^2(X_H (\mathfrak{n}))$  of ($L^2$-normalized with respect to the probability measure) Maas forms with spectral parameter $\nu$, then one can prove there is $f \in V_H(\mathfrak{n}, \nu)$ such that
\begin{equation} \label{eq232323}
	\Vert f \Vert_{\infty} \geq \mu_H(\mathfrak{n}, \nu)^{\frac{1}{2}}.
\end{equation}
While Theorem \ref{theo1.2.1} ensures the existence of Maass forms with large values compared to the average $\asymp 1$, one can ask for Maass forms with large values compared to the "trivial lower bound" \eqref{eq232323}.\\
We do not prove such statement, but rather a weaker version.

We let $\beta_S$ (resp. $\beta_{S'}$) denote the spectral density function (see Section \ref{sec3.5}) of ${S = G(F_{v_0}) / K_{v_0}}$ (resp. $S' = G'(F_{v_0}) / K'_{v_0}$). Let $\mathfrak{g} = \mathfrak{k} \oplus \mathfrak{p}$ be the Cartan decomposition of the Lie algebra of $G(F_{v_0})$ with respect to the Lie algebra of $K_{v_0}$ and let $\mathfrak{a} \subset \mathfrak{p}$ be a maximal abelian subspace. Note that the real rank of $G(F_{v_0})$ is the same as that of $G'(F_{v_0})$; hence, we can view a spectral parameter $\nu$ simultaneously for $S$ and $S'$ (see also \cite[§2.2]{BM2} for more details in the case $E = F$).
\begin{theorem} \label{theo1.2.2}
Fix integers $n > m \geq 1$ with $n + m \geq 4$ even and assume $W$ is split. Fix an archimedean place $v_0$ of $F$ and assume $V$ has signature $(n,m)$ at $v_0$, $V$ is positive definite at every other real places and $X$ is negative definite at $v_0$. For sufficiently regular $\nu \in i \mathfrak{a}^*$ with sufficiently large norm and for $\mathfrak{n} \subset \mathcal{O}$ an ideal prime to $\mathcal{R}$, there is a $L^2$-normalized (with respect to the probability measure) Maass form $f \in L^2(X_{H}(\mathfrak{n}))$ of spectral parameter $\nu + O(1)$ such that
\begin{equation*}
	\Vert f \Vert_{\infty} \gg \dfrac{\vol (X_H(\mathfrak{n}))^{\frac{1}{2}}}{\vol (X'(\mathfrak{n}))^{\frac{1}{2}}} \dfrac{\beta_S(\nu)^{\frac{1}{2}}}{\beta_{S'}(\nu)^{\frac{1}{2}}} \mu'(\mathfrak{n}, \nu)^{\frac{1}{2}} \mathcal{L} (\mathfrak{n}, \nu)^{-\frac{m [F : \mathbb{Q}]}{2}}
\end{equation*}
where $\mathcal{L} (\mathfrak{n}, \nu) = \log \big( \mathcal{N}(\mathfrak{n}) (1 + \Vert \nu \Vert) \big)$ and $\mu'(\mathfrak{n}, \nu)$ is the dimension of ${V'(\mathfrak{n}, \nu) \subset L^2(X'(\mathfrak{n}))}$, the space of Maass forms having spectral parameter $\nu$.
\end{theorem}
We do not expect $\mu'(\mathfrak{n}, \nu) \geq \mu_H(\mathfrak{n}, \nu)$. Indeed, the distinction method primarily relies on the observation that the number of Maass forms on $X_H(\mathfrak{n})$ significantly exceeds that on $X'(\mathfrak{n})$. However, we one can hope for improvements in a "Hecke-Maass" setting, see remark \ref{remarquebrumley}. 

The only known hybrid result on higher rank manifolds is due to Brumley and Marshall \cite{BM1}. As mentioned earlier, since our manifolds fall within their framework, the main interest in the following Theorem lies in the quality of the hybrid exponent.

Note that, for fixed level, this extends \cite[Theorem 2.1]{BM2} to the complex hyperbolic Grassmannian $\mathbb{H}^{n,m}_{\mathbb{C}}$.
\subsubsection{Global period relation and test function} \label{sec1.2.4}
In the global theta correspondence, a well known relation holds between Fourier-Whittaker periods of the lifted form and an orthogonal period of the form itself. This relation can be "extended" to connect Bessel periods with $H$-periods, by considering $X$ as a subspace of both $V$ and $W$. More precisely, if $f$ is an automorphic form of $G$, then
\begin{equation*}
\mathcal{P}^{\mathbf{1} \times \psi_q}_{R_q}(\theta (f ; s)) =  \int_{H (\mathbb{A}) \backslash G(\mathbb{A})}\mathcal{I}_s(x) \text{ } \mathcal{P}_{H}(R(x){f}) dx
\end{equation*}
where $s$ is a test function, $\theta (f ; s))$ is the theta lift of $f$, $R(x)f$ is defined as $g \mapsto f(gx)$ and $\mathcal{I}_s$ is a function on $H (\mathbb{A}) \backslash G(\mathbb{A})$ that is uniquely determined by $s$. The term $\mathcal{P}^{\mathbf{1} \times \psi_q}_{R_q}(\theta (f ; s))$ is the $(R_q, \mathbf{1} \times \psi_q)$-period (or Bessel period) of $\theta(f, s)$, which is defined as the integral
\begin{equation*}
\mathcal{P}^{\mathbf{1} \times \psi_q}_{R_q}(\theta (f ; s)) = \int_{[R_q]} \theta(f ; s) (r) (\mathbf{1} \times \psi_q) (r) dr
\end{equation*}
where $R_q \subset G'$ is the semi-direct product of the unipotent radical $N$ of the Siegel parabolic subgroup of $G'$ with $\U(X)$ (considered as a subgroup of $G'$) and $(\mathbf{1} \times \psi_q)$ is a character of $R_q$ defined with respect to $H$.

In \cite[Proposition 9.3]{BM2}, it is shown that for an open compact subgroup ${K \subset G(\mathbb{A}_f)}$ and a spectral parameter $\lambda \in \mathfrak{a}^*_{\mathbb{C}}$, if $f_{\lambda}$ is right-$KK_{\infty}$-invariant with spectral parameter $\lambda$, then
\begin{equation*}
\mathcal{P}^{\mathbf{1} \times \psi_q}_{R_q}(\theta (f ; s_{\lambda})) = C(s_{K_f, \lambda} ; \lambda) \sum_{x \in p_K} w(x) \mathcal{P}_{H}(R(x) f_{\lambda})
\end{equation*}
where $s_{K_f, \lambda}$ is a test function defined with respect to $K_f$ (not $K$) and $\lambda$, $C(s_{K_f, \lambda} ; \lambda) \neq 0$, $w(x)$ are non-zero weights and $p_K \subset G(\mathbb{A}_f)$ is a finite set of points containing $1$.

In contrast to \cite{BM2}, the present approach uses a test function $s_{K, C}$ defined with respect to $K$ and $C \subset \mathfrak{a}_{\mathbb{C}}^*$, a compact set (namely, the spectral window $\nu + O(1)$ appearing in Theorems \ref{theo1.2.1} and \ref{theo1.2.2}). Roughly speaking, this test function picks up only the contribution of $1 \in p_K$ while ensuring $C(s_{K, C} ; \lambda) \neq 0$ for all $\lambda \in C$.

This non-vanishing condition on the spectral window allows us for a simpler treatment of local multiplicities when counting the number of contributing forms (\textit{cf.} step \ref{step3}). In the spectral aspect, one estimates this number using Bernstein’s uniform admissibility Theorem (see \cite[Lemma 10.2]{BM2}) while, in the volume aspect, one should perform a "place by place" analysis to bound the dimension of the local theta lifts.

Bounds relating the dimension of the fixed vectors under the action of $K$ to the dimension of the lifted vector that are fixed under the action of $K'$ do exist when ${K = K(\mathfrak{n})}$ and $K' = K'(\mathfrak{n})$. More precisely, if $v$ is a finite place prime to $\mathcal{R}$ and if $\pi_v = \theta(\pi'_v)$ is an irreducible smooth representation of $G(F_v)$, then (see \textit{e.g.} \cite{Cossutta})
\begin{equation*}
	\dim \pi_v^{K_H(\mathfrak{n})} \leq \dim \pi_v^{K(\mathfrak{n})}  \ll \dim S_{\mathfrak{n}, v} \dim {\pi'_v}^{K'(\mathfrak{n})}
\end{equation*}
for some local space $S_{\mathfrak{n}, v} \subset S_{v}$ depending on $\mathfrak{n}$. Note that $S_{\mathfrak{n}} = \otimes_v S_{\mathfrak{n}, v}$ is contained in the space $S_{\mathcal{O}} = S = \otimes_v S_v$ of test functions $s$ used for the period relation and ${\dim (S_{\mathfrak{n}}) \ll \mathcal{N}(\mathfrak{n})^{\frac{\dim_E V \dim_E W}{2}}}$. However, these local dimensions (or multiplicities) contaminate the final lower bound, yielding the weaker bound
\begin{equation*}
	\Vert f \Vert_{\infty} \gg \dfrac{\vol (X_H(\mathfrak{n}))^{\frac{1}{2}}}{\vol (X'(\mathfrak{n}))^{\frac{1}{2}}} \dfrac{\beta_S(\nu)^{\frac{1}{2}}}{\beta_{S'}(\nu)^{\frac{1}{2}}}  \dim (S_{\mathfrak{n}})^{- \frac{1}{2}} \mathcal{L} (\mathfrak{n}, \nu)^{-\frac{m [F : \mathbb{Q}]}{2}}.
\end{equation*}
We bypass this local problem by defining the space of global lifted forms as the image of the map $\theta( \cdot \text{ } ; s_{K,C})$ in contrast to \cite{BM2} where, roughly speaking, the authors considered the space of global lifted forms as the image of "all" the various maps $\theta( \cdot \text{ } ; s_{\lambda})$ for $\lambda \in C$.
\begin{rmrk} \label{remarquebrumley}
In Section \ref{secc8.2}, we apply this period relation to derive the distinction argument outlined in step \ref{step2}. However, we believe that our method can be improved to give stronger bounds, as there is a loss of information during the current execution process. Indeed, in Section \ref{secc8.2},  we use condition \ref{sstteepp1} on $K_H(\mathfrak{n})$ but not condition \ref{sstteepp2}. (Nevertheless, we emphasize that condition \ref{sstteepp2} is used in the deduction of the main result, as explained in Section \ref{secccc1.2.1}). We believe that further analysis of the transfer of invariants in the local theta correspondence, taking condition \ref{sstteepp2} into account, could yield additional (equivariant) constraints on the lifted forms. In particular, this would reduce the number of corresponding forms on $G'$ (in the sense of step \ref{step3}) and thus strengthen the lower bound stated in Theorem \ref{theo1.2.2}.

Moreover, upon considering the "Hecke-Mass" multiplicity $\mu^{\mathcal{HM}}_H(\mathfrak{n}, \nu)$ (in contrast to the "Maass multiplicity" $\mu_H(\mathfrak{n}, \nu)$, see Section \ref{secsecsecty}), we think one could to beat \eqref{eq232323} by using surjectivity results for non-tempered representations, see \textit{e.g.} \cite[Theorem 4.2]{BergeronMillsonMoeglin}.

We hope to address these two (connected) questions in future work together with F. Brumley, J. Hou, S. Marshall and R. Takloo-Bighash.
\end{rmrk}
\subsubsection{Outline} \label{sec1.2.3}
The similarity between this thesis and \cite{BM2} should be evident to the reader. In particular, the overall structure of this text parallels that of the latter.

We begin in Section \ref{chap3} with purely Archimedean considerations. We discuss $\tau$-spherical harmonic analysis as in the theta correspondence, spherical forms on $G$ correspond to $\tau$-spherical forms on $G'$. In particular, the counting problem outlined in step \ref{step3} involves $\tau$-spherical forms, thus necessitating a $\tau$-spherical test function. This is based on Shimeno's work \cite{Shimeno50} on the Plancherel formula for $\tau$-spherical functions on simple Lie groups, with $\tau$ a one dimensional $K$-type.

In Section \ref{chap2}, we examine both local and global aspects of the theta correspondence and focus on certain properties of orthogonal-symplectic and unitary pairs. In particular, we establish the correspondence mentioned above and explicit $\tau$.

In Section \ref{chap5}, we further investigate the period relation and the test function sketched in Section \ref{sec1.2.4}. The key outcome of this Section is the distinction principle outlined in step \ref{step2}.

In Section \ref{chap6} we tackle step \ref{step1}, that is, we estimate the average size of $H$-periods of level $K_H(\mathfrak{n})$ spherical automorphic forms on $G$. As explained in Section \ref{secccc1.2.1}, the assumptions on $H$ and $K_H(\mathfrak{n})$ reduce this estimation to that of a (finite) weighted sum of point evaluations, where both the weights and the number of terms are independent of $\mathfrak{n}$. Our treatment is classical (\textit{i.e.} non adelic) and follows from the pre-trace formula.

In Section \ref{chap7}, we address the counting problem outlined in step \ref{step3}. More precisely, we count the number of ($\tau$-spherical) cuspidal automorphic forms of level $K'(\mathfrak{n})$ and spectral parameter $\nu + O(1)$ on $G'$. We show that such cusp forms are concentrated below height $(\Vert \nu \Vert \vol (X'(\mathfrak{n})))^{1 +\epsilon}$ in the cusp and then bound the integral of the geometric kernel function over this region. As mentioned in the paragraph following Theorem \ref{theo1.2.1}, this coarse approach is responsible for the logarithmic loss in Theorem \ref{theo1.2.2} (and \ref{theo1.2.1}).

In Section \ref{chap8}, we combine these ingredients to derive Theorem \ref{theo1.2.2} (and \ref{theo1.2.1}). The constancy of both the weights and the number of terms in $\mathcal{P}_H(f)$ with respect to $\mathfrak{n}$ and $\nu$ is used again to deduce the lower bound on $\Vert f \Vert_{\infty}$ from a lower bound on $\vert \mathcal{P}_H(f) \vert$.

\section{$\tau$-spherical harmonic analysis} \label{chap3}
In this Section, we introduce $\tau$-spherical objects such as $\tau$-spherical functions, representations and the $\tau$-spherical inversion formula. Essentially we refine \cite[§3-§5]{BM2} to include the (indefinite special) unitary groups. The main results are Lemma \ref{lem3.3.1}, giving a parametrization of irreducible $\tau$-spherical representations by their spectral parameter, and Proposition \ref{prop3.7.1}, providing us with a $\tau$-spherical test function we shall use in Sections \ref{chap6} and \ref{chap7}.
\subsection{Notations} \label{sec3.1}
If $V$ is a finite dimensional vector space over $\mathbb{R}$, we let $V_{\mathbb{C}} = V \otimes \mathbb{C}$ and denote by $\Symm(V_{\mathbb{C}})$ the symmetric algebra of $V_{\mathbb{C}}$ which may be identified with $\mathcal{P}(V_{\mathbb{C}}^*)$, the polynomial algebra of $V_{\mathbb{C}}^*$.

Let $G$ be real connected reductive Lie group \textit{i.e.} a closed connected group of real matrices that is stable under conjugate transpose. Let $\Theta$ be the Cartan involution of $G$ given by inverse conjugate transpose and let $K = \lbrace g \in G : \Theta(g) = g \rbrace$. The differential $\theta = d \Theta$, given by negative conjugate transpose, is an automorphism of the Lie algebra $\mathfrak{g}$ of $G$. As $\theta$ is also an involution, if $\mathfrak{k}$ (resp. $\mathfrak{p}$) denote the $1$-eigenspace (resp. $-1$-eigenspace) for $\theta$, we obtain the Cartan decomposition $\mathfrak{g} = \mathfrak{k} \oplus \mathfrak{p}$. Let $B_0(X,Y) = \tr(XY)$ be the trace form and note that $\langle X , Y \rangle {=} - B_0(X, \theta(Y))$ is an inner product on $\mathfrak{g}$ with respect to whom there is a direct sum decomposition $\mathfrak{p} = \mathfrak{m} \oplus \mathfrak{a}  \oplus \mathfrak{n}$. In particular, $\mathfrak{a} \subset \mathfrak{p}$ is a maximal abelian subspace and $\mathfrak{m}$ is the centralizer $\lbrace K \in \mathfrak{k} : \ad(K) A = 0 \text{, } \forall A \in \mathfrak{a} \rbrace$ of $\mathfrak{a}$ in $\mathfrak{k}$. Let $\mathfrak{h}_{\mathfrak{m}} \subset \mathfrak{m}$ be a Cartan subalgebra and note that $\mathfrak{h} = \mathfrak{a} \oplus \mathfrak{h}_{\mathfrak{m}} \subset \mathfrak{g}$ is a Cartan subalgebra.

Let $B(X,Y) = \tr(\ad(X) \circ \ad(Y))$ be the Killing form on $\mathfrak{g}$, an invariant symmetric bilinear form on $\mathfrak{g}$ that is non-degenerated when restricted to $\mathfrak{h}$. For $\alpha \in \mathfrak{h}^*$, we define $H_{\alpha} \in H$ as the unique element satisfying $B(H,H_{\alpha}) = \alpha(H)$ for all $H \in \mathfrak{h}$. The restriction of the Killing form on $\mathfrak{h}$ induces a non-degenerate symmetric bilinear form ${(\alpha_1, \alpha_2) = B(H_{\alpha_1}, H_{\alpha_2})}$ on $\mathfrak{h}^*$. If $\alpha \in \Delta(\mathfrak{h})$, we define the Weyl reflection $s_{\alpha}$ (of $\mathfrak{h}$) associated to $\alpha$ by $s_{\alpha}(H) = H - 2\alpha(H) H_{\alpha} / B(\alpha, \alpha)$. The Weyl group of $\mathfrak{h}$ in $\mathfrak{g}$, denoted by $W(\mathfrak{h})$, is the group generated by the Weyl reflections associated to the $\alpha$'s in $\Delta(\mathfrak{h})$. We define the Weyl group of $\mathfrak{a}$ in $\mathfrak{g}$ as $W(\mathfrak{a}) = \lbrace \Ad(k)_{\vert \mathfrak{a}} : k \in K \text{ s.t.} \Ad(k) \mathfrak{a} = \mathfrak{a} \rbrace$ ; this is the quotient of the normalizer of $\mathfrak{a}$ in $K$ by the centralizer of $\mathfrak{a}$ in $K$. We denote by $W(\mathfrak{h}_{\mathbb{C}})$ the Weyl group of $\mathfrak{h}_{\mathbb{C}}$ in $\mathfrak{g}_{\mathbb{C}}$; the group generated by the Weyl reflections of $\mathfrak{h}_{\mathbb{C}}$ associated to the $\alpha$'s in $\Delta(\mathfrak{h}_{\mathbb{C}})$, the root system of $\mathfrak{h}_{\mathbb{C}}$ in $\mathfrak{g}_{\mathbb{C}}$. We naturally extend $s \in W(\mathfrak{a})$ to an automorphism of $\mathfrak{a}_{\mathbb{C}}$ by setting $s_{\vert i \mathfrak{a}} = i s$.

We denote the set of roots of $\mathfrak{a}$ in $\mathfrak{g}$ by $\Sigma(\mathfrak{a})$, the set of roots of $\mathfrak{h}_{\mathfrak{m}}$ in $\mathfrak{m}$ by $\Delta(\mathfrak{h}_{\mathfrak{m}})$ and the set of roots of $\mathfrak{h}$ in $\mathfrak{g}$ by $\Delta(\mathfrak{h})$. Similarly, we define $\Sigma(\mathfrak{a}_{\mathbb{C}})$, $\Delta(\mathfrak{h}_{\mathfrak{m}_{\mathbb{C}}})$ and $\Delta(\mathfrak{h}_{\mathbb{C}})$. We choose positive roots $\Delta^+(\mathfrak{h}) \subset \Delta(\mathfrak{h})$ such that
\begin{equation*}
\Sigma^+(\mathfrak{a}) \overset{\defin}{=} \Delta^+(\mathfrak{h}) \cap \Sigma(\mathfrak{a}) \subset \Sigma(\mathfrak{a})
\end{equation*}
are the positive roots defined by $\mathfrak{n}$. Similarly, we define $\Delta^+(\mathfrak{h}_{\mathfrak{m}_{\mathbb{C}}})$ and $\Delta^+(\mathfrak{h}_{\mathbb{C}})$. Let $\rho$ denote the half-sum of $\Delta^+(\mathfrak{h}_{\mathbb{C}})$, then $\rho = \rho_{\mathfrak{a}} + \rho_{\mathfrak{m}}$ where $\rho_{\mathfrak{a}}$ (resp. $\rho_{\mathfrak{m}}$) is the half-sum of $\Sigma^+(\mathfrak{a})$ (resp. $\Delta^+(\mathfrak{h}_{\mathfrak{m}_{\mathbb{C}}})$).

Let $M_0$, $A$ and $N$ denote the analytic subgroups of $G$ associated to $\mathfrak{m}$, $\mathfrak{a}$ and $\mathfrak{n}$ respectively. Let $M = Z_K(\mathfrak{a})M_0$, where $Z_K(\mathfrak{a}) = \lbrace k \in K : \Ad(k) A = A \text{, } \forall A \in \mathfrak{a} \rbrace$, then $MAN$ is the Langlands decomposition of a minimal parabolic subgroup $P$ of $G$ containing $A$. Let $H : G \rightarrow \mathfrak{a}$, $\kappa : G \rightarrow K$ denote the Iwasawa projections so that $g \in G$ can be written as
\begin{equation*}
	g = m \exp(H(a)) n \kappa(k)
\end{equation*}
for some $m \in M$ and $n \in N$.
The Killing form induces an inner product on $\mathfrak{a}$ and we define the ball $B_{\mathfrak{a}} (0,R) \subset \mathfrak{a}$ centered at the origin and of radius $R > 0$ with respect to the norm induced by this inner product. We then define $A_R = \exp \big( B_{\mathfrak{a}} (0,R) \big)$ and set $G_R = K A_R K \subset KAK$, with $KAK$ the Cartan decomposition of $G$.
\subsection{$\tau$-spherical function} \label{sec3.2}
Let $\tau : K \rightarrow C^{\times}$ be a continuous character. Let $L(\tau)$ denote the homogeneous line bundle on $G/K$ associated to $\chi$ that is, elements in $L(\tau) = G \times_K \mathbb{C}$ are equivalence classes of $(x, \lambda) \in G \times \mathbb{C}$ for the relation $(g,\lambda) \sim (gk, \chi(k^{-1}) \lambda)$; with the class of $(g, \lambda)$ projecting onto $gK \in G/K$. The space $H^0(G/K,L(\tau))$ of smooth sections ${s : G/K \rightarrow L({\tau})}$ identifies with $C^{\infty}(G/K, \tau)$, the space of smooth functions ${f : G \rightarrow \mathbb{C}}$ such that
\begin{equation*}
	f(gk) = \tau(k)^{-1} f(g) \quad (k \in K, g \in G).
\end{equation*}
The action of $x \in G$ by left-translation on the class of $(g, \lambda)$ is well-defined and yields the class $(x \cdot g, \lambda)$. This induces an action of $G$ on smooth functions $f \in C^{\infty}(G/K ; \tau)$; namely $g \cdot f : x \mapsto f(gx)$. Let $\mathcal{D}(G/K ; \tau)$ denote the left-$G$-invariant differential operators on $C^{\infty}(G/K ; \tau)$ \textit{i.e.} linear differential operators ${D : C^{\infty}(G/K ; \tau) \rightarrow C^{\infty}(G/K ; \tau)}$ such that $g \cdot Df = D (g \cdot f)$ for all $g \in G$ and all $f \in C^{\infty}(G/K ; \tau)$. Let $\mathcal{U}(\mathfrak{g}_{\mathbb{C}})$ denote the universal enveloping algebra of $\mathfrak{g}_{\mathbb{C}}$, let $K$ act on $\mathcal{U}(\mathfrak{g}_{\mathbb{C}})$ by $k \cdot X = \Ad(k) X$ for $X \in \mathfrak{g}$ and denote by $\mathcal{U}(\mathfrak{g}_{\mathbb{C}})^K$ the $K$-invariant elements in $\mathcal{U}(\mathfrak{g}_{\mathbb{C}})$. We have natural map $\mathcal{U}(\mathfrak{g}_{\mathbb{C}})^K \rightarrow \mathcal{D}(G/K ; \tau)$; namely $X \in \mathcal{U}(\mathfrak{g}_{\mathbb{C}})^K$ defines $D_X \in \mathcal{D}(G/K ; \tau)$ given on $f \in C^{\infty}(G)$ by
\begin{equation*}
	D_X f (g) = \frac{d}{dt} f(g \exp(tX))  \bigg|_{t = 0} \quad (g \in G).
\end{equation*}
In \cite[§2]{Shimeno51} Shimeno shows that $X \mapsto D_X$ induces an isomorphism
\begin{equation*}
\Symm(\mathfrak{a}_{\mathbb{C}})^{W(\mathfrak{a})} \simeq \mathcal{D}(G/K ; \tau)
\end{equation*}
where $\Symm(\mathfrak{a}_{\mathbb{C}})^{W(\mathfrak{a})}$ are the $W(\mathfrak{a})$-invariants of $\Symm(\mathfrak{a}_{\mathbb{C}})$, with ${w \in W(\mathfrak{a})}$ acting on $P \in \Symm(\mathfrak{a}_{\mathbb{C}})$ by $w \cdot P(A) = P(w^{-1} \cdot A)$. Let $K$ act on $\Symm(\mathfrak{p}_{\mathbb{C}})$ by ${k \cdot P(X) = P(\Ad(k^{-1}) X)}$ and let $\Symm(\mathfrak{p}_{\mathbb{C}})^K$ denote the $K$-invariants elements in $\Symm(\mathfrak{p}_{\mathbb{C}})$. Similarily, in \cite[§4]{Rouviere}, Rouvière shows that $\mathcal{D}(G/K ; \tau)$ is isomorphic to $\Symm(\mathfrak{p}_{\mathbb{C}})^K$ which, in turn, is isomorphic to $\Symm(\mathfrak{a}_{\mathbb{C}})^{W(\mathfrak{a})}$ by Chevalley's restriction Theorem (see also \cite[§2]{Shimura}). This way, we obtain the \textit{$\tau$-spherical Harish-Chandra isomorphism}
\begin{equation*}
	\gamma_{\tau} : \mathcal{D}(G/K ; \tau) \longrightarrow \Symm(\mathfrak{a}_{\mathbb{C}})^{W(\mathfrak{a})}.
\end{equation*}
Let $\lambda \in \mathfrak{a}^*_{\mathbb{C}}$ and extend $\lambda$ to $\mathcal{U}(\mathfrak{a}_{\mathbb{C}}) \rightarrow \mathbb{C}$ using the universal property of $\mathcal{U}$; note that $\mathcal{U}(\mathfrak{a}_{\mathbb{C}}) = \Symm(\mathfrak{a}_{\mathbb{C}})$ since $\mathfrak{a}$ is abelian. Then, composing $\gamma_{\tau}$ with $\lambda$, we obtain a morphism
\begin{equation*}
	\gamma_{\lambda} : \mathcal{D}(G/K ; \tau) \longrightarrow \mathbb{C}
\end{equation*}
satisfying $\gamma_{\lambda} = \gamma_{\mu}$ if and only if $\lambda = w \cdot \mu$ for some $w \in W(\mathfrak{a})$. Let $f \in C^{\infty}(G/K ; \tau)$ be a \textit{joint eigenfunction} of $\mathcal{D}(G/K ; \tau)$ \textit{i.e.} there is a morphism ${\chi_f : \mathcal{D}(G/K ; \tau) \rightarrow \mathbb{C}}$ such that $D \cdot f = \chi_f (D) f$ for all $D \in \mathcal{D}(G/K ; \tau)$. Then, there is $\lambda \in \mathfrak{a}^*_{\mathbb{C}}$ such that $\chi_f = \gamma_{\lambda}$ and we call $\lambda$ the (up to the Weyl's group) \textit{spectral parameter} of $f$. Given $\lambda \in \mathfrak{a}^*_{\mathbb{C}} / W(\mathfrak{a})$, we denote by $\mathcal{E}_{\lambda}(G/K ; \tau)$ the space of joint $\mathcal{D}(G/K ; \tau)$-eigenfunctions with spectral parameter $\lambda$.

Let $C^{\infty}(G/K ; \tau, \tau)$ denote the subspace of functions $f \in C^{\infty}(G/K ; \tau)$ such that
\begin{equation*}
	f(kg) = \tau(k)^{-1} f(g) \quad (k \in K, g \in G).
\end{equation*}
One can show that this is a commutative algebra (with convolution product) and, given $\lambda \in \mathfrak{a}_{\mathbb{C}}^*$, there is a unique $\varphi_{\lambda, \tau} \in C^{\infty}(G/K ; \tau, \tau) \cap \mathcal{E}_{\lambda}(G/K ; \tau)$ such that $\varphi_{\lambda, \tau}(1) = 1$ (see \textit{e.g.} \cite[§1.3]{GangolliVara}). We call $\varphi_{\lambda, \tau}$ the \textit{$\tau$-spherical function} of spectral parameter $\lambda$. Note that the line spanned by $\varphi_{\lambda, \tau}$ is the $\tau$-isotypic subspace of $\mathcal{E}_{\lambda}(G/K ; \tau)$ and $\varphi_{\lambda, \tau} = \varphi_{\mu, \tau}$ if and only if $\lambda = w \cdot \mu$ for some $w \in W(\mathfrak{a})$. More explicitly, we have
\begin{equation*}
	\varphi_{\lambda, \tau}(g) = \int_K \tau^{-1}(k) \exp( (\lambda - \rho_{\mathfrak{a}}) H(kg)) \tau(\kappa(kg)) dk \quad (g \in G).
\end{equation*}
We define the unitary spectrum $\mathfrak{a}^*_{\tau, \text{unit}}$ of $\mathcal{D}(G/K ; \tau)$ as
\begin{equation*}
\mathfrak{a}^*_{\tau, \text{unit}} = \lbrace \lambda \in \mathfrak{a}^*_{\mathbb{C}} : \varphi_{\lambda, \tau} \text{ is positive definite } \rbrace.
\end{equation*}
\subsection{$\tau$-spherical representation} \label{sec3.3}
In this Section, we shall technically work with admissible $(\mathfrak{g}, K)$-modules but refer to them as representations. By an \textit{admissible representation} of $G$ we mean the following data: a complex vector space $V$ with a continuous action ${G \times V \rightarrow V}$ such that ${G \rightarrow \GL(V)}$ is a group morphism, $K$ operates by unitary operators and ${\dim \Hom_K (\tau, \pi) < \infty}$ for all $\tau \in \widehat{K}$, the equivalence classes of irreducible representations of $K$. Let $\tau : K \rightarrow \mathbb{C}^{\times}$ be a continuous character of $K$, we say that $\pi$ is \textit{$\tau$-spherical} if $\Hom_K(\tau, \pi) \neq 0$. Given $\pi$ a $\tau$-spherical representation of $G$ and $\phi$ a non-zero element in $\Hom_K(\tau, \pi)$, $\mathcal{D}(G/K ; \tau)$ acts on $\phi(1)$ by scalars, yielding a morphism $\chi_{\pi} : \mathcal{D}(G/K ; \tau) \rightarrow \mathbb{C}$. The latter can be written as $\chi_{\pi} = \chi_{\lambda}$ for some unique $\lambda \in \mathfrak{a}^*_{\mathbb{C}} / W(\mathfrak{a})$ called the \textit{spectral parameter} of $\pi$.

Let $(\pi, V)$ be an admissible representation of $G$ and assume that the center ${\mathfrak{Z} = \mathfrak{Z}(\mathcal{U}(\mathfrak{g}_{\mathbb{C}}))}$ of the universal enveloping algebra of $\mathfrak{g}_{\mathbb{C}}$ acts on $V$ by scalars (\textit{e.g.} if $\pi$ is irreducible), yielding a character $\chi_{\pi} : \mathfrak{Z} \rightarrow \mathbb{C}$ called the \textit{infinitesimal character} of $\pi$. Let ${\gamma : \mathfrak{Z} \overset{\simeq}{\longrightarrow} \Symm(\mathfrak{h}_{\mathbb{C}})^{W(\mathfrak{h}_{\mathbb{C}})}}$ be the Harish-Chandra isomorphism \cite[Theorem 3.2.3]{Wallach} and define, for $\omega \in \mathfrak{h}_{\mathbb{C}}^*$, the morphism
\begin{equation*}
	\chi_{\omega} : \mathfrak{Z} \overset{\gamma}{\longrightarrow} \Symm(\mathfrak{h}_{\mathbb{C}}^*)^{W(\mathfrak{h}_{\mathbb{C}})} \overset{\omega}{\longrightarrow} \mathbb{C}
\end{equation*}
where $\omega : \Symm(\mathfrak{h}_{\mathbb{C}}^*) \rightarrow \mathbb{C}$ is defined by extending $\omega : \mathfrak{h}_{\mathbb{C}} \rightarrow \mathbb{C}$ using the universal property of $\mathcal{U}$; note that $\mathcal{U}(\mathfrak{h}_{\mathbb{C}}^*) = \Symm(\mathfrak{h}_{\mathbb{C}}^*)$ since $\mathfrak{h}_{\mathbb{C}}$ is abelian. In particular, there is a unique $\omega \in \mathfrak{h}_{\mathbb{C}}^* / W(\mathfrak{h}_{\mathbb{C}})$ such that $\chi_{\pi} = \chi_{\omega}$ and we call $\omega$ the \textit{Harish-Chandra parameter} of $\pi$.

Recall the Langlands decomposition $MAN$ of a minimal parabolic subgroup $P$ of $G$ containing $A$, the analytic subgroup of $G$ whose Lie algebra is $\mathfrak{a}$. Let $\sigma$ be an irreducible unitary representation of $M$, let $\lambda \in \mathfrak{a}^*_{\mathbb{C}}$ and let $\exp(\lambda) : A \rightarrow \mathbb{C}$ be the continuous character given by $a \mapsto \exp (\lambda(H(a))$. We denote by $\Pi_{\lambda, \sigma}$ the induced representation of $\sigma^{-1} \otimes \exp(- \lambda) \otimes \mathbf{1}$ of $G$, where $\mathbf{1}$ denote the trivial character of $N$. This representation is admissible \cite[Proposition 8.4]{Knapp} and a dense subspace of $\Pi_{\lambda, \sigma}$ is the space of smooth functions $f : G \rightarrow \mathbb{C}$ such that
\begin{equation*}
	f(mang) = \exp \big( (\nu - \rho_{\mathfrak{a}})(H(a) \big) \sigma (m) f(g) \quad (m \in M, a \in A, N \in N, g \in G).
\end{equation*}
We now specialize the discussion to the case $\sigma = \tau_{\vert M}$ and set $\Pi_{\lambda, \tau} = \Pi_{\lambda, \tau_{\vert M}}$. This representation is $\tau$-spherical; indeed, if $V_{\lambda, \tau} \subset \Pi_{\lambda, \tau}$ is the line spanned by the function
\begin{equation*}
	f_{\lambda, \tau} (g) =  \exp \big( (\lambda - \rho_{a}) H(g) \big) \tau(\kappa(g)) \quad (g \in G)
\end{equation*}
then, by Frobenius reciprocity, we have
\begin{equation*}
	\lbrace 0 \rbrace \neq \Hom_G (V_{\lambda, \tau}, \Pi_{\lambda, \tau}) = \Hom_K (\tau,  \Pi_{\lambda, \tau}).
\end{equation*}
Note that $f_{\lambda, \tau}$ is the unique function in $\Pi_{\lambda, \tau}$ extending the function $\tau$ on $K$. Recall the space of joint $\mathcal{D}(G/K ; \tau^{-1})$-eigenfunctions with spectral parameter $\lambda$ denoted by $\mathcal{E}_{\lambda}(G/K ; \tau^{-1})$ and defined in Section \ref{sec3.2}. We endow $\mathcal{E}_{\lambda}(G/K ; \tau^{-1})$ with its natural Fréchet space topology and this defines a smooth representation of $G$. Let ${\Pi_{\lambda, \tau}^{\infty} \subset \Pi_{\lambda, \tau}}$ denote the smooth vectors in $\Pi_{\lambda, \tau}$ and note that $f_{\lambda, \tau} \in \Pi_{\lambda, \tau}^{\infty}$. A $\tau^{-1}$-averaging of $f_{\lambda, \tau}$ over $K$ defines a matrix coefficient \textit{i.e.} we have
\begin{align*}
	\langle f_{\lambda, \tau^{-1}}, \Pi_{\lambda, \tau} (g) f_{\lambda, \tau} \rangle
	&=  \int_{K} \tau^{-1}(k) \exp \big( (\lambda - \rho_{\mathfrak{a}}) (H(kg)) +  (\lambda - \rho_{\mathfrak{a}}) (H(k)) \big) \tau( \kappa (kg)) dk \\
	&=  \int_{K} \tau^{-1}(k) \exp \big( (\lambda - \rho_{\mathfrak{a}}) (H(kg)) \big) \tau( \kappa (kg)) dk = \varphi_{\lambda, \tau^{-1}}(g) \\
	&= \int_{K} \tau^{-1}(k) f_{\lambda, \tau}(kg) dk
\end{align*}
and this recipe yields an intertwining map $\Pi_{\lambda, \tau}^{\infty} \rightarrow \mathcal{E}_{\lambda}(G/K ; \tau^{-1})$. Let $V(\lambda, \tau^{-1})$ denote the smallest closed invariant subspace of $\mathcal{E}_{\lambda}(G/K ; \tau^{-1})$ containing $\varphi_{\lambda, \tau^{-1}}$ and endow it with the relative topology; then $V(\lambda, \tau^{-1})$ is irreducible.

We now discuss the Harish-Chandra parameter of $\Pi_{\lambda, \sigma}$. If $\sigma$ is an irreducible unitary representation of $M$ and if $\omega_{\sigma} \in \mathfrak{h}^*_{\mathfrak{m}, \mathbb{C}}$ is the Harish-Chandra parameter of $\sigma^{-1}$, then the Harish-Chandra parameter of $\Pi_{\lambda, \sigma}$ is $\omega_{\sigma} + \lambda \in \mathfrak{h}^*_{\mathfrak{m}, \mathbb{C}} \oplus \mathfrak{a}^*_{\mathbb{C}} = \mathfrak{h}_{\mathbb{C}}^*$ \cite[Proposition 8.22]{Knapp}. On the other hand, the infinitesimal character of an irreducible representation $\mathfrak{r}$ of $K$ is $\lambda_{\mathfrak{r}} + \rho_{\mathfrak{m}}$, where $\lambda_{\mathfrak{r}}$ is the highest weight of $\mathfrak{r}$. Let $\tau : K \rightarrow \mathbb{C}^{\times}$ is a continuous character of $K$ and let $\lambda \in \mathfrak{a}_{\mathbb{C}}^* / W(\mathfrak{a})$. Recall $\rho_{\mathfrak{m}} \in \mathfrak{h}_{\mathfrak{m}_{\mathbb{C}}}^*$ and $\mathfrak{h}_{\mathbb{C}}^* =  \mathfrak{h}_{\mathfrak{m}_{\mathbb{C}}}^* \oplus \mathfrak{a}^*_{\mathbb{C}}$. The following Lemma ensures, in particular, that the Harish-Chandra parameter of $\Pi_{\lambda, \tau}$ is $\rho_{\mathfrak{m}} + \lambda$.
\begin{lemma} \label{lem3.3.1}
There is a unique irreducible $\tau$-spherical representation $\pi_{\lambda, \tau}$ of $G$ with spectral parameter $\lambda$. Moreover, the Harish-Chandra parameter of $\pi_{\lambda, \tau}$ is $\rho_{\mathfrak{m}} + \lambda$.
\end{lemma}
When $G$ is split, then $\mathfrak{a}$ is a Cartan subalgebra of $\mathfrak{g}$ so that $\rho_{\mathfrak{m}} = 0$.
\begin{corollary} \label{coro3.3.2}
Assume $G$ is split. There is a unique irreducible $\tau$-spherical representation $\pi_{\lambda, \tau}$ of $G$ with spectral parameter $\lambda$. Moreover, the Harish-Chandra parameter of $\pi_{\lambda}$ is $\lambda$.
\end{corollary}
\begin{proof}[Proof of Lemma \ref{lem3.3.1}]
For the existence, we construct $\pi_{\lambda, \tau}$ as the irreducible sub-quotient of $\Pi_{\lambda, \tau_{\vert M}}$ containing $\tau$ as a $K$-type \textit{i.e.} having non-zero $\Hom_K (\tau, \pi)$. For unicity, we argue as in \cite[Proposition 7.2]{BM2} : let $\pi_1, \pi_2$ be two irreducible $\tau$-spherical representations with the same spectral parameter and let $v \in \pi_1(\tau) \oplus \pi_2(\tau)$ be a non-zero vector that does not lie in either $\pi_1(\tau)$ or $\pi_2(\tau)$; where $\pi_i (\tau) = \lbrace v \in \pi_i : \pi_i(k)v = \tau(k) v \text{, } \forall k \in K \rbrace$ is the $\tau$-isotypic subspace of $\pi_i$ for $i = 1, 2$. Let $V$ be the $(\mathfrak{g}, K)$-module generated by $v$ and let $w \in V(\tau)$. Then, there are $k_i \in K$ and $X_i \in \mathcal{U}(\mathfrak{g})$ such that
\begin{equation*}
	w = \sum_i X_i k_i \cdot v = \left( \sum_i \tau(k_i) X_i  \right) v = Y \cdot v
\end{equation*}
for some $Y \in \mathcal{D}(G/K ; \tau)$. As both $\pi_1(\tau)$ and $\pi_2(\tau)$ share same spectral parameter, $\mathcal{D}(G/K ; \tau)$ acts by scalars on $\pi_1(\tau) \oplus \pi_2(\tau)$, ensuring $w \in \mathbb{C} \cdot v$ and thus $V(\tau) = \mathbb{C} \cdot v$. This way, we obtained a non-zero sub-module $V$ of $\pi_1 \oplus \pi_2$ which is equal to neither $\pi_1, \pi_2$ nor $\pi_1 \oplus \pi_2$; $V$ therefore induces an isomorphism $\pi_1 \simeq \pi_2$.\\
We now prove the statement about the Harish-Chandra parameter. We argue as in the proof of \cite[Proposition 2.1]{HelgasonSome}. Recall the Iwasawa projections $H : G \rightarrow \mathfrak{a}$ and $\kappa : G \rightarrow K$, recall the function
\begin{equation*}
	f_{\lambda, \tau}(g) = \exp \big( (\lambda - \rho_{a}) H(g) \big) \tau (\kappa(g)) \quad (g \in G)
\end{equation*}
and the line $V_{\lambda, \tau} = \mathbb{C} \cdot f_{\lambda, \tau}$ lying in $\Pi_{\lambda, \tau}$. By Frobenius reciprocity, we have
\begin{equation*}
	\lbrace 0 \rbrace \neq  \Hom_K (\tau,  \pi_{\lambda, \tau})  = \Hom_G (V_{\lambda, \tau}, \pi_{\lambda, \tau}).
\end{equation*}
In particular, the $\tau^{-1}$-averaging of $f_{\lambda, \tau}$ over $K$
\begin{equation*}
	\varphi_{\lambda, \tau^{-1}}(g) = \int_{K} \tau^{-1}(k) f_{\lambda, \tau}(gk) dk
\end{equation*}
is a matrix coefficient of $\pi_{\lambda, \tau}$. Let ${\mathfrak{z} \in \mathfrak{Z} = \mathfrak{Z} (\mathcal{U}(\mathfrak{g}_{\mathbb{C}}))}$ and recall the Harish-Chandra isomorphism ${\gamma : \mathfrak{Z} \overset{\simeq}{\longrightarrow} \Symm(\mathfrak{h}_{\mathbb{C}})^{W(\mathfrak{h}_{\mathbb{C}})}}$ is given by $\gamma = \mathfrak{r}_{\rho} \circ \gamma'$ where ${\gamma' : \mathfrak{Z} \overset{\simeq}{\longrightarrow} \Symm(\mathfrak{h}_{\mathbb{C}})^{W(\mathfrak{h}_{\mathbb{C}})}}$ is induced by (see \cite[§3.2]{Wallach})
\begin{equation*}
	\mathfrak{Z} \hookrightarrow \mathcal{U}(\mathfrak{g}_{\mathbb{C}}) =  \mathcal{U}(\mathfrak{h}_{\mathbb{C}}) \oplus \big( \mathfrak{n}^- \mathcal{U}(\mathfrak{g}_{\mathbb{C}}) + \mathcal{U}(\mathfrak{g}_{\mathbb{C}})  \mathfrak{n} \big)
\end{equation*}
and $\mathfrak{r}_{\rho}$ is defined on $\mathfrak{h}_{\mathbb{C}}$ by $H \mapsto H - \rho(H) H$. As $\varphi_{\lambda, \tau}$ is a matrix coefficient of $\pi_{\lambda, \tau}$, it suffices to show that
\begin{equation*}
	\mathfrak{z} \cdot \varphi_{\lambda, \tau} = \chi_{\lambda + \rho_{\mathfrak{m}}}(\mathfrak{z}) \varphi_{\lambda, \tau}.
\end{equation*}
where $\chi_{\lambda + \rho_{\mathfrak{m}}} = (\lambda + \rho_{\mathfrak{m}}) \circ \gamma$. First we note that, for any $H = H_{\mathfrak{m}} + H_{\mathfrak{a}} \in \mathfrak{h}_{\mathbb{C}} =  \mathfrak{h}_{\mathfrak{m}_{\mathbb{C}}} \oplus \mathfrak{a}_{\mathbb{C}}$, we have
\begin{align*}
	f_{\lambda, \tau}(g \exp(H)) &= \exp \big( (\lambda - \rho_{\mathfrak{a}})H_{\mathfrak{a}} \big) f_{\lambda, \tau}(g) \\
	&= \exp \big( (\lambda + \rho_{\mathfrak{m}} - \rho) H \big) f_{\lambda, \tau}(g) \\
	&= H \cdot f_{\lambda, \tau}(g)
\end{align*}
for any $g \in G$. Then we can show, as in the proof of \cite[Proposition 2.1]{HelgasonSome}, that
\begin{equation*}
	\mathfrak{z} \cdot f_{\lambda, \tau} = \gamma' (\mathfrak{z}) f_{\lambda, \tau}
\end{equation*}
because $f_{\lambda, \tau}$ is right-$MN$-invariant. In other words, we have
\begin{equation*}
	\mathfrak{z} \cdot f = (\lambda + \rho_{\mathfrak{m}} - \rho) \circ \gamma'(\mathfrak{z}) f_{\lambda, \tau} = (\lambda + \rho_{\mathfrak{m}}) \circ \gamma (\mathfrak{z}) f_{\lambda, \tau} \quad (\mathfrak{z} \in \mathfrak{Z})
\end{equation*}
by definition of $\gamma$. This proves that the center $\mathfrak{Z}$ acts on $f_{\lambda, \tau}$ (and hence on $\varphi_{\lambda, \tau}$) by $\lambda + \rho_{\mathfrak{m}}$, and concludes the proof.
\end{proof}
\begin{rmrk} \label{rmrkOrtho}
In the proof of the main result, we shall use the analog of Lemma \ref{lem3.3.1} for the non-connected real Lie group $\Ortho(n,m)$ with $n > m \geq 1$; see \cite[Proposition 7.2]{BM2}. Note that the assumption $n > m$ ensures the extended Weyl group of $\mathfrak{a}$ (denoted by $\widetilde{W}_{\mathfrak{a}}$ in \cite[§7]{BM2}) equals the Weyl group (denoted by ${W}_{\mathfrak{a}}$ in \cite[§7]{BM2}).
\end{rmrk}
\subsection{Additional notations}
In view of the $\tau$-spherical inversion formula proved by \cite{Shimeno51}, we now restrict to the case $G$ has simple Lie algebra and reduced root system. We still work with a continuous character $\tau$ of $K$ and make the following assumption
\begin{equation*}
\text{whenever $\tau$ is non-trivial, $G$ is of Hermitian type.}
\end{equation*}
Recall the center of $\mathfrak{k}$ is either trivial or one dimensional: in the latter case, $G$ is said to be of \textit{Hermitian type}. We keep this setting until Section \ref{sec3.7}.

Let $\mathfrak{m}$ be the centralized of $\mathfrak{a}$ in $\mathfrak{k}$ and
\begin{equation*}
\mathfrak{g} = \mathfrak{m} \oplus \mathfrak{a} \oplus \sum_{\alpha \in \Sigma(\mathfrak{a})} \mathfrak{g}_{\alpha}
\end{equation*}
be the root space decomposition of $\mathfrak{g}$. We choose a subset $\Sigma^+ (\mathfrak{a}) \subset \Sigma (\mathfrak{a})$ of positive roots and let $\Sigma_s^+ (\mathfrak{a}) \subset \Sigma^+ (\mathfrak{a})$ be the set of simple roots. Let $M$, $A$ and $N$ be the analytic subgroups of $G$ corresponding to $\mathfrak{m}$, $\mathfrak{a}$ and
\begin{equation*}
\mathfrak{n} \overset{\defin}{=} \sum_{\alpha \in \Sigma^+(\alpha)} \mathfrak{g}_{\alpha}
\end{equation*}
respectively. Let $G = NAK$ be the associated Iwasawa decomposition and let $\kappa : G \rightarrow K$, $H : G \rightarrow \mathfrak{a}$ be the associated Iwasawa projections. Note that $P = NAM$ is a minimal parabolic subgroup of $G$. Let $\overline{N}$ be the analytic subgroup of $G$ corresponding to
\begin{equation*}
\overline{\mathfrak{n}} \overset{\defin}{=} \sum_{- \alpha \in \Sigma^+ (\mathfrak{a})} \mathfrak{g}_{\alpha}.
\end{equation*}
Let $\sigma \subset \Sigma_s^+(\mathfrak{a})$ be a subset of the simple roots and define
\begin{equation*}
	\langle \sigma \rangle = \Sigma^+(\mathfrak{a}) \bigcap \left( \sum_{\alpha \in \sigma} \mathbb{N} \alpha \right)
\end{equation*}
the system of positive roots generated by $\sigma$. Let $W_{\sigma}(\mathfrak{a}) \subset W(\mathfrak{a})$ be the subgroup generated by reflections $s_{\alpha}$ with $\alpha \in \sigma$ and define the associated standard (\textit{i.e} contained in $P$) parabolic subgroup
\begin{equation*}
	P_{\sigma} = P W_{\sigma}(\mathfrak{a}) P
\end{equation*}
of $G$. Let $P_{\sigma} = N_{\sigma} A_{\sigma} M_{\sigma}$ be its Langlands decomposition and let $\overline{N_{\sigma}}$ be the analytic subgroup of $G$ corresponding to
\begin{equation*}
	\overline{\mathfrak{n}_{\sigma}} = \sum_{-\alpha \in \Sigma^+ (\mathfrak{a}) - \langle \sigma \rangle} \mathfrak{g}_{\alpha}.
\end{equation*}
The Lie algebra of $A_{\sigma}$ is given by
\begin{equation*}
	\mathfrak{a}_{\sigma} = \langle a \in \mathfrak{a} : \alpha(a) = 0 \text{, } \forall \alpha \in \sigma \rangle.
\end{equation*}
We let $a(\sigma)$ denote the orthogonal complement of $\mathfrak{a}_{\sigma}$ in $\mathfrak{a}$ (with respect to $\langle \cdot , \cdot \rangle$), $N(\sigma)$ denote the analytic subgroup of $G$ associated to
\begin{equation*}
	\mathfrak{n}(\sigma) = \sum_{\alpha \in \langle \sigma \rangle} \mathfrak{g}_{\alpha}
\end{equation*}
and $\overline{N(\sigma)}$ the analytic subgroup of $G$ associated to
\begin{equation*}
	\overline{\mathfrak{n}(\sigma)} = \sum_{- \alpha \in \langle \sigma \rangle} \mathfrak{g}_{\alpha}.
\end{equation*}
Note that, if $\sigma$ is the empty set, then $P_{\sigma} = P$, $A_{\sigma} = A$, $M_{\sigma} = M$, $N_{\sigma} = N$ and $\overline{N_{\sigma}} = \overline{N}$. 
\subsection{$\tau$-spherical inversion formula}
Let $\sigma \subset \Sigma^+_s (\mathfrak{a})$ be a subset of the simple roots and $\langle \sigma \rangle$ be the system of positive roots generated by $\sigma$. The integral
\begin{equation*}
	\mathbf{c}^{\sigma}(\lambda ; \tau) = \int_{\overline{N_{\sigma}}} a(\overline{n})^{\lambda + \rho_{\mathfrak{a}}} \tau^{-1}(\kappa (\overline{n})) d\overline{n}
\end{equation*}
is absolutely convergent on $\lambda \in \mathfrak{a}_{\mathbb{C}}^*$ with $\re (\langle \lambda, \alpha \rangle) > 0$ and $\alpha \in \Sigma^+(\mathfrak{a}) - \langle \sigma \rangle$. Similarily, the integral
\begin{equation*}
	\mathbf{c}_{\sigma}(\lambda ; \tau) = \int_{\overline{N({\sigma})}} a(\overline{n})^{\lambda + \rho_{\mathfrak{a}}} \tau^{-1}(\kappa (\overline{n})) d\overline{n}
\end{equation*}
is absolutely convergent on $\lambda \in \mathfrak{a}_{\mathbb{C}}^*$ with $\re (\langle \lambda, \alpha \rangle) > 0$ and $\alpha \in \langle \sigma \rangle$. We normalize these integral by choosing the unique Haar measure $d\overline{n}$ on $\overline{N}$ such that
\begin{equation*}
\mathbf{c}^{\sigma}(\lambda ; \mathbf{1}) = 1 \text{ and } \mathbf{c}_{\sigma}(\lambda ; \mathbf{1}) = 1.
\end{equation*}
When $\tau = \mathbf{1}$ and $\sigma$ is the empty set, we simply write $\mathbf{c}(\lambda)$ for $\mathbf{c}^{\emptyset}(\lambda ; \mathbf{1})$. In this case, there is a unique $\sigma$-finite measure $\mu_{\text{Pl}}$ on $\mathfrak{a}^*_{\mathbb{C}}$ for which the Harish-Chandra transform
\begin{equation*}
\widehat{k}(\lambda) = \int_{G} k(g) \varphi_{- \lambda}(g) dg \quad (k \in C_c^{\infty}(S, \mathbf{1}, \mathbf{1}))
\end{equation*}
extends to an isometry
\begin{equation*}
	L^2(S, \mathbf{1}, \mathbf{1}) \longrightarrow L^2 ( \mathfrak{a}^*_{\mathbb{C}}, \vert W(\mathfrak{a}) \vert^{-1} \mu_{\text{Pl}} ).
\end{equation*}
In particular, we have the inversion formula
\begin{equation*}
	k(g) = \int_{\mathfrak{a}_{\mathbb{C}}^*} \widehat{k} (\lambda) \varphi_{\lambda} (g) \dfrac{d\mu_{\text{Pl}} (\lambda)}{\vert W(\mathfrak{a}) \vert}
\end{equation*}
for $k \in C_{c}^{\infty}(S, \mathbf{1}, \mathbf{1})$. This measure $\mu_{\text{Pl}}$, called the \textit{spherical Plancherel measure}, was determined by Gangolli \cite{Gangolli} and Helgason \cite{Helgason} and is given explicitely by (see \textit{e.g.} \cite[Theorem 3.5]{Gangolli})
\begin{equation*}
	d\mu_{\text{Pl}}(\lambda) = \dfrac{d\mu(\lambda)}{\vert \mathbf{c} (\lambda) \vert^2}
\end{equation*}
where $d\mu$ is the Lebesgue measure on $i \mathfrak{a}^*$.

For non-trivial $\tau$, the same result holds: there is a unique $\sigma$-finite measure $\mu_{\text{Pl}, \tau}$ on $\mathfrak{a}^*_{\mathbb{C}}$ for which the $\tau$-Harish-Chandra transform
\begin{equation*}
\widehat{k}(\lambda) = \int_{G} k(g) \varphi_{\tau^{-1}, - \lambda}(g) dg \quad (k \in C_c^{\infty}(S, \tau, \tau))
\end{equation*}
extends to an isometry
\begin{equation*}
	L^2(S, \tau, \tau) \longrightarrow L^2 ( \mathfrak{a}^*_{\mathbb{C}}, \vert W(\mathfrak{a}) \vert^{-1} \mu_{\text{Pl}, \tau} ).
\end{equation*}
In particular, we have the inversion formula
\begin{equation*}
	k(g) = \int_{\mathfrak{a}_{\mathbb{C}}^*} \widehat{k} (\lambda ; \tau) \varphi_{\lambda, \tau} (g) \dfrac{d\mu_{\text{Pl}, \tau} (\lambda)}{\vert W(\mathfrak{a}) \vert}
\end{equation*}
for $k \in C_{c}^{\infty}(S, \tau, \tau)$. Let
\begin{equation*}
	\sigma_i = \lbrace \beta_{m - (i - 1)}, \ldots, \beta_m \rbrace
\end{equation*}
for $1 \leq i \leq m$. The \textit{$\tau$-spherical Plancherel measure} $\mu_{\text{Pl}, \tau}$ was determined by Shimeno \cite{Shimeno50} and is given by
\begin{equation*}
	d\mu_{\text{Pl}, \tau} = \sum_{i = 0}^m d\mu^{(i)}_{\text{Pl}, \tau}
\end{equation*}
where
\begin{equation*}
	d\mu^{(0)}_{\text{Pl}, \tau} = \dfrac{d\mu(\lambda)}{\vert \mathbf{c} (\lambda ; \tau) \vert^2}
\end{equation*}
with $d\lambda$ the Lebesgue measure on $i \mathfrak{a}^*$ and where each $d\mu^{(i)}_{\text{Pl}, \tau}$ are explicitely described in terms of the various $\mathbf{c}$-functions $\mathbf{c}^{\sigma_i}(\lambda ; \tau)$, $\mathbf{c}_{\sigma_i}(\lambda ; \tau)$ on certain subsets of $i \mathfrak{a}^*$ indexed by $i$ (see also \cite[§4.6]{BM2}) .
\subsection{Spectral density function} \label{sec3.5}
Let $\sigma \subset \Sigma_s^+ (\mathfrak{a})$ be a subset of the simple roots. We define the $\sigma$-spectral density function of $S = G/K$ by
\begin{equation} \label{eq3.1}
	\beta_S^{\sigma} (\lambda) = \vert \mathbf{c}^{\sigma} (\lambda) \vert^{-2}.
\end{equation}
If $\sigma$ is empty, then we write $\beta_S (\lambda)$ for the spectral density function $\beta_S^{\emptyset} (\lambda) = \vert \mathbf{c} (\lambda) \vert^{-2}$. We also define
\begin{equation} \label{eq3.2}
	\widetilde{\beta_S^{\sigma}} (\lambda) {=} \prod_{\alpha \in \Sigma^+ (\mathfrak{a}) - \langle \sigma \rangle} (1 + \vert \langle \lambda , \alpha^{\vee} \rangle \vert )^{d_{\alpha}}
\end{equation}
where $d_{\alpha} = m_{\alpha} + m_{2 \alpha}$ and $\alpha^{\vee} = 2 \alpha / \langle \alpha, \alpha \rangle$. If $\sigma$ is empty, then we write $\widetilde{\beta_S} (\lambda)$ for $\widetilde{\beta_S^{\emptyset}} (\lambda)$ and we have
\begin{equation*}
	\widetilde{\beta_S^{\sigma}} (\lambda) \leq \widetilde{\beta_S} (\lambda) \text{ } (\sigma \subseteq \Sigma_s^+(\mathfrak{a})).
\end{equation*}
This $\widetilde{\beta_S^{\sigma}}$ is a majorant of $\beta_S^{\sigma}$ on $i \mathfrak{a}^*$ in the sense that (see \cite[Lemma 5.6 (1)]{BM2})
\begin{equation*}
\beta_S^{\sigma} (\lambda) \ll \widetilde{\beta_S^{\sigma}}(\lambda) \text{ } (\lambda \in i \mathfrak{a}^*).
\end{equation*}
Recall $\lambda \in \mathfrak{a}^*_{\mathbb{C}}$ is said to be regular if $\langle \lambda, \alpha \rangle \neq 0$ for all $\alpha \in \Sigma(\mathfrak{a})$ and sufficiently regular if there is $T > 0$ such that $\vert \langle \lambda, \alpha \rangle \vert \geq T$ for all $\alpha \in \Sigma^+(\mathfrak{a})$. For sufficiently regular $\lambda \in i \mathfrak{a}^*$, then $\widetilde{\beta_S^{\sigma}}(\lambda)$ approximate $\beta_S^{\sigma}(\lambda)$. More precisely if $\lambda \in i \mathfrak{a}^*$ satisfies $\vert \langle \lambda, \alpha \rangle \vert \geq T$ for all $\alpha \in \Sigma^+(\mathfrak{a})$ (for some $T > 0$), then \cite[(3.44a)]{DKV}
 \begin{equation} \label{eqreg}
\beta_S^{\sigma} (\lambda) \asymp_T \widetilde{\beta_S^{\sigma}}(\lambda)
\end{equation}
Moreover, for any $\lambda, \nu \in \mathfrak{a}^*_{\mathbb{C}}$ and any $\alpha \in \Sigma^+(\mathfrak{a})$, using triangle inequality, we have
\begin{align*}
	1 + \vert \langle \lambda + \nu, \alpha^{\vee} \rangle \vert
	&\leq 1 +  \vert \langle \lambda, \alpha^{\vee} \rangle \vert +  \vert \langle \nu, \alpha^{\vee} \rangle \vert \\
	&\leq 1 + \Vert \lambda \Vert + \vert \langle \nu, \alpha^{\vee} \rangle \vert \\
	&\leq (1 + \Vert \lambda \Vert) (1 + \vert \langle \nu, \alpha^{\vee} \rangle \vert)
\end{align*}
and thus
\begin{equation} \label{eq3.4}
	\widetilde{\beta_S^{\sigma}} (\lambda + \nu) \ll (1 + \Vert \lambda \Vert)^{\delta^{\sigma}} \widetilde{\beta_S^{\sigma}} (\lambda)  
\end{equation}
where $\delta^{\sigma} = \sum_{\alpha \in \Sigma^+(\mathfrak{a}) - \langle \sigma \rangle} d_{\alpha}$. From these, we deduce that, for any $c > 0$, we have (see \cite[Lemma 5.6 (3)]{BM2})
\begin{equation} \label{eq3.5}
	\int_{\substack{\mu \in i \mathfrak{a}^* \\ \Vert \mu - \nu \Vert < c}} \beta_S^{\sigma} (\lambda)  d\lambda(\mu) \asymp_c \widetilde{\beta_S^{\sigma}}(\nu)
\end{equation}
for any $\nu \in i \mathfrak{a}^*$.
\subsection{$\tau$-spherical test function} \label{sec3.6}
Following \cite[§5]{BM2} (see also \cite[§4]{BM1} for trivial $\tau$), we construct a function ${k \in C^{\infty}_c(S, \tau, \tau)}$ whose $\tau$-Harish-Chandra concentrate around a given $\nu \in i\mathfrak{a}^*$.
\begin{proposition} \cite[Proposition 5.5]{BM2} \label{prop3.7.1}
There are constants $R_0, c > 0$ such that the following holds. Let $\nu \in i \mathfrak{a}^*$ and let $0 < R \leq R_0$. Then, there is function $k_{\nu, \tau} \in C^{\infty}_c (S, \tau, \tau)$ whose support is contained in $G_R$ and whose Harish-Chandra transform $\widehat{k_{\nu, \tau}}$ satisfies
\begin{enumerate}
\item $\widehat{k_{\nu, \tau}} (\lambda) \ll_{A,R} \exp (R \Vert \re (\lambda) \Vert) \sum_{w \in W} (1 + \Vert w \cdot \lambda - \nu \Vert)^{-A}$ for all $\lambda \in \mathfrak{a}^*_{\mathbb{C}}$
\item $\widehat{k_{\nu, \tau}}(\lambda) \geq 0$ for all $\lambda \in \mathfrak{a}^*_{\tau, \text{unit}}$ \label{propppp2}
\item $\widehat{k_{\nu, \tau}}(\lambda) \geq c$ for all $\lambda \in \mathfrak{a}^*_{\tau, \text{unit}}$ satisfying $\Vert \im(\lambda) - \nu \Vert \leq 1$. \label{propppp3}
\end{enumerate}
\end{proposition}
\begin{proof}
We briefly describe the construction of $k_{\nu}$ as it will be useful in Section \ref{chap7} (in particular \ref{ss7.33} and \ref{secc7.6}). We fix $R > 0$, $\nu \in i \mathfrak{a}^*$ and we start with a bump function $g$ on $\mathfrak{a}$, supported on $B_{\mathfrak{a}}(0, R/2)$ (in particular $g$ is smooth, non-negative, even with integral over $\mathfrak{a}$ of value $1$). Then we let $h$ be its Fourier transform: a function on $\mathfrak{a}^*_{\mathbb{C}}$ which is even and non-negative on $i\mathfrak{a}^*$. Then, we center $h$ at $\nu$ and force $W$-invariance by considering
\begin{equation*}
	h_{\nu}(\lambda) = \sum_{w \in W} h(w \cdot \lambda - \nu).
\end{equation*}
Finally, we put $k_{\nu, \tau} = k^0_{\nu, \tau} \ast k^0_{\nu, \tau}$, where $k_{\nu, \tau}^0$ is the inverse $\tau$-spherical transform of $h_{\nu}$, and $k_{\nu,\tau}$ is in $C^{\infty}_{R} (S, \tau, \tau)$.
\end{proof}
Assume $\tau$ is trivial and let $\nu \in i \mathfrak{a}^*$. Using the Plancherel inversion formula and a bound on spherical functions \cite[Theorem 2]{BlomerPohl}, we can bound a test function $k_{\nu} = k_{\nu, \mathbf{1}}$ (with support in $G_R$ for some $R > 0$) given by the above Proposition. More precisely, we have (see \cite[Lemma 5.9]{BM2})
\begin{equation} \label{eq3.6}
	k_{\nu} (\exp (a)) \ll_{R} (1 + \Vert \nu \Vert \Vert a \Vert)^{-1/2} \widetilde{\beta_S}(\nu) \quad (a \in \mathfrak{a}).
\end{equation}
In particular, this provides an upper bound on $k_{\nu}(e)$.
\begin{lemma} \label{lem3.7.3}
For any $\nu \in i \mathfrak{a}^*$, we have
	\begin{equation*}
		k_{\nu}(e) \asymp \widetilde{\beta_S}(\nu).
	\end{equation*}
\end{lemma}
\begin{proof}
Using Plancherel inversion formula, we have
\begin{equation*}
	k_{\nu}(e) = \int_{i \mathfrak{a}^*} \widehat{k}(\lambda) \varphi_{\lambda}(e) \dfrac{d\mu (\lambda)}{\vert \mathbf{c} (\lambda) \vert^2}
\end{equation*}
with $\varphi_{\lambda}(e) = 1$. We apply Property \ref{propppp2} of $k_{\nu}$ to drop all terms but those for which ${\Vert \lambda - \nu \Vert \leq 1}$ and then Property \ref{propppp3} of $k_{\nu}$ to obtain
\begin{equation*}
	k_{\nu} (e) \geq \int_{\Vert \lambda - \nu \Vert \leq 1} \beta_S (\lambda) {d\mu(\lambda)}.
\end{equation*}
Using \eqref{eq3.5} (with $\sigma$ the emptyset) we find $k_{\nu}(e) \gg \widetilde{\beta_S}(\nu)$.
\end{proof}

\section{Theta correspondence} \label{chap2}
In this Section, we recall the construction of the Weil representation (in the local non-archimedean case) and recollect some facts about the local and global theta correspondence. For non-archimedean local fields the main result, Theorem \ref{theo2.5.1}, is due to Waldpsurger and is used to deduce a statement about the preservation of the level in the theta correspondence: in particular, that spherical (or unramified) representations corresponds to spherical representations. For archimedean local fields, the main results are Propositions \ref{prop3.6.1} and \ref{prop3.6.2}: in contrast to the non-archimedean case, we show that spherical representations correponds to $\tau$-spherical representations. We end this section with the global theta correspondence.
\subsection{Weil representation} \label{sec2.1}
Let $F$ be a non-archimedean local field of characteristic not $2$. Let $(\mathcal{W}, \langle \cdot, \cdot \rangle)$ be a non-degenerate symplectic $F$-space and let $H(\mathcal{W})$ be the associated Heisenberg group, namely
\begin{equation*}
H(\mathcal{W}) = \mathcal{W} \times F
\end{equation*}
with group law
\begin{equation*}
	(w_1, \lambda_1) \cdot (w_2, \lambda_2) = (w_1 + w_2, \lambda_1 + \lambda_2 + \langle w_1, w_2 \rangle / 2).
\end{equation*}
Note that the center of $H(\mathcal{W})$ is $\lbrace (w , \lambda) \in \mathcal{W} \times F : w = 0 \rbrace \simeq F$.
\begin{theorem}[Stone-von Neumann]
Let ${\psi : F \rightarrow \mathbb{C}^{\times}}$ be a non-trivial additive character. There is a unique (up isomorphism) irreducible smooth admissible representation $\rho_{\psi}$ of $H(\mathcal{W})$ with central character $\psi$.
\end{theorem}
\begin{proof}
See \cite[Chapitre 2, §I]{MVW} or, in a more general context, see \cite[§1]{Tata3}.
\end{proof}

Fix a non-trivial additive character $\psi$ of $F$, let $\rho_{\psi}$ be the irreducible representation given by Stone-von Neumann Theorem acting on a $\mathbb{C}$-vector space $S_{\psi}$ and let $\Sp(\mathcal{W})$ act on $H(\mathcal{W})$ by $g \cdot (w, \lambda) = (g \cdot w, \lambda)$. For any $g \in \Sp(\mathcal{W})$, the representation
\begin{equation*}
	\rho^g : h \mapsto \rho (g \cdot h)
\end{equation*}
of the Heisenberg group has central character $\psi$. By Stone-von Neumann Theorem, there is $M(g) \in \GL(S_{\psi})$ such that
\begin{equation} \label{eq2.1}
	M(g) \circ \rho_{\psi}^g (w, \lambda) = \rho_{\psi} (w, \lambda) \circ M(g).
\end{equation}
Schur's Lemma then ensures we have $\dim_{\mathbb{C}} \Hom_F (\rho_{\psi}, \rho_{\psi}) = 1$ and thus $M(g)$ is well-defined, up to $\mathbb{C}^{\times}$.
In other words, the map $g \mapsto M(g)$ defines a projective representation
\begin{equation*}
	\Sp ({\mathcal{W}}) \longrightarrow \PGL(S_{\psi}).
\end{equation*}
We define the $\mathbb{C}^{\times}$-metaplectic group $\widetilde{\Sp}(\mathcal{W})^{\mathbb{C}^{\times}}_{\psi}$ by
\begin{equation*}
\widetilde{\Sp}(\mathcal{W})^{\mathbb{C}^{\times}}_{\psi} = \lbrace (g, M(g)) \in \Sp(\mathcal{W}) \times \PGL(S_{\psi}) : M(g) \circ \rho_{\psi}^g (w, \lambda) = \rho_{\psi} (w, \lambda) \circ M(g) \rbrace.
\end{equation*}
Then, we have an exact sequence
\begin{equation*}
	1 \longrightarrow \mathbb{C}^{\times} \longrightarrow \widetilde{\Sp}(\mathcal{W})^{\mathbb{C}^{\times}}_{\psi}  \longrightarrow \Sp(\mathcal{W}) \longrightarrow 1.
\end{equation*}
One can show (see \textit{e.g.} \cite[§1.3]{Thetabook}) $\rho$ is unitarizable and hence $\widetilde{\Sp}(\mathcal{W})^{\mathbb{C}^{\times}}_{\psi}$ naturally contains the $\mathbb{S}^1$-metaplectic group
\begin{equation*}
\widetilde{\Sp}(\mathcal{W})^{\mathbb{S}^{1}}_{\psi} = \lbrace (g, M(g)) : \Sp(\mathcal{W}) \times \U(S_{\psi}) : M(g) \circ \rho_{\psi}^g (w, \lambda) = \rho_{\psi} (w, \lambda) \circ M(g) \rbrace
\end{equation*}
where $\mathbb{S}^{1} = \lbrace z \in \mathbb{C}^{\times} : \vert z \vert = 1 \rbrace$ with corresponding exact sequence
\begin{equation*}
	1 \longrightarrow \mathbb{S}^{1} \longrightarrow \widetilde{\Sp}(\mathcal{W})^{\mathbb{S}^{1}}_{\psi} \longrightarrow \Sp(\mathcal{W}) \longrightarrow 1.
\end{equation*}
To ease notations we simply write $\widetilde{\Sp}(\mathcal{W})_{\psi}$ for $\widetilde{\Sp}(\mathcal{W})^{\mathbb{S}^{1}}_{\psi}$ and call $\widetilde{\Sp}(\mathcal{W})_{\psi}$ the metaplectic group. Let $\omega_{\psi}$ the projection
\begin{equation*}
	\omega_{\psi} : \widetilde{\Sp}(\mathcal{W})_{\psi} \longrightarrow \U(S_{\psi}) \text{, } (g,M(g)) \mapsto M(g).
\end{equation*}
Then, $\omega_{\psi}$ is a genuine (unitary) representation of the metaplectic group called the Weil representation.
\subsection{Local theta lift} \label{sec2.2}
Let $F$ be a non-archimedean local field of characteristic not $2$ and let $E$ be $F$ itself or a quadratic extension. Let $V$ be a finite-dimensional non-degenerate Hermitian $E$-space and $W$ be a finite-dimensional non-degenerate skew-Hermitian $E$-space. By restricting scalars, we may view
\begin{equation*}
\mathcal{W} = V \otimes W
\end{equation*}
as a vector space over $F$, naturally equipped with the non-degenerate symplectic form
\begin{equation*}
	\langle v_1 \otimes w_1, v_2 \otimes w_2 \rangle_{\mathcal{W}} = \tr_{E/F} \big( \langle v_1, v_2 \rangle_V \langle w_1, w_2 \rangle_W \big).
\end{equation*}
This way, we obtain a (not necessarily injective) map of isometry groups
\begin{equation*}
	\iota : \U(V) \times \U(W) \longrightarrow \Sp(\mathcal{W}).
\end{equation*}
Identifying $\U(V)$ and $\U(W)$ with their images, we get an irreducible dual reductive pair of type $1$ in $\Sp(\mathcal{W})$. In particular \cite[§5]{HoweTheta}
\begin{itemize}
\item $\U(V)$ and $\U(W)$ are mutual commutants of each other;
\item $\U(V)$ and $\U(W)$ act absolutely reductively on $\mathcal{W}$.
\end{itemize}
Let $\psi : F \longrightarrow \mathbb{C}^{\times}$ be a non-trivial additive character and let $\omega = \omega_{\psi}$ be the Weil representation of the metaplectic group $\widetilde{\Sp}(\mathcal{W}) = \widetilde{\Sp}(\mathcal{W})^{\mathbb{S}^{1}}_{\psi}$. For any subgroup ${G \subseteq \Sp(\mathcal{W})}$, we let $\widetilde{G} \subseteq \widetilde{\Sp}(\mathcal{W})$ denote its inverse image. Since $\U(V), \U(W)$ are mutual commutants of each other, the same holds for their inverse images $\widetilde{\U}(V), \widetilde{\U}(W)$ under the natural projection
\begin{equation*}
\widetilde{\Sp}(\mathcal{W}) \longrightarrow {\Sp}(\mathcal{W}).
\end{equation*}
The restriction of $\omega$ to the subgroup $\widetilde{\U}(V) \cdot \widetilde{\U}(W)$ is a representation of $\widetilde{\U}(V) \times \widetilde{\U}(W)$.

Let $\Irr(\widetilde{G})$ denote the set of (equivalence classes of) irreducible smooth representations of $\widetilde{G}$ and set 
\begin{equation*}
\mathcal{R}(\widetilde{G}) = \lbrace \pi \in \Irr(\widetilde{G}) : \Hom_{\widetilde{G}}(\omega, \pi) \neq 0 \rbrace.
\end{equation*}
Let $\Pi \in \Irr(\widetilde{\U}(V) \cdot \widetilde{\U}(W))$, then \cite{Flath}
\begin{equation*}
\Pi \simeq \pi_V \otimes \pi_W
\end{equation*}
for some $\pi_V \in \Irr(\widetilde{\U}(V))$ and $\pi_W \in \Irr(\widetilde{\U}(W))$. If, moreover, $\Pi \in \mathcal{R}(\widetilde{\U}(V) \cdot \widetilde{\U}(W))$, then $\pi_V \in  \mathcal{R}(\widetilde{\U}(V))$ and $\pi_W \in  \mathcal{R}(\widetilde{\U}(W))$. Hence, $\mathcal{R}(\widetilde{\U}(V) \cdot \widetilde{\U}(W))$ identifies with a subset of $\mathcal{R}(\widetilde{\U}(V) \times \mathcal{R}(\widetilde{\U}(W))$ and Howe \cite{HoweTheta} conjectured that $\mathcal{R}(\widetilde{\U}(V) \cdot \widetilde{\U}(W))$ is the graph of a bijection between $\mathcal{R}(\widetilde{\U}(V))$ and $\mathcal{R}(\widetilde{\U}(W))$.

We can reformulate this conjecture as follows. Let $\pi_V \in  \mathcal{R}(\widetilde{\U}(V))$ and consider the maximal $\pi_V$-isotypic quotient $\omega / \omega[\pi_V]$ of $\omega$ where
\begin{equation*}
	\omega[\pi_V] = \bigcap_{f \in \Hom_{\widetilde{\U}(V)}(\omega, \pi_V)} \ker (f).
\end{equation*}
Then (see \cite[Chapitre 2, §III.4]{MVW})
\begin{equation*}
	\omega / \omega[\pi_V] \simeq \pi_V \otimes \Theta(\pi_V ; W)
\end{equation*}
for some smooth representation $\Theta(\pi_V ; W)$ of $\U(W)$. Kudla \cite{Kudla} showed that $\Theta(\pi_V ; W)$ has finite (possibly zero) length and we may consider the (possibly zero) maximal semi-simple quotient $\theta(\pi_V ; W)$ of $\Theta(\pi_V ; W)$.\\
Note that $\Theta(\pi_V ; W)$ may also be defined as the module of co-$\widetilde{\U}(V)$-invariants of $\omega \otimes \pi_V$ and Howe's duality conjecture may be stated as
\begin{equation*}
	\dim \Hom_{\widetilde{U}(W)} \big( \theta(\pi_V ; W), \theta(\pi_V' ; W) \big) \leq \delta_{\pi_V, \pi_V'}
\end{equation*}
for any $\pi_V, \pi_V' \in \Irr(\widetilde{U}(V))$ where
\[
\delta_{\pi_V, \pi_V'} = \left. \begin{cases} 1& \text{if } \pi_V \simeq \pi_V'; \\ 0 & \text{if } \pi_V \not\simeq \pi_V'. \end{cases} \right.
\]
This conjecture is a theorem due to Waldspurger \cite{Wald2} if $F$ is $p$-adic with $p \neq 2$ and Gan-Takeda \cite{GanTakeda} unconditionally (\textit{i.e.} for any non-archimedean local field of characteristic not $2$). We discuss the archimedean case in Section \ref{sec3.7}.
\subsection{Models of the Weil representation} \label{sec2.3}
In this Section, we assume $F$ is $p$-adic \textit{i.e.} non-archimedean with characteristic $0$. A common way to construct a model for the Weil representation is by induction. Let $\psi$ be a non-trivial additive character of $F$, $\mathcal{A} \subset \mathcal{W}$ be closed subgroup and assume $\mathcal{A} = \mathcal{A}^{\perp}$, where
\begin{equation*}
	\mathcal{A}^{\perp} =  \lbrace w \in \mathcal{W} : \psi (\langle w, a \rangle) = 1 \text{, } \forall a \in \mathcal{A} \rbrace.
\end{equation*}
Let $H(\mathcal{A}) = \mathcal{A} \times F$ be the Heisenberg group associated to $\mathcal{A}$, a subgroup of $\mathcal{H}(\mathcal{W})$, and extend $\psi$ to a character $\psi_{\mathcal{A}}$ of $H(\mathcal{A})$ \textit{i.e.} $\psi_{\mathcal{A}} : H(\mathcal{A}) \longrightarrow \mathbb{C}^{\times}$ such that
\begin{equation*}
	\psi_{\mathcal{A}}(0, \lambda) = \psi (\lambda) \quad (\lambda \in F).
\end{equation*}
Then, by inducing $\psi_{\mathcal{A}}$ to $H(\mathcal{W})$, one obtain an irreducible representation $\rho_{\mathcal{A}}$ of $H(\mathcal{W})$ with central character $\psi$ and one can write down explicitly the automorphisms $M(g)$ of $\rho_{\mathcal{A}}$ defined in \eqref{eq2.1}, see \textit{e.g.} \cite[Chapitre 2, §II.2]{MVW}.
\subsubsection{Schrödinger model}
Let $\mathcal{W} = \mathcal{X} \oplus \mathcal{Y}$ be a complete polarisation. Then $\mathcal{X} = \mathcal{X}^{\perp}$ and the space $\rho_{\mathcal{X}}$ identifies with smooth and locally constant functions on $\mathcal{Y}$. Writing $g \in \Sp(\mathcal{W})$ as
\begin{equation*}
	g = \begin{pmatrix} a & b \\ c & d \end{pmatrix}
\end{equation*}
for $a \in \End(\mathcal{X}) \text{, } d \in \End(\mathcal{Y}) \text{, } b \in \Hom (\mathcal{Y}, \mathcal{X}) \text{, } c \in \Hom (\mathcal{X}, \mathcal{Y})$, then \cite[Chapitre 2, §II.6]{MVW}
\begin{itemize}
\item $M(g) f (y) = \vert \det_X (a) \vert^{1/2} f( a^{\vee} y)$ for $g = \begin{pmatrix} a & \\ & a^{\vee} \end{pmatrix}$ with $a \in \GL(\mathcal{X})$ and $a^{\vee} \in \GL(\mathcal{Y})$ uniquely determined by the relation
	\begin{equation*}
		\langle ax , a^{\vee} y \rangle = \langle x,y \rangle \quad (x \in \mathcal{X}, y \in \mathcal{Y}).
	\end{equation*}
\item $M(g) f (y) = \psi ( \langle b y, y \rangle / 2) f(y)$ for $g = \begin{pmatrix} \id & b \\ & \id \end{pmatrix}$ with $b \in \Hom (\mathcal{Y}, \mathcal{X})$ symmetric \textit{i.e.} satisfying
	\begin{equation*}
		\langle b y_1, y_2 \rangle = \langle by_2, y_1 \rangle \quad (y_1, y_2 \in \mathcal{Y}).
	\end{equation*}
\end{itemize}
\subsubsection{Lattice model}
We assume $p \neq 2$ (see \cite{TakedaLattice} for $p = 2$). Let $\mathcal{L} \subset \mathcal{W}$ be a self-dual lattice, that is a free-$\mathcal{O}$-module $\mathcal{L} $ of maximal rank such that $\mathcal{L} = \mathcal{L}^{\perp}$. Such lattices always exists in non-degenerated symplectic spaces \cite[Chapitre 1, §I.11]{MVW}. The space of $\rho_{\mathcal{L}}$ identifies with smooth and locally constant functions $f$ on $\mathcal{W}$ such that
\begin{equation*}
	f(\ell + w) = \psi (\langle w, \ell \rangle / 2) f(w)  \quad (\ell \in \mathcal{L}, w \in \mathcal{W}).
\end{equation*}
Then \cite[Chapitre 2, §II.8]{MVW}
\begin{equation*}
	M(k) f (w) = f (k^{-1} \cdot w)
\end{equation*}
for all $k \in K$, the stabilizer of $\mathcal{L}$ in $\Sp(\mathcal{W})$. Moreover, this defines a splitting of $K$ in $\widetilde{\Sp}(\mathcal{W})$ \cite[Chapitre 2, §II.10]{MVW}.
\subsection{Splitting metaplectic cover} \label{sec2.4}
We still assume $F$ is $p$-adic. Let $E$ be $F$ itself or a quadratic extension. Let $V$ be a finite-dimensional non-degenerate Hermitian $E$-space and let $W$ be a finite-dimensional non-degenerate skew-Hermitian $E$-space.
\begin{equation} \label{assump}
\text{We assume both $\dim_E V$ and $\dim_E W$ are even.}
\end{equation}
Let $\psi : F \rightarrow \mathbb{C}^{\times}$ be a non-trivial additive character and let $\omega = \omega_{\psi}$ denote the Weil representation. Recall the metaplectic cover
\begin{equation*}
	1 \longrightarrow \mathbb{S}^{1} \longrightarrow \widetilde{\Sp}(\mathcal{W}) \longrightarrow \Sp (\mathcal{W}) \longrightarrow 1
\end{equation*}
where $\mathcal{W}$ is the symplectic vector space $V \otimes_E W$. By assumption \eqref{assump}\footnote[1]{This assumption is not necessary when $E \neq F$, see \cite[Chapitre 3, Théorème I.1]{MVW}}, one can split the metaplectic group over the dual reductive pair $(\U(V), \U(W))$. In other words, the morphism
\begin{equation*}
	\iota : \U(V) \times \U(W) \longrightarrow \Sp(\mathcal{W})
\end{equation*}
can be lifted to the metaplectic group \textit{i.e.} there is a morphism
\begin{equation*}
	\widetilde{\iota} : \U(V) \times \U(W) \longrightarrow \widetilde{\Sp}(\mathcal{W})
\end{equation*}
such that $\iota = \text{proj} \circ \widetilde{\iota}$, where
\begin{equation*}
\text{proj} : \widetilde{\Sp}(\mathcal{W}) \longrightarrow {\Sp(\mathcal{W})}
\end{equation*}
is the natural projection. By pulling back, we can define the Weil representation of $\U(V) \times \U(W)$; namely, by setting $\omega \circ \widetilde{\iota}$. Such splittings can be constructed using only $\psi$ \footnote[2]{More precisely, they depend on an additional pair of characters $(\chi_V, \chi_W)$ of $E^{\times}$ which, thanks to assumption \eqref{assump}, can be chosen trivial.} \textit{i.e.} there are maps
\begin{equation*}
	\beta_{V, \psi} : \U(V) \longrightarrow \widetilde{\Sp}(\mathcal{W}) \quad \text{,} \quad \beta_{W, \psi} : \U(W) \longrightarrow \widetilde{\Sp}(\mathcal{W})
\end{equation*}
such that
\begin{equation*}
	\omega_{\U(V) \times \U(W)} = (\omega \circ \beta_{V, \psi}) \otimes  (\omega \circ \beta_{W, \psi})
\end{equation*}	
is a genuine representation of $\U(V) \times \U(W)$. The construction of $\beta_{V, \psi}, \beta_{W, \psi}$ can be found in \cite{KudlaSplitting} (see also \cite[§1]{HarrisTheta}).
\subsection{Level preservation}\label{sec2.5}
Let $F$ be a $p$-adic field with $p \neq 2$. Let $E$ be $F$ itself or a quadratic extension. Let $\mathcal{O}$ be the ring of integers of $F$ and  $\varpi$ be a uniformizing parameter of $F$. We assume $E/F$ is unramified, so that $\varpi$ is also a unformizing parameter for $E$.

Let $\psi : F \rightarrow \mathbb{C}^{\times}$ be a non-trivial character with the conductor $c \in \mathbb{Z}$ \textit{i.e.} $\psi_{\vert \mathfrak{p}^{c}} = 1$ and $\psi_{\vert \mathfrak{p}^{c - 1}} \neq 1$. Let $\psi^{-c} : F \rightarrow \mathbb{C}^{\times}$ be the character $\psi^{-c} (\lambda) = \psi ( \varpi^{-c} \lambda)$. Then $\psi^{-c}$ has conductor $\mathcal{O}$ and ${\omega_{\psi} \simeq \omega_{\psi^{-c}}}$ (see \cite[Proposition 1.81]{Thetabook}). Therefore, we can assume $\psi$ has conductor $\mathcal{O}$.

Let $V$ be a finite-dimensional non-degenerate Hermitian $E$-space and let $W$ be a finite-dimensional non-degenerate skew-Hermitian $E$-space. We assume $V$ and $W$ are unramified \textit{i.e.} there are self-dual lattices $L_V \subset V$ and $L_W \subset W$. For any integer $k \geq 0$, we define the principal congruence subgroup of level $k$ as
\begin{equation*}
	K_V(\varpi^{k}) = \lbrace g \in \U(V) : (g - \id) L_V \subseteq \varpi^k L_V \rbrace
\end{equation*}
and
\begin{equation*}
	K_W(\varpi^{k}) = \lbrace g \in \U(W) : (g - \id) L_W \subseteq \varpi^k L_W \rbrace.
\end{equation*}
Let $\mathcal{W}$ be the real non-degenerated symplectic vector space $V \otimes_E W$, then
\begin{equation*}
\mathcal{L} = L_V \otimes_{\mathcal{O}_E} L_W
\end{equation*}
is a self-dual lattice in $\mathcal{W}$ (where $\mathcal{O}_{E}$ is the ring of integers of $E$). Let $\omega = \omega_{\psi}$ be the Weil representation of $\widetilde{\U}(V) \times \widetilde{\U}(W)$ realized on the lattice model $S$ defined with respect to $\mathcal{L}$. This model defines a splitting of the stabilizer $K \subset \Sp(\mathcal{W})$ of $\mathcal{L}$ and we identify $K$ with its image in $\widetilde{\Sp}(\mathcal{W})$.

Let $K_V$ (resp. $K_W$) be the stabilizer of $L_V$ in $\U(V)$ (resp. of $L_W$ in $\U(W)$), then $K_V \times K_W \subset K$ and, using the splitting of $K$, we can identify $K_V$ and $K_W$ with a maximal open compact subgroup of $\widetilde{\U}(V)$ and $\widetilde{\U}(W)$ respectively.

Let $k \geq 1$ be an integer and define the subspace $S_{k} \subset S$ as the space of functions with support contained in $\varpi^{-k/2} \mathcal{L}$ if $k$ is even and $\varpi^{- (k + 1)/2} \mathcal{L}$ if $k$ is odd. Let $\mathcal{H}_W$ be the Hecke algebra of $\U(W)$. Let $k \geq 0$ be an integer. The following Theorem, generalized by Pan \cite[§6]{Pan}, is due to Waldspurger \cite[Chapitre 5, §I.4]{MVW}.
\begin{theorem}  \label{theo2.5.1}
	We have $\omega(\mathcal{H}_W) S_k = S^{K_V (\varpi^k)}$.
\end{theorem}
Let $\pi_W \in \Irr(\U(W))$, let $k \geq 0$ be an integer and consider the $\U(V) \times \U(W)$-projection
\begin{equation*}
	\pi_W \otimes \omega \longrightarrow \Theta(\pi_W ; V) = (\pi_W \otimes \omega)_{\U(V)}
\end{equation*}
onto co-$\U(V)$-invariants. Assume $\Theta(\pi_W ; V) \neq 0$, let $\theta(\pi_W ; V)$ be the maximal semi-simple quotient of $\Theta(\pi_W ; V)$. Composing the projection $\pi_W \otimes \omega \rightarrow \Theta(\pi_W ; V)$ with the projection $\Theta(\pi_W ; V) \rightarrow \theta(\pi_W ; V)$, we obtain the $\U(V) \times \U(W)$-projection
\begin{equation*}
	\pi_W \otimes \omega \longrightarrow \theta(\pi_W ; V).
\end{equation*}
By applying the idempotent $e_{K_V(\mathfrak{p}^k)} = \vol (K_V(\mathfrak{p}^k))^{-1} \mathbf{1}_{K_V(\mathfrak{p}^k)}$ of the Hecke algebra of $\U(V)$ to the above projection, we obtain a surjective map
\begin{equation*}
	\pi_W \otimes \omega^{K_V(\mathfrak{p}^k)} \longrightarrow \theta(\pi_W ; V)^{K_V(\mathfrak{p}^k)}.
\end{equation*}
Applying Theorem \ref{theo2.5.1}, we deduce $\pi_W \otimes \omega(\mathcal{H}_W) S_k \twoheadrightarrow \theta(\pi_W ; V)^{K_V(\mathfrak{p}^k)}$ and, following \cite[§3]{Cossutta}, we deduce a surjective map
\begin{equation*}
	\pi_W^{K_W(\mathfrak{p}^k)} \otimes S_k  \longrightarrow \theta(\pi_W ; V)^{K_V(\mathfrak{p}^k)}.
\end{equation*}
Hence, if $\pi_W \in \Irr(\U(W))$ and if $\theta(\pi_W ; V)^{K_V(\mathfrak{p}^k)} \neq 0$, then $\pi_W^{K_W(\mathfrak{p}^k)} \neq 0$.\\
Conversely, we can reverse the role of $V$ and $W$ to obtain a surjective map
\begin{equation*}
	\pi_V^{K_V(\mathfrak{p}^k)} \otimes S_k  \longrightarrow \theta(\pi_V ; W)^{K_W(\mathfrak{p}^k)}.
\end{equation*}
We may draw the important conclusion that, if $\pi_V \in \Irr(\U(V))$ and $\pi_W \in \Irr(\U(W))$ occur in the local theta correspondence, then $\pi_V^{K_V(\mathfrak{p}^k)} \neq 0$ if and only if $\pi_W^{K_W(\mathfrak{p}^k)} \neq 0$.
\subsection{Archimedean theta correspondence} \label{sec3.7}
We now address the local archimedean theta correspondence. The construction of the Weil representation is analogous to the non-archimedean case and may be found in \cite{HoweTrans}.

Let $d \geq 1$ be an integer and let $(G,G')$ be an irreducible dual reductive pair in the symplectic group $\Sp_{2d}(\mathbb{R})$. Let $\widetilde{\Sp}_{2d}(\mathbb{R})$ be the $\mathbb{S}^1$-cover of $\Sp_{2d}(\mathbb{R})$ and let $\widetilde{G}$ and $\widetilde{G}'$ denote the inverse image of $G$ and $G$' respectively by the covering map ${\widetilde{\Sp}_{2d}(\mathbb{R}) \longrightarrow \Sp_{2d}(\mathbb{R})}$.

Let $\psi : \mathbb{R} \rightarrow \mathbb{C}^{\times}$ be a non-trivial additive character and let $\omega = \omega_{\psi}$ be the Weil representation, associated to $\psi$, of $\widetilde{\Sp}_{2d}(\mathbb{R})$. Let $\omega^{\infty}$ be the smooth representation of $\widetilde{\Sp}_{2d}(\mathbb{R})$ associated to $\omega$. Since $\widetilde{G}$ and $\widetilde{G}'$ commute with one another, we may restrict $\omega^{\infty}$ to $\widetilde{G} \cdot \widetilde{G}'$ to obtain a representation of $\widetilde{G} \times \widetilde{G}'$.

Let $\Irr(\widetilde{G})$ denote set of (infinitesimal classes of) continuous irreducible admissible representations $\pi$ of $\widetilde{G}$ (on locally convex spaces) and $\mathcal{R}(\widetilde{G}) \subset \Irr(\widetilde{G})$ denote the subset of $\pi$'s such that
\begin{equation*}
	\Hom_{\widetilde{G}}(\omega^{\infty} , \pi) \neq 0.
\end{equation*}
For $\pi \in \mathcal{R}(\widetilde{G})$, define
\begin{equation*}
	\omega^{\infty} [\pi] = \bigcap_{\phi \in \Hom_{\widetilde{G}} (\omega^{\infty}, \pi)} \ker (\phi)
\end{equation*}
and set $\omega^{\infty} (\pi) = \omega^{\infty} / \omega^{\infty} [\pi]$. Then $\omega^{\infty} (\pi)$ is a smooth representation of $\widetilde{G} \cdot \widetilde{G}'$ of the form
\begin{equation*}
	\pi \otimes \Theta(\pi)
\end{equation*}
for some smooth representation $\Theta(\pi)$ of $\widetilde{G}'$. Howe \cite{HoweTrans} proved $\Theta(\pi)$ is finitely generated, admissible and quasi-simple. Furthermore, it admits a unique irreducible sub-$\widetilde{G}'$-quotient $\theta(\pi')$ and the correspondence $\pi \mapsto \theta(\pi')$ defines a bijection from $\mathcal{R}(\widetilde{G})$ to $\mathcal{R}(\widetilde{G}')$, where
\begin{equation*}
	\mathcal{R}(\widetilde{G}') = \lbrace \pi' \in \Irr(\widetilde{G}') : \Hom_{\widetilde{G}'} (\omega^{\infty}, \pi') \neq 0 \rbrace.
\end{equation*}
The algebraic version of this result, that is for $(\mathfrak{g}, K)$-modules, also holds (see \cite[Theorem 2.1]{HoweTrans}).
\subsubsection{Embeddings and splittings}
\subsubsection*{Case $E = F = \mathbb{R}$}
Let $n \geq m \geq 1$ be integers. Let $V$ be a non-degenerate orthogonal space of signature $(n,m)$ and let $W$ be a symplectic space of dimension $2m$. The tensor product $\mathcal{W} = V \otimes W$ is naturally a symplectic space by the rule
\begin{equation*}
	\langle v_1 \otimes w_1 , v_2 \otimes w_2 \rangle_{\mathcal{W}} = \langle v_1, v_2 \rangle_V \langle v_2, w_2 \rangle_W.
\end{equation*}
Let $W = X \oplus X^*$ be a complete polarization of $W$, then $\mathcal{W} = \mathcal{X} \oplus \mathcal{Y} = (V \otimes X) \oplus (V \otimes X^*)$ is a complete polarization of $\mathcal{W}$ and we write elements of $\Sp(\mathcal{W})$ as $\langle \cdot , \cdot \rangle_{\mathcal{W}}$-preserving matrices
\begin{equation*}
	G = \begin{pmatrix} A & B \\ C & D \end{pmatrix}
\end{equation*}
for $A \in \End ( \mathcal{X})$, $B \in \Hom( \mathcal{Y}, \mathcal{X})$, $C \in \Hom( \mathcal{X},  \mathcal{Y})$ and $D \in \End( \mathcal{Y})$. Then $\Ortho(V)$ acts on $\mathcal{W}$ by sending $v \otimes w$ to $g \cdot \otimes w$ and this defines an embedding $G \rightarrow \Sp(\mathcal{W})$. The group $\Sp(W)$ embeds in $\Sp(\mathcal{W})$ by the rule
\begin{equation*}
	g = \begin{pmatrix} a & b \\ c & d \end{pmatrix} \mapsto G = \begin{pmatrix} \id_V \otimes a & \id_V \otimes b \\ \id_V \otimes c & \id_V \otimes d \end{pmatrix}
\end{equation*}
for $g \in \Sp(W)$. This way, we obtain an embedding
\begin{equation*}
\iota_V \otimes \iota_W : \Ortho(V) \times \Sp(W) \longrightarrow \Sp(\mathcal{W}) \simeq \Sp_{2(n+m)m}(\mathbb{R})
\end{equation*}
and $\iota_V (\Ortho(V))$, $\iota_W(\Sp(W))$ are mutual commutants in $\Sp(\mathcal{W})$. If $n + m$ is even, then there are maps ${\beta_V} : \Ortho(V) \rightarrow \mathbb{S}^1$, ${\beta_W} : \Sp(W) \rightarrow \mathbb{S}^1$ such that \cite[Corollary 3.2, case 1]{KudlaSplitting}
\begin{equation*}
	\widetilde{\iota_V} \otimes \widetilde{\iota_W} = (\iota_V, \beta_V) \otimes (\iota_W, \beta_W) :  \Ortho(V) \times \Sp(W) \longrightarrow \widetilde{\Sp(W)}
\end{equation*}
lift $\iota_V \otimes \iota_W$.
\subsubsection*{Case $E = \mathbb{C}$, $F = \mathbb{R}$}
Let $n \geq m \geq 1$ be integers. Let $V$ be a non-degenerate Hermitian space of signature $(n,m)$ and let $W$ be a skew-Hermitian space signature $(m,m)$. The tensor product $\mathcal{W} = V \otimes_{\mathbb{C}} W$ is naturally a skew-Hermitian space by the rule
\begin{equation*}
	\langle v_1 \otimes w_1 , v_2 \otimes w_2 \rangle_{\mathcal{W}} = \langle v_1, v_2 \rangle_V \langle v_2, w_2 \rangle_W.
\end{equation*}
Choose bases of $V$, $W$, and $\mathcal{W}$ in such way that $\langle \cdot , \cdot \rangle_V$ is given by the matrix
\begin{equation*}
	I_{n,m} = \begin{pmatrix} \Id_{n} & \\ & -\Id_m \end{pmatrix}
\end{equation*}
$\langle \cdot , \cdot \rangle_W$ is given by the matrix $i I_{m,m}$ (where $i = \sqrt{-1}$) and $\langle \cdot , \cdot \rangle_{\mathcal{W}}$ is given by the matrix $\mathcal{J} = i J$ for some real matrix $J$ such that $J^2 = \Id_{2(n+m)m}$. Then, the map
\begin{equation*}
	\iota_V \otimes \iota_W : \U(V) \times \U(W) \longrightarrow \U(\mathcal{W})
\end{equation*}
given by $\iota_V (g) (v \otimes w) = g \cdot v \otimes w$ and $\iota_W (g') (v \otimes w) = v \otimes g' \cdot w$ is an embedding. Write $G \in \U(\mathcal{W})$ as $G = A + i B$ with $A,B \in \GL_{2(n+m)m}(\mathbb{R})$ and let ${\nu_J : \U(\mathcal{W}) \rightarrow \Sp_{4(n+m)m} (\mathbb{R})}$ be defined by
\begin{equation*}
	\nu_J (G) = \begin{pmatrix} \Id_{2(n+m)m} & \\ & J \end{pmatrix} \begin{pmatrix} A & B \\ -B & A \end{pmatrix} \begin{pmatrix} \Id_{2(n+m)m} & \\ & J \end{pmatrix}.
\end{equation*}
Then
\begin{equation*}
	(\nu_J \circ \iota_V) \otimes (\nu_J \circ \iota_W) : \U(V) \times \U(W) \longrightarrow \Sp (V \otimes_{\mathbb{R}} W) \simeq \Sp_{4(n+m)m}(\mathbb{R})
\end{equation*}
is an embedding and $(\nu_J \circ \iota_V)(\U(V))$, $(\nu_J \circ \iota_W) (\U(W))$ are mutual commutant in ${\Sp (V \otimes_{\mathbb{R}} W)}$. Moreover, there are maps ${\beta_V} : \U(V) \rightarrow \mathbb{S}^1$, ${\beta_W} : \U(W) \rightarrow \mathbb{S}^1$ such that \cite[Corollary 3.2, case 3]{KudlaSplitting}
\begin{equation*}
	\widetilde{\iota_V} \otimes \widetilde{\iota_W} = \big((\nu_J \circ \iota_V), \beta_V \big) \otimes \big( (\nu_J \circ \iota_W), \beta_W \big) :  \U(V) \times \U(W) \longrightarrow \widetilde{\Sp} (V \otimes_{\mathbb{R}} W)
\end{equation*}
lift $(\nu_J \circ \iota_V) \otimes (\nu_J \circ \iota_W)$.
\subsubsection{Tempered spherical transfer}
Let $n > m \geq 1$ be integers such that $n + m$ is even and let $\psi : \mathbb{R} \rightarrow \mathbb{C}^{\times}$ be the non-trivial additive character $\psi(x) = \exp(2\pi i x)$. The tempered spherical transfer has already been studied in \cite[§8.4]{BM2} for the reductive dual pair
\begin{equation*}
(G,G') = (\Ortho(n,m), \Sp_{2m}(\mathbb{R}))
\end{equation*}
in $\Sp_{2(n+m)m}(\mathbb{R})$. We summarize their result in Proposition \ref{prop3.6.1}. Recall $\Irr(G)$ stands for the set of (infinitesimal equivalence classes of) continuous irreducible admissible representations of $G$. For any positive integer $m'$, we let $W_{m'}$ be a symplectic space of dimension $2m'$ and we choose the standard coordinates on $W_{m'}$, so that ${\Sp(W_{m'}) = \Sp_{2m'}(\mathbb{R})}$. For any $\pi \in \Irr(G)$ we let $\theta(\pi ; W_{m'}) \in \Irr(\U(m',m'))$ be the theta lift of $\pi$ to $\Sp(W'_m)$. When $m' = m$, $\Ortho(n,m)$ and $\Sp_{2m}(\mathbb{R})$ both have rank $m$ and we can identify $\mathfrak{a}^*_{\mathbb{C}}$ with ${\mathfrak{a}'}^*_{\mathbb{C}}$.
\begin{proposition} \label{prop3.6.1}
Let $\pi \in \Irr(\Ortho(n,m))$ be tempered, spherical with spectral parameter ${\lambda \in i \mathfrak{a}^*}$. Then $\theta(\pi ; W_{m'}) = 0$ for all $1 \leq m' < m$ and if
\begin{equation*}
\pi' = \theta(\pi ; W_{m}) \neq 0
\end{equation*}
then $\pi' \in \Irr(\Sp_{2m}(\mathbb{R}))$ is $\det^{(n-m)/2}$-spherical with spectral parameter $\lambda$.

If ${\mathbf{1} \in \Irr(\Ortho(n+m))}$ is the trivial representation and if
\begin{equation*}
\pi' = \theta(\mathbf{1} ; W_{m}) \neq 0
\end{equation*}
then $\pi' \in \Irr(\Sp_{2m}(\mathbb{R}))$ is $\det^{(n + m)/2}$-spherical with spectral parameter
\begin{equation*}
\Big( \frac{n + m}{2} - 1, \ldots, \frac{n + m}{2} - m \Big).
\end{equation*}
\end{proposition}
In the sequel we prove the analog of Proposition \ref{prop3.6.1} for the reductive dual pair
\begin{equation*}
	(G,G') = (\U(n,m), \U(m,m))
\end{equation*}
in $\Sp_{4(n+m)m}(\mathbb{R})$. For any positive integer $m'$, we let $W_{m'}$ be a split skew-Hermitian $E$-space of dimension $2m'$ and we choose the standard coordinates on $W_{m'}$, so that ${\U(W_{m'}) = \U(m',m')}$. For any $\pi \in \Irr(\U(n,m))$, we let $\theta(\pi ; W_{m'}) \in \Irr(\U(m',m'))$ be the theta lift of $\pi$. When $m' = m$, $\U(W_m) = \U(m,m)$ so that $\U(n,m)$ and $\U(m,m)$ have same rank $m$; therefore, we can identify $\mathfrak{a}^*_{\mathbb{C}}$ with ${\mathfrak{a}'}^*_{\mathbb{C}}$.
\begin{proposition} \label{prop3.6.2}
Let $\pi \in \Irr(\U(n,m))$ be tempered, spherical with spectral parameter $\lambda \in i \mathfrak{a}^*$. Then $\theta(\pi ; W_{m'}) = 0$ for all $1 \leq m' < m$ and, if
\begin{equation*}
\pi' = \theta(\pi ; W_{m}) \neq 0
\end{equation*}
then $\pi' \in \Irr(\U(m,m))$ is $\det^{(n-m)/2} \otimes \det^{(m-n) / 2}$-spherical with spectral parameter $\lambda$. If ${\mathbf{1} \in \Irr(\U(n+m))}$ is the trivial representation and if
\begin{equation*}
\pi' = \theta(\mathbf{1} ; W_{m}) \neq 0
\end{equation*}
then $\pi' \in \Irr(\U(m,m))$ is $\det^{(n+m)/2} \otimes \det^{-(n+m)/2}$-spherical with spectral parameter
\begin{equation*}
	\Big( \dfrac{3m + n - 1}{2}, \ldots, \frac{n - m + 1}{2}   \Big).
\end{equation*}
\end{proposition}
\begin{proof}
The statement about $K$-types may be directly extracted from \cite[Lemma 1.4.5]{PaulUnitary} which, in turn, quotes the work of Kashiwara and Vergne \cite{KashiwaraVergne}. We refer to Lemma \ref{lem3.3.1} for the construction of such representations and note that $\mathfrak{g}_{\mathbb{C}}$ (resp. $\mathfrak{g}'_{\mathbb{C}}$) is the Lie algebra of $\GL_{n+m}(\mathbb{C})$ (resp. $\GL_{2m'}(\mathbb{C})$). Then, we choose Cartan subalgebras $\mathfrak{h}_{\mathbb{C}}$ and $\mathfrak{h}'_{\mathbb{C}}$ of $\mathfrak{g}_{\mathbb{C}}$ and $\mathfrak{g}'_{\mathbb{C}}$ respectively. We let $e_1, \ldots , e_{n + m}$ be the standard basis of $\mathfrak{h}^*_{\mathbb{C}}$, $e'_1, \ldots , e'_{2m'}$ be the standard basis of ${\mathfrak{h}'}^*_{\mathbb{C}}$ and ${\iota : {\mathfrak{h}'}^*_{\mathbb{C}} \rightarrow  {\mathfrak{h}}^*_{\mathbb{C}}}$ be the natural embedding; namely $\iota(e_i') = e_i$. Following \cite[(1.17)]{Przebinda} we set
\begin{align*}
	\mathfrak{z}_{m'} &= \sum_{i = 2m' + 1}^{n + m} \Big( \frac{2m' + 1 + n + m}{2} - i \Big) e_i \\
			    &= \Big( \underbrace{0 , \ldots, 0}_{2m'} , \frac{n + m - 2m' - 1}{2} , \ldots,  \frac{2m' + 1 - n - m}{2} \Big).
\end{align*}
Recall the Harish-Chandra parameter of an irreducible representation $\pi$ of a compact Lie group $G$ is $\lambda_{\pi} + \rho$, with $\lambda_{\pi}$ the highest weight of $\pi$ and $\rho$ the half-sum of positive roots (see \cite[Chapter VIII, §6]{Knapp}). If $\mathbf{1}$ is the trivial representation of $\U(n+m)$ and if ${\pi' = \theta(\mathbf{1} ; W_m) \neq 0}$, then the Harish-Chandra parameter $\omega_{\pi'}$ of $\pi'$ satisfies the relation \cite[Theorem 1.19]{Przebinda}
\begin{equation*}
	\rho = \iota(\omega_{\pi'}) + \mathfrak{z}_{m}.
\end{equation*}
Hence, the Harish-Chandra parameter of $\pi'$ satisfies
\begin{equation*}
\iota(\omega_{\pi'}) = \sum_{i = 1}^{n + m} \Big( \frac{3m + n + 1}{2} - i \Big) e_i - \sum_{i = 2m + 1}^{n + m} \Big( \frac{3m + n + 1}{2} - i \Big) e_i
\end{equation*}
that is $\omega_{\pi'} =  \sum_{i = 1}^{2m} \big( \frac{3m + n + 1}{2} - i \big) e_i$. We now quote \cite[Theorem 1.1]{ZhuRep} to ensure the infinitesimal class of $\pi$' is uniquely determined. In particular, since $\U(m,m)$ is split, the spectral parameter of $\pi'$ is (see Corollary \ref{coro3.3.2})
\begin{equation*}
	\Big( \dfrac{3m + n - 1}{2}, \ldots, \frac{n - m + 1}{2}   \Big).
\end{equation*}
Let $\pi \in \Irr(\U(n,m))$ with spectral parameter $\lambda \in \mathfrak{a}^*_{\mathbb{C}}$, let $1 \leq m' \leq m'$ and assume $\pi' \overset{\defin}{=} \theta(\pi ; W_{m'}) \neq 0$. Let $\omega_{\pi}$ and $\omega_{\pi'}$ denote the Harish-Chandra parameters of $\pi$ and $\pi'$ respectively, then (see \cite[Theorem 1.19]{Przebinda})
\begin{equation*}
	\omega_{\pi} = \iota (\omega_{\pi'}) + \mathfrak{z}_{m'}.
\end{equation*}
By Lemma \ref{lem3.3.1}, we have $\omega_{\pi} = \rho_{\mathfrak{m}} + \lambda$. In particular
\begin{align*}
	\lambda    &= \iota (\omega_{\pi'}) + \mathfrak{z}_{m'} - \rho_{\mathfrak{m}} \\
			&= \Big( \iota (\omega_{\pi'})_1, \ldots, \iota (\omega_{\pi'})_{2m'} , \frac{n + 3m - 4m' - 1}{2}, \ldots, \frac{n - m + 1}{2}, 0, \ldots, 0 \Big).
\end{align*}
Thus, if $1 \leq m' < m$ and $\theta(\pi ; W_{m'}) \neq 0$, then $\lambda \notin i \mathfrak{a}^*$. By contrapositive this proves that, if $\pi$ is tempered, then its theta lift to $W_{m'}$ is zero for all $m' < m$. If $\pi \in \Irr(\U(n,m))$ is tempered \textit{i.e.} has spectral parameter $\lambda \in i \mathfrak{a}^*$ and if $\pi' = \theta(\pi ; W_m) \neq 0$, then the Harish-Chandra parameter $\omega_{\pi'}$ of $\pi'$ satisfies (see \cite[Theorem 1.19]{Przebinda})
\begin{equation*}
	\omega_{\pi} = \iota (\omega_{\pi'}) + \mathfrak{z}_{m}
\end{equation*}
with $\omega_{\pi} = \lambda + \rho_{\mathfrak{m}}$ the Harish-Chandra parameter of $\pi$. Since $\rho_{\mathfrak{m}} =  \mathfrak{z}_{m}$, we deduce $\pi'$ has Harish-Chandra parameter $\lambda$ and, using \cite[Theorem 1.1]{ZhuRep} and the fact $\U(m,m)$ is split, this proves $\pi'$ has spectral parameter $\lambda$.
\end{proof}
\subsection{Global theta correspondence} \label{sec4.1}
We now turn to the global theta correspondence.

Let $F$ be a number field and, for any place $v$ of $F$, let $F_v$ be its completion. Let $\mathbb{A}$ be the ring of adèles of $F$ and let $E$ be $F$ itself or a totally imaginary quadratic extension of $F$. Let $V$ be a non-degenerate Hermitian $E$-space and $W$ be a non-degenerate skew-Hermitian $E$-space. Set $G = \U(V)$ and $G' = \U(W)$ for the associated $F$-groups.

Let $\psi = \otimes_v \psi_v$ be a non-trivial additive character of $F \backslash \mathbb{A}$ and consider ${\mathcal{W} = V \otimes W}$ as a symplectic $F$-space. For each place $v$, we have a metaplectic cover
\begin{equation*}
	1 \longrightarrow \mathbb{S}^{1} \longrightarrow \widetilde{\Sp}(\mathcal{W}_v)_{\psi_v} \longrightarrow \Sp (\mathcal{W}_v) \longrightarrow 1.
\end{equation*}
For any finite place $v$, let $\mathcal{K}_v \subset \Sp (\mathcal{W}_v)$ be the stabilizer of a self-dual lattice. Then the covering splits over $\mathcal{K}_v$ and we obtain an open compact subgroup $\mathcal{K}_v$ of $\widetilde{\Sp}(\mathcal{W}_v)$. We may then form the (restricted with respect to the $\mathcal{K}_v$'s) direct product
\begin{equation*}
	\widetilde{\Sp}(\mathcal{W})(\mathbb{A}) = \prod_v \widetilde{\Sp}(\mathcal{W}_v).
\end{equation*}
This define a cover
\begin{equation*}
	1 \longrightarrow \mathbb{S}^{1} \longrightarrow \widetilde{\Sp}(\mathcal{W}) (\mathbb{A}) \longrightarrow \Sp (\mathcal{W})(\mathbb{A}) \longrightarrow 1
\end{equation*}
called the \textit{adelic metaplectic group}. Note that we are not considering the adelic points of an algebraic group. For each place $v$, we have a Weil representation $\omega_{\psi_v, v}$ acting on a space $S_{\psi_v, v}$ and we define the global Weil representation
of $\widetilde{\Sp}(\mathcal{W}) (\mathbb{A})$ as the (restricted) tensor product
\begin{equation*}
	\omega_{\psi} = \bigotimes_v \omega_{\psi_v,v}
\end{equation*}
which acts on the (restricted) tensor product
\begin{equation*}
	S_{\psi}(\mathbb{A}) = \bigotimes_v S_{\psi_v,v}.
\end{equation*}
We fix a complete polarization $\mathcal{W} = \mathcal{X} \oplus \mathcal{Y}$ of $\mathcal{W}$ and we consider the Schrödinger model of the Weil representation on $\mathcal{Y}(F_v) = \mathcal{Y} \otimes F_v$ at every places.

The natural inclusion $\iota :  \Sp (\mathcal{W})(F) \rightarrow \Sp (\mathcal{W})(\mathbb{A})$ lifts to a diagonal embedding \cite[§40]{Weil}
\begin{equation*}
\widetilde{\iota} : \Sp (\mathcal{W})(F) \longrightarrow \widetilde{\Sp} (\mathcal{W})(\mathbb{A})
\end{equation*}
and we identify $\Sp (\mathcal{W})(F)$ with its image in $\widetilde{\Sp}(\mathcal{W})(\mathbb{A})$. Moreover Weil showed that the distribution
\begin{equation*}
	s \mapsto \sum_{y \in \mathcal{Y}(F)} s(y) \quad (s \in S_{\psi}(\mathbb{A}))
\end{equation*}
obtained by averaging $s$ over the $F$-rational points of $\mathcal{Y}$ is $\Sp (\mathcal{W})(F)$-invariant and that, for any fixed $s \in S_{\psi}(\mathbb{A})$, the function
\begin{equation*}
	\Theta (g; s) = \sum_{y \in \mathcal{Y}(F)} \omega_{\psi}(g) s(y)
\end{equation*}
has moderate growth on $\widetilde{\Sp}(\mathcal{W}) (\mathbb{A})$.
\subsubsection{Global theta lift} \label{subsec4.2}
Recall from Section \ref{sec2.4} the local splittings
\begin{equation*}
	{G}(F_v) \longrightarrow \widetilde{\Sp}(\mathcal{W}) (F_v) \quad , \quad {G}'(F_v) \longrightarrow \widetilde{\Sp}(\mathcal{W}) (F_v).
\end{equation*}
Considering the restricted tensor product of these local splittings, we deduce global splittings
\begin{equation*}
	{G}(\mathbb{A}) \longrightarrow \widetilde{\Sp}(\mathcal{W}) (\mathbb{A}) \quad , \quad {G}'(\mathbb{A}) \longrightarrow \widetilde{\Sp}(\mathcal{W}) (\mathbb{A})
\end{equation*}
and, given $s \in S_{\psi}(\mathbb{A})$, we set
\begin{equation*}
\Theta(g,g' ; s) =  \sum_{y \in \mathcal{Y}(F)} \omega_{\psi}(g, g') s(y) \quad (g \in G(\mathbb{A}), g' \in G'(\mathbb{A})).
\end{equation*}
This way, we can define the global Weil representation of ${G}(\mathbb{A}) \times {G}'(\mathbb{A})$. Let $f$ be a cuspidal automorphic form on $G$, $s \in S(\mathbb{A})$ and define the automorphic form $\theta(f ; s)$ on $G'$ by
\begin{equation*}
	\theta(f ; s) (g') = \int_{[G]} f(g) \Theta(g,g' ; s) dg \quad (g' \in G'(\mathbb{A})).
\end{equation*} 
Now, let $\pi$ be a cuspidal automorphic representation of $G$ and let $\Theta (\pi ; W)$ be the space generated by the $\theta(f ; \alpha)$'s for $f \in \pi$ and ${s \in S(\mathbb{A})}$. We call $\Theta (\pi ; W)$ the global theta lift of $\pi$.

Conversely, if $f'$ is a cuspidal automorphic form on $G$ and if $s \in S(\mathbb{A})$, we can define the automorphic form
\begin{equation*}
\theta(f' ; \overline{s})(g) = \int_{[G']} f(g) \overline{\Theta(g,g'; \alpha)} dg' \quad (g \in G(\mathbb{A}))
\end{equation*}
on $G$. Assuming both $\theta(f; s)$ and $\theta(f ; \overline{s})$ are cuspidal (to ensure every integral is absolutely convergent) we have the adjoint property
\begin{equation} \label{eqadj}
\langle f, \theta(f' ; \overline{s}) \rangle_G = \langle \theta( f; s) , f' \rangle_{G'}.
\end{equation}
\subsubsection{Representation theoretic properties} \label{sec4.3}
We now take note of two important results in the global theta correspondence and state a straightforward Corollary. The first one, due to Kudla and Rallis \cite{KudlaRallis}, provides conditions for the compatibility between local and global lifts and the second one, due to Rallis \cite{RallisOn}, gives the cuspidality of the first occurence in the global theta correspondence.
\begin{theorem} \label{theo6.4.1}
If $\pi$ is a cuspidal irreducible automorphic representation of $G$ and if $\Theta(\pi ; W)$ is square-integrable, then $\Theta (\pi ; W)$ is irreducible and
\begin{equation*}
	\Theta (\pi ; W) \simeq \bigotimes_v \theta (\pi_v ; W).
\end{equation*}
\end{theorem}
For $m > 0$ an integer, we denote by $W_m$ the split skew-Hermitian space of dimension $2m$.
\begin{theorem} \label{theo6.4.2}
Let $m \geq 2$. If $\pi$ is a cuspidal irreducible automorphic representation of $G$ and if $\Theta(\pi ; W_{m - 1}) = 0$, then $\Theta(\pi ; W_{m})$ (if non-zero) is cuspidal.
\end{theorem}
\begin{rmrk}
The articles \cite{KudlaRallis} and \cite{RallisOn} deal with orthogonal-symplectic pairs. However, both Theorems \ref{theo6.4.1} and \ref{theo6.4.2} hold for unitary pairs (see \textit{e.g.} \cite[§3]{GanAutoTheta}).\\
Note that, in these statements, one can reverse the roles of $G$ and $G'$ (upon replacing $m$ by the Witt index of $V$).
\end{rmrk}
\begin{corollary} \label{coro6.4.3}
Let $n \geq m \geq 1$ be integers such that $n+m$ is even. Assume the dimension of $V$ is $n+m$ and $W$ is split of dimension $2m$. Let $v_0$ be a fixed archimedean place of $F$ and assume $V$ has signature $(n,m)$ at $v_0$. Let $\pi$ be a cuspidal irreducible automorphic representation of $G$. If the local component of $\pi$ at $v_0$ is spherical and tempered, then $\Theta(\pi ; W)$ (if non-zero) is cuspidal and $\Theta(\pi ; W) \simeq \otimes_v \theta(\pi_v ; W_{v})$.
\end{corollary}
\begin{proof}
We extract this from \cite[Lemma 8.10]{BM2}. Let $\pi_{v_0}$ be the local component of $\pi$ at $v_0$. Since $\pi_{v_0}$ is spherical and tempered, the local theta lift of $\pi_{v_0}$ to $\U(W_{m'})_{v_0}$ vanishes for $m' < m$. This follows from Proposition \ref{prop3.6.1} if $E = F$ or Proposition \ref{prop3.6.2} if $E \neq F$. Therefore, the global theta lift $\Theta(\pi ; W_{m'})$ vanishes for $m' < m$. Since $W = W_m$, Theorem \ref{theo6.4.2} ensures $\Theta(\pi ; W)$ is cuspidal (hence square-integrable). The factorization statement now follows from Theorem \ref{theo6.4.1}.
\end{proof}

\section{Global Period relation} \label{chap5}
In this Section, we state a relation between periods and explain how the choice of a test function yield a distinction argument. This Section is analogous to \cite[§9]{BM2} which generalizes the period relation in Rudnick and Sarnak \cite[(3.24)]{RudnickSarnak}. However, we shall use a different test function to that of \cite{BM2}. We begin with an overview of these relations.

Let $F$ be a number field and $E$ be $F$ itself or a totally imaginary quadratic extension of $F$. Let $V$ be a non-degenerate anisotropic Hermitian $E$-space, $W$ be a non-degenerate split skew-Hermitian $E$-space and $G = \U(V)$, $G' = \U(W)$ be the associated $F$-groups. We assume both $\dim_E V$ and $\dim_E W$ are even. As $W$ is split, there is a maximal isotropic subspace $X \subset W$ of dimension $\frac{1}{2} \dim_E W$ such that
\begin{equation*}
	W = X \oplus X^*
\end{equation*}	
for any isotropic space $X^*$ which is in perfect duality with $X$. We let $P_X$ denote the stabilizer in $\U(W)$ of the flag $\lbrace 0 \rbrace \subset X \subset W$. Then $P_X$ is a maximal (Siegel) parabolic subgroup and we fix a Levi decomposition
\begin{equation*}
	P_X = \GL(X) \cdot N_X
\end{equation*}
with $N_X \subset \Hom(X^*,X)$ the unipotent radical of $P_X$. Note that $N_X$ can be identified with the set of Hermitian forms on $X^*$ and its Pontrjagin dual can be identified with $N_{X^*}$, the unipotent radical of the stabilizer $P_{X^*}$ of the flag $\lbrace 0 \rbrace \subset X^* \subset W$. We then identify ${N_{X^*} \subset \Hom(X,X^*)}$ with Hermitian forms on $X$. Under these identifications, the Hermitian form $q$ on $X$ corresponds to an element $n_q^* \in N_{X^*}$ which, in turn, corresponds to the character
\begin{equation*}
	\psi_q (n) = \psi \big( \tr_{X} (n \circ n_q^*) \big)
\end{equation*}
of $N_X(F) \backslash N_X(\mathbb{A})$.

Let $q$ be a non-degenerate Hermitian form on $X$, let $\U(X) \subset \GL(X)$ be the associated unitary group, assume there is a non-zero embedding $j : X \hookrightarrow V$ of Hermitian spaces and define the subgroup
\begin{equation*}
	H_j^{\perp} =  \U(j(X)^{\perp})
\end{equation*}
of $G$. The adelic quotient $[N_X]$ is compact and, recalling $V$ is anisotropic, the adelic quotients $[G]$ and $[H_j]$ also compacts. Let $dh$ and $dn$ denote the $H_j^{\perp}(\mathbb{A})$-probability measure on $[H_j^{\perp}]$ and the $N_X(\mathbb{A})$-probability measure on $[N_X]$ respectively. We define the $H_j^{\perp}$-period of $f \in C^{\infty}([G])$ by
\begin{equation*}
	\mathcal{P}_{H_j^{\perp}}(f) = \int_{[H_j^{\perp}]} f(h) dh
\end{equation*}
and the $(N_X, \psi_q)$-period (also called the $\psi_q$-Fourier-Whittaker coefficient) of $f' \in C^{\infty}([G'])$ by
\begin{equation*}
	\mathcal{P}^{\psi_q}_{N_X}(f') = \int_{[N_X]} f'(n) \psi_q^{-1}(n) dn.
\end{equation*}
Note that, if $\pi$ and $\pi$' are irreducible automorphic representation of $G$ and $G$' respectively, then
\begin{itemize}
	\item $f \mapsto \mathcal{P}_{H_j^{\perp}}(f)$ defines an element in $\Hom_{H_j^{\perp}(\mathbb{A})} (\pi, 1)$;
	\item $f' \mapsto \mathcal{P}_{N_X}^{\psi_q}(f')$ defines an element in $\Hom_{N_X(\mathbb{A})} (\pi', \psi_q)$.
\end{itemize}
The global Schrödinger model of the global Weil representation $\omega$ with respect to the complete polarization $\mathcal{W} = \mathcal{X} \oplus \mathcal{Y}$ of $\mathcal{W} = V \otimes W$ is the space of Bruhat-Schwartz functions on $\mathcal{Y}(\mathbb{A})$. We may identity $\mathcal{Y} = X^* \otimes V$ with $\Hom(X,V)$ so that, for any $s \in S(\mathbb{A})$ and any $y \in \Hom(X,V)(\mathbb{A})$, we have
\begin{itemize}
	\item $\omega(g, 1) s(y) = s(g^{-1} \circ y) \quad (g \in G(\mathbb{A}))$;
	\item $\omega(1, n) s(y) = \psi_{y}(n) s(y) \quad (n \in N(\mathbb{A}))$;
\end{itemize}
where $\psi_y$ is the character of $N_X(F) \backslash N_X(\mathbb{A})$ corresponding to the Hermitian form on $X$ obtained by pulling back the Hermitian form on $V$ using $y$.

One can prove (see \textit{e.g.} \cite{WallsArticle})
\begin{equation} \label{eq5.1}
\mathcal{P}^{\psi_q}_{N_X}(\theta (f ; s)) =  \int_{H_j^{\perp}(\mathbb{A}) \backslash G(\mathbb{A})} \omega(x) s(y) \mathcal{P}_{H_j^{\perp}}(R(x){f}) dx
\end{equation}
for any $s \in \omega$. This mainly follows from the computation
\begin{equation} \label{eq5.2}
	\int_{[N_X]} \psi_{q}^{-1}(n) \sum_{y \in \mathcal{Y}(F)} \omega(n) s(y) dn
	= \sum_{y \in \mathcal{Y}_q(F)} s(y) dn = \sum_{x \in H_j^{\perp}(F) \backslash G(F)} s(x \circ j) dn
\end{equation}
of the integral of the theta kernel which, in turn, follows from the identification
\begin{equation*}
H_j^{\perp}(F) \backslash G(F) \simeq \mathcal{Y}_q(F)
\end{equation*}
given by Witt's Theorem.

In the local setting, \eqref{eq5.1} may be seen as a global analog of the isomorphism \cite{GanPeriods}\footnote{We implicitly choosed $\chi_V = \chi_W = \mathbf{1}$ thanks to our assumption on the dimensions of $V$ and $W$.}
\begin{equation*}
	\Hom_{N_X} (\Theta(\pi), \psi_q) \simeq \Hom_{H_j^{\perp}} (\pi^{\vee}, 1)
\end{equation*}
where $\Theta(\pi) = (\omega \otimes \pi)_G$ is the co-$G$-invariants module of $\omega \otimes \pi$ for a given $\pi \in \Irr(G)$ (with $\omega$ the local Weil representation) and $\pi^{\vee}$ is the contragredient of $\pi$. The proof essentially relies on the calculation of the $\psi_q$-twisted Jacquet module of the local Weil representation: one can show that it is isomorphic to the induced (to $G$) representation of the trivial representation of $H_j^{\perp}$. The global analog of this local calculation is \eqref{eq5.2}. Moreover, by keeping track of the action of the diagonally embedded subgroup
\begin{equation*}
	\U(X,q) \longrightarrow \U(X,q) \times \U(j(X)) \longrightarrow \U(X,q) \times \U(V)
\end{equation*}
one can prove \cite{GanPeriods}
\begin{equation} \label{eq5.3}
	\Hom_{R_X} (\Theta(\pi), 1 \times \psi_q) \simeq \Hom_{H_j} (\pi^{\vee}, 1)
\end{equation}
where $R_X = \U(X) \rtimes N_X$, $H_j = \U(j(X)^{\perp}) \times \U(j(X))$ and $1 \times \psi_q$ is the character of $R_X$ such that $(1 \times \psi_q)(r) = \psi_q(n)$ for all $r = (t,n) \in R_X$.

The period relation we prove in the next Section is the global analog of \eqref{eq5.3}.\\
In \cite{Jun} the author generalizes these global period relations by considering various partitions of nilpotent orbits in the orthogonoal-symplectic case and this can be seen as a global analog to the local results of Gomez and Zhu \cite{GomezZhu}. In a wider context, analogous period relations have been used to prove Gan-Gross-Prasad conjecture \cite{BeuzartChaudouard}.
\subsection{Global period relation} \label{sec5.1}
We use the notations and hypothesis previously adopted for $V$, $W$, $X$, $G$ and $G'$. For $q$ a non-degenerate Hermitian form on $X$, we denote by $\U(X) \subset \GL(X)$ the associated unitary group and set
\begin{equation*}
R_q = \U(X) \ltimes N_X.
\end{equation*}
It may be seen that $\psi_q$, being invariant under conjugation by $\U(X)$, extends to a character $\mathbf{1} \times \psi_q$ of $R_q(F) \backslash R_q(\mathbb{A})$ such that
\begin{equation*}
	\mathbf{1} \times \psi_q (r) = \psi_q (n) \quad (r = tn \in R(\mathbb{A})).
\end{equation*}
Assuming moreover that there is a non-zero embedding $j : X \hookrightarrow V$ of Hermitian spaces, we define the subgroup
\begin{equation*}
H_j = \U(j(X)^{\perp}) \times \U(j(X))
\end{equation*}
of $G$. The adelic quotient $[N_X]$ is compact and, recalling $V$ is anisotropic, the adelic quotients $[G]$, $[H_j]$ and $[R_q]$ are also compacts. Let $dh$, $dt$ and $dn$ denote the $H_j(\mathbb{A})$-probability measure on $[H_j]$, the $\U(X)(\mathbb{A})$-probability measure on $[\U(X)]$ and the $N_X(\mathbb{A})$-probability measure on $[N_X]$ respectively. We define the $H_j$-period of $f \in C^{\infty}([G])$ by
\begin{equation*}
	\mathcal{P}_{H_j}(f) = \int_{[H_j]} f(h) dh
\end{equation*}
and the $(R_q, \mathbf{1} \times \psi_q)$-period (also called the Shalika or Bessel period) of $f' \in C^{\infty}([G'])$ by
\begin{equation*}
	\mathcal{P}^{\mathbf{1} \times \psi_q}_{R_q}(f') = \int_{[\U(X)]} \int_{[N_X]} f'(tn) \psi_q^{-1}(n) dn dt.
\end{equation*}
Note that, if $\pi$ and $\pi$' are irreducible automorphic representation of $G$ and $G$' respectively, then
\begin{itemize}
	\item $f \mapsto \mathcal{P}_{H_j}(f)$ defines an element in $\Hom_{H_j(\mathbb{A})} (\pi, 1)$;
	\item $f' \mapsto \mathcal{P}_{R_q}^{\mathbf{1} \times \psi_q}(f')$ defines an element in $\Hom_{R_q(\mathbb{A})} (\pi', \psi_q)$.
\end{itemize}
The following Proposition is well known (see \textit{e.g.} \cite[§9]{BM2} or \cite[§8]{BeuzartChaudouard}) but we prove it for the convenience of the reader. Recall $\omega$ denotes the global Weil representation.
\begin{proposition} \label{prop5.1.1}
Let $f$ be an automorphic form of $G$ and $s \in \omega$, then
\begin{equation*}
\mathcal{P}^{\mathbf{1} \times \psi_q}_{R_q}(\theta (f ; s)) =  \int_{H_j (\mathbb{A}) \backslash G(\mathbb{A})}\mathcal{I}_{j,s}(x) \text{ } \mathcal{P}_{H_j}(R(x){f}) dx
\end{equation*}
where $R(x)f : g \mapsto f(gx)$ and
\begin{equation*}
	\mathcal{I}_{j,s}(x) = \int_{\U(X)(\mathbb{A})} \omega(x, t) s (j) dt.
\end{equation*}
\end{proposition}
\begin{proof}
We consider the complete polarisation
\begin{equation*}
	\mathcal{W} = (X \otimes V) \oplus (X^* \otimes V)
\end{equation*}
of the $F$-symplectic space $\mathcal{W} = W \otimes V$ and we consider the Schrödinger model of the Weil representation. In particular $\omega$ acts on $S(\mathbb{A})$, the space of Bruhat-Schwartz functions on $\mathcal{Y}(\mathbb{A})$. We may identify $\mathcal{Y} = X^* \otimes V$ with $\Hom(X, V)$ so that, for any $s \in S(\mathbb{A})$ and $y \in \Hom (X,V)(\mathbb{A})$ we have
\begin{itemize}
	\item $\omega(g, 1) s(y) = s(g^{-1} \circ y)$, for all $g \in G(\mathbb{A})$
	\item $\omega(1,t) s(y) = s(y \circ t^{-1})$, for all $t \in \U(X,q)(\mathbb{A})$
	\item $\omega(1, n) s(y) = \psi_{y}(n) s(y)$, for all $n \in N(\mathbb{A})$
\end{itemize}
where $\psi_y$ is the character of $N_X(F) \backslash N_X(\mathbb{A})$ corresponding to the Hermitian form on $X$ obtained by pulling back the Hermitian form on $V$ using $y$. By definition, we have
\begin{equation*}
	\mathcal{P}^{\mathbf{1} \times \psi_q}_{R_q}(\theta (f ; s)) = \int_{[\U(X)]} \int_{[N_X]} \psi_q^{-1} (n) \int_{[G]} {f(g)} \sum_{y \in \mathcal{Y}(F)} \omega(g, tn) s(y) dg dn dt
\end{equation*}
with
\begin{equation*}
	\int_{[N_X]} \sum_{y \in \mathcal{Y}(F)} \omega(g, tn) s(y)  \psi_q (n)  dn = \sum_{y \in \mathcal{Y}(F)} \omega(g,t) s(y) \int_{[N_X]} (\psi_y - \psi_q)(n) dn.
\end{equation*}
The integral
\begin{equation*}
	\int_{[N_X]} (\psi_y - \psi_q)(n) dn
\end{equation*}
is zero unless $\psi_y = \psi_q$ in which case it is equal to $1$. By definition, $\psi_y = \psi_q$ if and only if $y$ is an embedding of Hermitian spaces $X \hookrightarrow V$. Let $Q$ denote the Hermitian form on $V$ and define
\begin{equation*}
	\mathcal{Y}_q(F) = \lbrace y \in \Hom_F(X,V) : Q \circ y = q \rbrace.
\end{equation*}
Note that we implicitly assumed $\mathcal{Y}_q(F) \neq 0$ by assuming there is a non-trivial embedding of Hermitian spaces $j : X \hookrightarrow V$ (the one used to define $H_j$). Thus, we have
\begin{equation*}
	\int_{[N_X]} \sum_{y \in \mathcal{Y}(F)} \omega(g, tn) s(y)  \psi_q (n)  dn = \omega(g,t) \sum_{y \in \mathcal{Y}_q(F)}  s(y)
\end{equation*}
and if we set
\begin{equation*}
	Y_q(F) = \mathcal{Y}_q(F) / \U(X)(F)
\end{equation*}
with projection map $y \in \mathcal{Y}_q(F) \mapsto [y] \in Y_q(F)$, then we obtain
\begin{equation*}
\mathcal{P}^{\mathbf{1} \times \psi_q}_{R_q}(\theta (f ; s)) = \int_{[G]} {f(g)} \int_{[\U(X)]} \sum_{[y] \in Y_q(F)} \sum_{\gamma \in \U(X)(F)} \omega(g, t \gamma) s(y) dt dg.
\end{equation*}
By unfolding the integral over $\U(X)$, we find
\begin{align*}
\mathcal{P}^{\mathbf{1} \times \psi_q}_{R_q}(\theta (f ; s)) &= \int_{[G]} {f(g)} \int_{\U(X)(\mathbb{A})} \sum_{[y] \in Y_q(F)} \omega(g, t) s(y) dt dg \\
&= \int_{[G]} {f(g)} \sum_{y \in Y_q(F)} \omega(g, 1) \int_{\U(X)(\mathbb{A})} \omega(1, t) s(y) dt dg.
\end{align*}
By Witt's Theorem, the map $g \mapsto g \circ j$ induces an isomorphism
\begin{equation*}
	\U(j(X)^{\perp}) \backslash G(F) \simeq \mathcal{Y}_q(F)
\end{equation*}
and we can thus identify
\begin{equation*}
	H_j(F) \backslash G(F) \simeq Y_q(F) .
\end{equation*}
By unfoldings, this yields
\begin{align*}
\mathcal{P}^{\mathbf{1} \times \psi_q}_{R_q}(\theta (f ; s))
&=  \int_{[G]} {f(g)} \sum_{x \in H_j (F) \backslash G(F)} \omega(g, 1) \int_{\U(X)(\mathbb{A})} \omega(1, t) s(x \circ j) dt dg \\
&=  \int_{H_j(F) \backslash G(\mathbb{A})} {f(g)} \mathcal{I}_{j,s}(g) dg \\
&=  \int_{H_j (\mathbb{A}) \backslash G(\mathbb{A})}  \int_{H_j(F) \backslash H_j(\mathbb{A})} {f(h x)} \mathcal{I}_{j,s}(hx)dh dx.
\end{align*}
The function $g \mapsto \mathcal{I}_{j,s}(g)$ is left-$H_j(\mathbb{A})$-invariant: by definition of $j$, this function is left-$\U(j(X)^{\perp})(\mathbb{A})$-invariant and the averaging over $\U(X)(\mathbb{A})$ yields the left-$\U(j(X))(\mathbb{A})$-invariance because every $t_{j(U)}^{-1} \in \U(j(X))$ can be written as $j \circ t_U^{-1} \circ j^{-1}$ for some unique $t_U^{-1} \in \U(X,q)$. This way, we have
\begin{equation*}
\mathcal{P}^{\mathbf{1} \times \psi_q}_{R_q}(\theta (f ; s)) = \int_{H_j (\mathbb{A}) \backslash G(\mathbb{A})}  \mathcal{I}_{j,s}(x) \int_{H_j(F) \backslash H_j(\mathbb{A})} {f(h x)}dh dx.
\end{equation*}
and this proves the Proposition.
\end{proof}
\subsection{Test function} \label{ssec5.2}
In this section, for a given place $v$ of $F$, we construct a test function $s_v \in \omega_v$ to be used in the global period relation established earlier.
\subsubsection{$p$-adic test function}
We begin by recalling the setup at a given finite place. Let $F$ be a $p$-adic field. Let $E$ be $F$ itself or a quadratic extension. Let $V$ be a finite-dimensional non-degenerate Hermitian $E$-space and let $W$ be a finite-dimensional non-degenerate split and skew-Hermitian $E$-space. As $W$ is split, there is a maximal isotropic subspace $X \subset W$ of dimension $\frac{1}{2} \dim_E W$ such that $W = X \oplus X^*$ for any isotropic space $X^*$ which is in perfect duality with $X$. The tensor product $\mathcal{W} = V \otimes W$ may then be written as the direct sum
\begin{equation*}
\mathcal{W} = (V \otimes X) \oplus (V \otimes X^*).
\end{equation*}
Upon identifying ${V \otimes X^* \simeq \Hom_F (X, V)}$, the Schrödinger model realizes the Weil representation $\omega$ on the space of compactly supported smooth functions ${\Hom_F (X, V)}$\\ $\rightarrow \mathbb{C}$. We endow $X$ with a non-degenerate Hermitian form $q$, we denote by $Q$ the form on $V$ such that $G = \U(Q)$ and we consider the closed sub-variety
\begin{equation*}
	\mathcal{Y} = \lbrace y \in \Hom_F (X, V) : Q \circ y = q \rbrace 
\end{equation*}
of Hermitian spaces embeddings.

We assume $\mathcal{Y}$ is not empty and we fix $j \in \mathcal{Y}$. By Witt's Theorem, $\mathcal{Y}$ is a homogeneous space under the action of $\U(V) \times \U(X)$ \textit{i.e.} ${\U(V) \times \U(X) \cdot j = \mathcal{Y}}$. Note that the action of $(g, t) \in \U(V) \times \U(X)$ on $s \in \omega$ is given by
\begin{equation*}
\omega(g,t)s(y) = s(g^{-1} \circ y \circ t) \quad (y \in \Hom_F(X,V)).
\end{equation*}
\begin{lemma} \label{lem5.2.1}
Let $K_V \subset \U(V)$ and $K_X \subset \U(X)$ be open compact subgroups. There is a compactly supported smooth function $s : \Hom_{F}(X,V) \rightarrow \mathbb{C}$ such that
\begin{equation*}
(g,t) \in \U(V) \times \U(X) \mapsto s(g^{-1} \circ j \circ t)
\end{equation*}
is the indicator function of $K_V \times K_X$.
\end{lemma}
\begin{proof}
Since $K_V \subset \U(V)$ and $K_X \subset \U(X)$ are open subsets, $K_V \times K_X \cdot j$ is of the form $\mathcal{U} \cap \mathcal{Y}$ for some open subset $\mathcal{U} \subset \Hom_F (X,V)$. As the open compact subsets of $\Hom_F(X,V)$ form a basis for its topology, we can write $\mathcal{U}$ as a union of compact open subset of $\Hom_F(X,V)$. Moreover, since $K_V \subset \U(V)$ and $K_X \subset \U(X)$ are compact subsets, finitely many of these open compact sets suffices to cover $K_V \times K_X \cdot j$. Their (finite) union is an open compact subset of $\Hom_F(X,V)$ which intersects $\mathcal{Y}$ in $K_V \times K_X \cdot j$ and we define $s$ to be its indicator function.
\end{proof}
Let $H_j = \U(j(X)) \times \U(j(X)^{\perp}) \subset \U(V)$.
\begin{lemma} \label{lem5.2.2}
Let $K_V \subset \U(V)$ be an open compact subgroup. There is a compactly supported smooth function $\mathbf{s} : \Hom_F(X,V) \rightarrow \mathbb{C}$ such that
\begin{equation*}
	\int_{\U(X)} \mathbf{s}( g^{-1} \circ j \circ t) dt = \mathbf{1}_{H_j \cdot K_V}(g) \quad (g \in \U(V))
\end{equation*}
\end{lemma}
\begin{proof}
Let $K_X$ denote the image of $K_V \cap \U(j(X))$ under the isomorphism ${\U(j(X)) \rightarrow \U(X)}$ given by $t \mapsto j^{-1} \circ t \circ j$. Consider the test function $s$ given (with respect to $K_V$ and $K_X$) by Lemma \ref{lem5.2.1} and set $\mathbf{s} = \vol(K_X)^{-1} s$. Let $h = h_{j(U)} h_{j(U)^{\perp}} \in H_j$ and let $k \in K_v$, then
\begin{align*}
\int_{\U(X)} \mathbf{s}((hk)^{-1} \circ j \circ t) dt &= \int_{\U(X)} \mathbf{s}( k^{-1} \circ j \circ t_{j(U)} \circ t) dt \\
&= \int_{\U(X)} \mathbf{s}( k^{-1} \circ  j \circ t) dt 
\end{align*}
where $t_{j(U)} = j^{-1} \circ h_{j(U)} \circ j$. Since $k^{-1} \in K_V$, we deduce
\begin{equation*}
\int_{\U(X)} \mathbf{s}( k^{-1} \circ  j \circ t) dt = \int_{K_X} \mathbf{s}( k^{-1} \circ j \circ t) dt = 1
\end{equation*}
by definition of $ \mathbf{s}$. Now, if $hk \notin H_j \cdot K_V$, then $k^{-1} \notin K_V$ and hence
\begin{equation*}
\int_{\U(X)} \mathbf{s}((hk)^{-1} \circ j \circ t) dt = \int_{K_X} \mathbf{s}( k^{-1} \circ  j \circ t) dt  = 0.
\end{equation*}
\end{proof}
\subsubsection{Archimedean test function}
In view of section \ref{sec5.3} and the assumptions on the signature of $V$ at infinity, in particular at $v_0$, we shall be interested in the type $1$ dual reductive pair $(G,G')$ with either
\begin{itemize}
\item $G = \Ortho(V) \simeq \Ortho(n,m)$ or $G = \U(V) \simeq \U(n,m)$;
\item $G' = \Sp(W) \simeq \Sp_{2m}(\mathbb{R})$ or $G' = \U(W) \simeq \U(m,m)$.
\end{itemize}
This way, we obtain a dual pair in either $\Sp_{2(n+m)m}(\mathbb{R})$ or $\Sp_{4(n+m)m}(\mathbb{R})$.

Fix a maximal isotropic subspace $X \subset W$ of real dimension $m$ (resp. $2m$) and endow it with a negative definite quadratic or Hermitian form. Let $\psi$ denote the non-trivial additive character $\psi(x) = \exp (2i \pi x)$ or $\mathbb{R}$ and consider the Schrödinger model of the Weil representation $\omega = \omega_{\psi}$, that is $(G,G')$ acts on Schwartz functions on $\Hom_{\mathbb{R}} (X, V) \simeq \mathbb{R}^d$, where $d = (n+m)m$ or $d = 2(n+m)m$. If we denote by $Q$ the form on $V$ and $q$ the form on $X$ then, inside $\Hom_{\mathbb{R}} (X, V)$, sits the closed sub-variety
\begin{equation*}
\mathcal{Y}(\mathbb{R}) = \lbrace y \in \Hom_{\mathbb{R}} (X, V) : Q \circ y = q \rbrace.
\end{equation*}
We assume $\mathcal{Y}(\mathbb{R})$ is non empty and we fix a non-zero quadratic (resp. Hermitian) space embedding $j : X \rightarrow V$. 

We endow the space of Schwartz functions $\mathcal{S}(\mathbb{R}^d)$ with the topology coming from the semi-norms $\Vert x^a f^{(b)} \Vert_{\infty}$ ($a,b \in \mathbb{N}$) and we consider the dense (see \cite[Proposition 9.5]{BM2}) sub-space $\mathcal{S}^{\text{alg}}(\mathbb{R}^d)$ made up of all products of polynomials on $\mathbb{R}^d$ with the Gaussian $\exp(- \Vert x \Vert^2 / 2)$. 

Let $K$ denoter either $\Ortho(j(X)^{\perp}) \times \Ortho(j(X))$ or $\U(j(X)^{\perp}) \times \U(j(X))$, a maximal compact subgroup of $G$. Let $\mathfrak{k}$ denote its Lie algebra and let $\mathfrak{g} = \mathfrak{k} \oplus \mathfrak{p}$ be the Cartan decomposition of the Lie algebra of $G$. Let $\mathfrak{a} \subset \mathfrak{p}$ be a maximal abelian subspace and, given $\lambda \in \mathfrak{a}^*_{\mathbb{C}}$, recall from Section \ref{sec3.2} the spherical function $\varphi_{\lambda}$ of spectral parameter $\lambda$.

Given $s \in \mathcal{S}(\mathcal{Y}(\mathbb{R}))$, we define a transform $\widehat{s}$ of $s$ by
\begin{equation*}
	\widehat{s}(\lambda) = \int_G s (g^{-1} \circ j) \varphi_{-\lambda}(g) dg \quad (\lambda \in \mathfrak{a}^*_{\mathbb{C}}).
\end{equation*}
\begin{lemma} \label{lem5.2.3}
Let $C \subset \mathfrak{a}^*_{\mathbb{C}}$ be a compact subset. There is $s_{C} \in \mathcal{S}^{\text{alg}}(\mathcal{Y}(\mathbb{R}))$ such that $\widehat{s_{C}}(\lambda) \neq 0$ for all $\lambda \in C$.
\end{lemma}
\begin{proof}
Since $\mathcal{S}^{\text{alg}}(\mathcal{Y}(\mathbb{R}))$ is dense in $\mathcal{S}(\mathcal{Y}(\mathbb{R}))$, it suffices to find $s \in \mathcal{S}(\mathcal{Y}(\mathbb{R}))$ such that $\widehat{s}(\lambda) \neq 0$ for $\lambda \in C$.

Recall by Witt's Theorem the map $g \mapsto g^{-1} \circ j$ induces an isomorphism ${H_j^{\perp} \backslash G \simeq \mathcal{Y}(\mathbb{R})}$, where $H_j^{\perp}$ is either $\Ortho(j(X)^{\perp})$ or $\U(j(X)^{\perp})$. In particular, $H_j^{\perp}$ is a subgroup of the maximal compact subgroup $K$ of $G$.

Let $S$ be a Weyl group invariant Paley-Wiener function on $\mathfrak{a}^*_{\mathbb{C}}$ with support containing $C$ and apply Plancherel inversion formula to obtain the function $s \in C_c^{\infty}(K \backslash G / K)$ defined by (see \cite[Theorem 3.5]{Gangolli})
\begin{equation*}
s(g) = \int_{i \mathfrak{a}^*} S(\lambda) \varphi_{\lambda}(g) \dfrac{d\mu_{\text{Pl}}(\lambda)}{\vert W(\mathfrak{a}) \vert}.
\end{equation*}
Since $H_j^{\perp} \subset K$, $s$ naturally gives rise to a compactly supported and smooth function on $H_j^{\perp} \backslash G$, which, in turn, yields a function $\widetilde{s} \in C^{\infty}_c(\mathcal{Y}(\mathbb{R})) \subset \mathcal{S}(\mathcal{Y}(\mathbb{R}))$; namely ${\widetilde{s}(g^{-1} \cdot j) = s(g)}$.
\end{proof}
\subsection{Application} \label{sec5.3}
With some additional archimedeans assumptions on $G$ (namely the assumptions in Theorem \ref{theo1.2.1}) and a test function constructed place by place using Section \ref{ssec5.2}, we compute the right hand side of the period relation as stated in Proposition \ref{prop5.1.1}. More precisely, we assume there is an archimedean place $v_0$ such that
\begin{itemize}
\item $G(F_{v_0}) = \U(n,m)$ (resp. $G(F_{v_0}) = \Ortho(n,m)$) and $\U(X)(F_{v_0}) = \U(0,m)$ (resp. ${\U(X)(F_{v_0}) = \Ortho(0,m)}$);
\item $G(F_v) = \U(n+m,0)$ (resp. $G(F_v) = \Ortho(n+m,0)$) and $\U(X)(F_v) = \U(m,0)$ (resp. $\U(X)(F_v) = \Ortho(m,0)$) for every archimedean place $v \neq v_0$.
\end{itemize}
Let $\mathfrak{g} = \mathfrak{k} \oplus \mathfrak{p}$ be the Cartan decomposition of the Lie algebra of $G(F_{v_0})$ with respect to the Lie algebra $\mathfrak{k}$ of a maximal compact subgroup $K_{v_0} \subset G(F_{v_0})$ and fix a maximal abelian subspace $\mathfrak{a} \subseteq \mathfrak{p}$.
\begin{lemma} \label{lem5.3.1}
Let $K = \prod_v K_v \subset G(\mathbb{A}_f)$ be an open compact subgroup and let $C \subset \mathfrak{a}_{\mathbb{C}}^*$ be a compact subset. There is $s = s_K \otimes s_{C} \in \omega$, with $s_C = \otimes_{v \mid \infty} s_v$, such that the following holds: if $f_{\lambda}$ is a right-$KK_{\infty}$-invariant automorphic form on $G$ which is an eigenfunction of the ring of invariant differential operators on $G(F_{v_0}) / K_{v_0}$ with spectral parameter $\lambda \in C$, then
\begin{equation*}
\mathcal{P}^{\mathbf{1} \times \psi_q}_{R_q}(\theta (f_{\lambda} ; s)) =  \mathcal{P}_{H_j} ({f_{\lambda}}) \vol \big( H_j (\mathbb{A}_f) K \big)  \widehat{s_{v_0}}({\lambda}) \prod_{v \neq v_0} s_{v}(j) 
\end{equation*}
with
\begin{equation*}
\vol \big( H_j (\mathbb{A}_f) K \big)  \widehat{s_{v_0}}({\lambda}) \prod_{v \neq v_0} s_{v}(j) \neq 0
\end{equation*}
and where $ \vol \big( H_j (\mathbb{A}_f) K \big)$ is the volume of $H_j (\mathbb{A}_f) \cdot K$ considered as a subset of $H_j(\mathbb{A}_f) \backslash G(\mathbb{A}_f)$.
\end{lemma}
\begin{proof}
We return to the period relation
 \begin{equation*}
\mathcal{P}^{\mathbf{1} \times \psi_q}_{R_q}(\theta (f ; s)) =  \int_{H_j (\mathbb{A}) \backslash G(\mathbb{A})} \mathcal{I}_{j,s}(x) \text{ } \mathcal{P}_{H_j}(R(x){f_{\lambda}}) dx
\end{equation*}
given by Proposition \ref{prop5.1.1} and we write its right hand side as
\begin{equation*}
\int_{H_j (\mathbb{A}_f) \backslash G(\mathbb{A}_f)} \mathcal{I}_{j,s_K}(x_f) \int_{H_j (\mathbb{A}_{\infty}) \backslash G(\mathbb{A}_{\infty})} \mathcal{I}_{j,s_{\infty}}(x_{\infty})  \mathcal{P}_{H_j}(R(x_f x_{\infty}){f_{\lambda}}) dx_{\infty} dx_f.
\end{equation*}
We define the test function $s_K \otimes s_{\lambda}$ place by place. At archimedean places we choose $s_{C} = \otimes_{v \mid \infty} s_v$ such that $s_v(j) \neq 0$ for all $v \neq v_0$ and $\widehat{s_{v_0}}(\lambda) \neq 0$ for all $\lambda \in C$. By averaging, we can assume $(g,t) \mapsto \omega(g,t) s_{C}$ is also right-$K_{\infty} \times \U(X)(\mathbb{A}_{\infty})$-invariant. Then, at each finite place $v$, we consider ${s_v : \Hom_{F_v}(X_v, V_v) \rightarrow \mathbb{C}}$ such that
\begin{equation*}
g \in G(F_v) \mapsto \int_{\U(X)(F_v)} \omega(g,t) s_v(j) dt
\end{equation*}
is the indicator function of the double coset $H_j(F_v) \cdot K_v$.

Given $x_f \in H_j(\mathbb{A}_f) \backslash G(\mathbb{A}_f)$, we have the archimedean integral
\begin{align*}
\int_{H_j (\mathbb{A}_{\infty}) \backslash G(\mathbb{A}_{\infty})} \mathcal{I}_{j,s_{\infty}}(x_{\infty})  \mathcal{P}_{H_j}(R(x_f x_{\infty}){f_{\lambda}}) dx_{\infty}
 \end{align*}
and, since $x_{v_0} \mapsto {f_{\lambda}}(x_{v_0})$ is an eigenfunction with spectral parameter ${\lambda}$, ${x_{v_0} \mapsto \mathcal{P}_{H_j} (R(x_f x_{v_0}){f_{\lambda}})}$ is also an eigenfunction with spectral parameter ${\lambda}$. From uniqueness of spherical functions, it follows that
\begin{equation*}
	\mathcal{P}_{H_j} (R(x_f x_{\infty}){f_{\lambda}}) = \varphi_{{\lambda}} (x_{v_0}) \mathcal{P}_{H_j} (R(x_f){f_{\lambda}}).
\end{equation*}
This way, the archimedean integral unfolds to
\begin{align*}
\hspace*{-6mm} \int_{H_j (\mathbb{A}_{\infty}) \backslash G(\mathbb{A}_{\infty})} \hspace*{-4mm} \mathcal{I}_{j,s_{\infty}}(x_{\infty})  \mathcal{P}_{H_j}(R(x_f x_{\infty}){f_{\lambda}}) dx_{\infty} = \prod_{v \neq v_0} s_v (j) \int_{G(F_{v_0})} \hspace*{-4mm} s_{v_0} (g_{v_0}^{-1} \circ j) \varphi_{\lambda} (g_{v_0}) dg_{v_0} \text{ } \mathcal{P}_{H_j}(R(x_f){f_{\lambda}})
\end{align*}
and we obtain
 \begin{equation*}
\mathcal{P}^{\mathbf{1} \times \psi_q}_{R_q}(\theta (f ; s)) = \widehat{s_{v_0}}({\lambda}) \prod_{v \neq v_0} s_{v}(j)   \int_{H_j(\mathbb{A}_f) \backslash G(\mathbb{A}_f)} \mathcal{I}_{j,s_K}(x_f) \text{ } \mathcal{P}_{H_j}(R(x){f_{\lambda}}) dx_f.
\end{equation*}
By construction of $s_K$ and right-$K$-invariance of $f_{\lambda}$, we have
\begin{equation*}
\int_{H_j(\mathbb{A}_f) \backslash G(\mathbb{A}_f)} \mathcal{I}_{j,s_K}(x) \text{ } \mathcal{P}_{H_j}(R(x){f_{\lambda}}) dx = \vol \big( H_j (\mathbb{A}_f) K \big)  \mathcal{P}_{H_j} ({f_{\lambda}}).
\end{equation*}
and this proves the Lemma.
\end{proof}
\subsection{Classical rewriting} \label{sec5.4}
Let $\mathcal{R}_0$ be the finite set of finite places at which either $E/F$ or $V$ is ramified. For any $v \in \mathcal{R}_0$, we fix open compact subgroups $K_v$ of $G(F_v)$. At a given finite place $v \notin \mathcal{R}_0$, we let $K_v$ denote the stabilizer of a self-dual lattice $L_{V,v}$ in $V_v$ and we set $K_f = \prod_{v \nmid \infty} K_v$. Let $\mathfrak{n}$ be an ideal of $\mathcal{O}$, the ring of integers of $F$, and assume $\mathfrak{n}$ is prime to $\mathcal{R}_0$. If $v \mid \mathfrak{n}$, we define $K_v(\mathfrak{n})$ as
\begin{equation*}
	K_v(\mathfrak{n}) = \lbrace g \in G(F_v) : (g- \id) L_{V,v} \subseteq \mathfrak{n} L_{V,v} \rbrace.
\end{equation*}
Then, we define the principal congruence subgroup of level $\mathfrak{n}$ as the product
\begin{equation*}
	K(\mathfrak{n}) = \prod_{v \mid \mathfrak{n}} K_v(\mathfrak{n}) \prod_{v \nmid \mathfrak{n} \infty} K_v.
\end{equation*}
Let $K_{H_j}(\mathfrak{n})$ denote a congruence subgroup of level $\mathfrak{n}$.\\
Let $C \subset \mathfrak{a}_{\mathbb{C}}^*$ be a compact subset and let $\lambda \in C$. Let $f_{\lambda}$ be a right-$K_{H_j}(\mathfrak{n})K_{\infty}$-invariant automorphic form on $G$ that is an eigenfunction of the ring of invariant differential operators on $G(F_{v_0}) / K_{v_0}$ with spectral parameter $\lambda$. Lemma \ref{lem5.3.1} ensures $\mathcal{P}_{H_j}(f_{\lambda})$ vanishes if and only $\mathcal{P}^{\mathbf{1} \times \psi_q}_{R_q}(\theta (f_{\lambda} ; s_{\mathfrak{n}, C}))$ does, with ${s_{\mathfrak{n}, C} = s_{\mathfrak{n}} \otimes s_{C}}$ defined (as in the proof of Lemma \ref{lem5.3.1}) with respect to $K = K_{H_j}(\mathfrak{n})$. Because the $H_j$-periods are discrete periods, we can write $\mathcal{P}_{H_j} (f_{\lambda})$ as a weighted sum of point evaluations as follows. Let $\gen_{H_j}(K_f^{H_j})$ denote the finite set (see \cite[Theorem 5.1]{BorelFiniteness})
\begin{equation*}
\gen_{H_j}(K_f^{H_j}) = H_j (F) \backslash H_j(\mathbb{A}_f) / K_f^{H_j}
\end{equation*}
where $K_f^{H_j} = K_f \cap H_j(\mathbb{A}_f)$. Assume\footnote[1]{These assumptions have been discussed in Section \ref{secccc1.2.1}; note that the compactness of $H_j$ at infinity follows from the conditions on the signature of $V$ and $X$ stated in Section \ref{sec5.3}}
\begin{equation*}
K_{H_j}(\mathfrak{n}) \cap H_j(\mathbb{A}_f) = K^{H_j}_{f} \quad , \quad {K_{\infty} \cap H_j (\mathbb{A}_{\infty}) = H_j(\mathbb{A}_{\infty})}.
\end{equation*}
Since $f_{\lambda}$ is right-$K_{H_j}(\mathfrak{n}) K_{\infty}$-invariant, we deduce
\begin{equation*}
\mathcal{P}_{H_j} (f_{\lambda}) = \sum_{h \in \gen_{H_j}(K_f^{H_j})} w_h f_{\lambda} (h)
\end{equation*}
with $w_h = \vol ( H_j(F) h  K^{H_j}_{f} )$ the volume of $H_j(F) \cdot h \cdot  K^{H_j}_{f}$ considered as a subset of $H_j(F) \backslash H_j(\mathbb{A}_f)$. Note that the number of terms (and also the weights) is independent of $\lambda$ and $\mathfrak{n}$.\\
\hspace*{8mm}We now write $\mathcal{P}_{H_j}(f_{\lambda})$ as a classical (\textit{i.e.} non adelic) weighted sum of point evaluation. Consider the finite (see \cite[Theorem 5.1]{BorelFiniteness}) set
\begin{equation*}
	\gen_G (\mathfrak{n}) = G(F) \backslash G(\mathbb{A}_f) / K_{H_j}(\mathfrak{n})
\end{equation*}
and recall the identification
\begin{equation*}
 G(F) \backslash G(\mathbb{A}_f) / K_{H_j}(\mathfrak{n}) K_{\infty} = \bigcup_{g \in \gen_G (\mathfrak{n}) } \Gamma_{H_j}^g ({\mathfrak{n}}) \backslash S
\end{equation*}
where $\Gamma_{H_j}^g ({\mathfrak{n}}) = G(F) \cap g K_{H_j}(\mathfrak{n}) g^{-1}$ and $S = G(F_{v_0}) / K_{v_0}$.

A function $f \in L^2( G(F) \backslash G(\mathbb{A}) / K_{H_j}(\mathfrak{n}) K_{\infty} )$ identifies with a collection of functions $\phi_f^{g} \in L^2(\Gamma_{H_j}^g ({\mathfrak{n}}) \backslash S)$ defined by $\phi_f^g (g_{\infty}) = f(gg_{\infty})$. Given $g \in \gen_G (\mathfrak{n})$ and letting $\mathcal{H}^g$ denote the set of those $h \in \gen_{H_j}(K_f^{H_j})$ that are mapped to $\Gamma_{H_j}^g ({\mathfrak{n}}) \backslash S$, we have
\begin{equation*}
\mathcal{P}_{H_j} (f_{\lambda}) = \sum_{g \in \gen_G (\mathfrak{n})} \sum_{h \in \mathcal{H}^g} w_h \phi_f^{g} (h).
\end{equation*}

\section{Mean square estimate of discrete periods} \label{chap6}
In this Section we explicit the hybrid asymptotic behavior of finite weighted sums of point evaluations. This may be seen as an analog of \cite[§7]{BM2} where one keeps track of the dependence in the volume of the underlying manifold.

Let $G$ be a non-compact real Lie group, let $K \subset G$ be a maximal compact subgroup and set $S = G/K$ for the associated globally symmetric space. Assume there is a connected simple Lie group $G_s$ with maximal compact subgroup $K_s$ such that $G/K = G_{s}/K_{s}$ and $\mathcal{D}(G/K) = \mathcal{D}(G_{s}/K_{s})$.

Let $\mathfrak{g}$ be the Lie algebra of $G$ and $\mathfrak{k}$ the Lie algebra of $K$. Let $\mathfrak{g} = \mathfrak{k} \oplus \mathfrak{p}$ the Cartan decomposition of $\mathfrak{g}$ and fix a maximal abelian subspace $\mathfrak{a} \subset \mathfrak{p}$. We fix a Haar measure $dg$ on $G$, a Haar measure $dk$ on $K$ and endow $S$ with the quotient measure $d\mu$ defined by $dg$ and $dk$. For any lattice $\Gamma \subseteq G$, denote by $\mu_{\Gamma}$ the quotient measure $d\mu_{\Gamma}$ on $X_{\Gamma} = \Gamma \backslash S$ defined by
\begin{equation*}
\int_{X_{\Gamma}} \sum_{\gamma \in \Gamma} \phi(\gamma \cdot x) d\mu_{\Gamma}(x) = \int_{S} \phi (x) d\mu(x) \quad (\phi \in C_c^{\infty}(S)).
\end{equation*}
We denote the volume of $X_{\Gamma}$, with respect to $d\mu_{\Gamma}$, as $\vol (X_{\Gamma})$ and set $d\overline{\mu_{\Gamma}}$ for the probability measure on $X_{\Gamma}$, that is $d\overline{\mu_{\Gamma}} = \frac{d\mu_{\Gamma}}{\vol(X_{\Gamma})}$. Finally, for $\Gamma \subseteq G$, we fix an orthonormal basis $\lbrace f^{\Gamma}_{\lambda} \rbrace_{\lambda}$ of $L^2(X_{\Gamma}, d \overline{\mu_{\Gamma}})$ consisting Maass forms $f^{\Gamma}_{\lambda}$ of spectral parameter $\lambda$.\\
Fix a uniform lattice $\Gamma_0 \subset G$.
\begin{proposition} \label{prop6.0.1}
There is $Q \geq 1$ such that the following holds: for any uniform lattice $\Gamma \subset G$ contained in $\Gamma_0$,  for any ${x_1, \ldots x_h \in X_{\Gamma}}$, any ${w_1, \ldots w_h \in \mathbb{C}^{\times}}$ and any $\nu \in i \mathfrak{a}^*$ of large enough norm (depending only on the minimal distance between the $x_i$'s, \textit{c.f.} \eqref{eqd}), we have
\begin{equation*}
\sum_{\Vert \im(\lambda) - \nu \Vert \leq Q} \left\lvert \sum_{i = 1}^{h} w_i f^{\Gamma}_{\lambda}(x_i) \right\rvert^2 \asymp \vol (X_{\Gamma}) \widetilde{\beta_S}(\nu).
\end{equation*}
The implied constant in the lower depends only on the weights and $h$ while, in the upper bound, it additionally depends on $\Gamma_0$ and $Q$.
\end{proposition}
The orthogonal and unitary groups $\Ortho(n,m)$ and $\U(n,m)$, for ${n > m \geq 1}$\footnote[1]{The proof given in \cite[§7]{BM2} uses the condition $n > m$ and applies to orthogonal groups. However, for unitary groups, this holds for $n \geq m$.}, satisfy the required conditions, that is, there is a connected simple Lie group $G_s$ with maximal compact subgroup $K_s$ such that
\begin{equation*}
G/K = G_s / K_s \quad \text{and} \quad \mathcal{D}(G/K) = \mathcal{D}(G_s/K_s).
\end{equation*}
For orthogonal groups, this follows from \cite[§7]{BM2} with
\begin{equation*}
G_s = \SO^0(n,m) \quad \text{and} \quad K_s = \SO^0(n) \times \SO^0(m).
\end{equation*}
For unitary groups, we can choose
\begin{equation*}
G_s = \SU(n,m) \quad \text{and} \quad K_s = \text{S}(\U(n) \times \U(m)).
\end{equation*}
To see this, let $\mathfrak{g} = \mathfrak{k} \oplus \mathfrak{p}$ be the Cartan decomposition of $\mathfrak{g} = \mathfrak{u}(n,m)$ induced by the Cartan involution $\text{ad}(i \cdot \Id_{n,m})$ on $\mathfrak{g}$, where $i = \sqrt{-1}$ and
\begin{equation*}
I_{,nm} = \begin{pmatrix} \Id_{n,n} & 0_{n,m} \\ 0_{m,n} & - \Id_{m,m} \end{pmatrix}
\end{equation*}
so that
\begin{equation*}
\mathfrak{k} = \begin{pmatrix} A & 0_{n,m} \\ 0_{m,n} & B \end{pmatrix} \text{ , } \mathfrak{p} = \begin{pmatrix} 0_{n,n} & C \\ {^{T} \overline{C}} & 0_{m,m} \end{pmatrix}.
\end{equation*}
Note that the center $\mathfrak{z} = \lbrace z \cdot \Id_{n+m, n+m} : z \in i\mathbb{R} \rbrace$ of $\mathfrak{g}$ as trivial intersection with $\mathfrak{p}$, proving in particular $G/K = G_s / K_s$ \cite[§4.6]{GangolliVara}. Let $\mathfrak{g}_s = \mathfrak{k}_s \oplus \mathfrak{p}_s$ be the Cartan decomposition of the Lie algebra $\mathfrak{g}_s$ of $G_s$ associated to $\mathfrak{k}_s$, the Lie algebra of $K_s$. Let $\mathfrak{a}_s \subset \mathfrak{p}_s$ be a maximal abelian subspace. Since $\mathfrak{z} \cap \mathfrak{p} = \lbrace 0 \rbrace$, $\mathfrak{a}_s$ is a maximal abelian subspace of $\mathfrak{p}$. Let $W(\mathfrak{a}_s, \mathfrak{g})$ and $W(\mathfrak{a}_s, \mathfrak{g}_s)$ be the associated Weyl groups. Let ${\mathfrak{a}_{s}}_{\mathbb{C}} = \mathfrak{a}_s \otimes \mathbb{C}$ and recall the Harish-Chandra isomorphisms \cite[Theorem 2.6.7]{GangolliVara}
\begin{equation*}
	\Symm ({\mathfrak{a}_{s}}_{\mathbb{C}})^{W(\mathfrak{a}_s, \mathfrak{g})}  \simeq \mathcal{D}(G/K) \quad , \quad \Symm ({\mathfrak{a}_{s}}_{\mathbb{C}})^{W(\mathfrak{a}_s, \mathfrak{g}_s)} \simeq \mathcal{D}(G_s/K_s).
\end{equation*}
By Chevalley's Theorem on invariants of finite reflexion groups \cite{Chevalley}, we have
\begin{equation*}
	\Symm ({\mathfrak{a}_{s}}_{\mathbb{C}})^{W(\mathfrak{a}_s, \mathfrak{g})} \simeq \mathbb{C}[t_1, \ldots , t_{\dim {\mathfrak{a}_{s}}_{\mathbb{C}}}] \simeq \Symm ({\mathfrak{a}_{s}}_{\mathbb{C}})^{W(\mathfrak{a}_s, \mathfrak{g}_s)}
\end{equation*}
and this proves $\mathcal{D}(G/K) = \mathcal{D}(G_{s}/K_{s})$.
\subsection{Test function} \label{sec6.1}
Let $\mathfrak{g}_{s}$ and $\mathfrak{k}_{s}$ be the Lie algebras of $G_{s}$ and $K_{s}$ respectively and let $\mathfrak{g}_{s} = \mathfrak{k}_{s} \oplus \mathfrak{p}$ be the Cartan decomposition of $\mathfrak{g}_{s}$. We fix a maximal abelian subspace $\mathfrak{a} \subset \mathfrak{p}$ and denote by $\langle \cdot , \cdot \rangle$ the restriction of the Killing form on $\mathfrak{g}_{s}$ to $\mathfrak{a}$. Let $A$ denote the analytic subgroup of $G_{s}$ corresponding to $\mathfrak{a}$ and $G_{s} = K_{s}AK_{s}$ the Cartan decomposition. For $R > 0$, we let $B_{\mathfrak{a}}(0,R) \subset \mathfrak{a}$ be the ball centered at the origin with radius $R$ (with respect to the norm induced by $\langle \cdot, \cdot \rangle$). Let $A_R = \exp \big( B_{\mathfrak{a}}(0,R) \big)$ and set
\begin{equation*}
	G_{s, R} = K_{s}A_RK_{s}.
\end{equation*}
Recall the unitary spectrum $\mathfrak{a}^*_{\text{unit}} \subset \mathfrak{a}^*_{\mathbb{C}}$ defined in Section \ref{sec3.2} as the set of those $\lambda$ in $\mathfrak{a}_{\mathbb{C}}^*$ for which the associated spherical function $\varphi_{\lambda}$ is positive definite. Using Proposition \ref{prop3.7.1}, we can construct a smooth and compactly supported bi-$K$-invariant test function $k_{\nu}$ on $G_{s}$ with support contained in $G_{s, R}$ and whose Harish-Chandra transform $\widehat{k_{\nu}}$ satisfies
\begin{enumerate}
\item $\widehat{k_{\nu}} (\lambda) \ll_{A,R} \exp (R \Vert \re (\lambda) \Vert) \sum_{w \in W} (1 + \Vert w \cdot \lambda - \nu \Vert)^{-A}$ for all $\lambda \in \mathfrak{a}^*_{\mathbb{C}}$; \label{property1}
\item $\widehat{k_{\nu}}(\lambda) \geq 0$ for all $\lambda \in \mathfrak{a}^*_{\text{unit}}$;  \label{property2}
\item $\widehat{k_{\nu}}(\lambda) \geq c$ for all $\lambda \in \mathfrak{a}^*_{\text{unit}}$ satisfying $\Vert \im(\lambda) - \nu \Vert \leq 1$.  \label{property3}
\end{enumerate}
The hypothesis $\mathcal{D}(G/K) = \mathcal{D}(G_s/K_s)$ ensures that spectral notions such that joint eigenfunctions, their spectral parameter, spherical functions and Harish-Chandra transform are insensitive to the choice of the presentation of $S = G/K = G_s / K_s$ as a symmetric space for $G$ or $G_{s}$.\\
Moreover, we may view $k_{\nu}$ as a bi-$K$-invariant smooth and compactly supported function on $G$. We may (and do) assume its support is small enough so that
\begin{equation*}
	\supp (k_{\nu}) \cap x_i {\Gamma_0} x_i^{-1} = \lbrace e \rbrace \quad (i = 1, \ldots, h).
\end{equation*}
In particular, for $\Gamma \subset \Gamma_0$, we have
\begin{equation}\label{eq6.1}
	\supp (k_{\nu}) \cap x_i {\Gamma} x_i^{-1} = \lbrace e \rbrace \quad (i = 1, \ldots, h).
\end{equation}
Recall the ${f^{\Gamma}_{\lambda}}$'s are $L^2$-normalized with respect to the probability measure on $X_{\Gamma}$. In other words, $\lbrace \vol (X_{\Gamma})^{-1/2}f^{\Gamma}_{\lambda}  \rbrace_{\lambda}$ is an orthonormal basis of $L^2(X_{\Gamma},d\mu_{\Gamma})$. Then, we spectrally expand the automorphic kernel associated to $k_{\nu}$: since $\Gamma$ is a co-compact lattice, there is no continuous spectrum and the pre-trace formula associated to $k_{\nu}$ is
\begin{equation} \label{eq6.2}
\sum_{\lambda} \widehat{k_{\nu}}(\lambda) f^{\Gamma}_{\lambda}(x_i) \overline{f^{\Gamma}_{\lambda}(x_j)}  = \vol (X_{\Gamma}) \sum_{\gamma \in \Gamma} k_{\nu}(x_i^{-1} \gamma x_j)
\end{equation}
for any $1 \leq i,j \leq h$.
\subsection{Local and upper bound}
Let $\nu \in i \mathfrak{a}^*$, construct $k_{\nu}$ as in Section \ref{sec6.1} and set $i = j$ in \eqref{eq6.2} so that
\begin{equation*}
\sum_{\lambda} \widehat{k}_{\nu}(\lambda) \vert f^{\Gamma}_{\lambda}(x_i) \vert^2 = \vol (X_{\Gamma}) \sum_{\gamma \in \Gamma} k_{\nu}(x_i^{-1} \gamma x_i) \quad (1 \leq i \leq h).
\end{equation*}
By construction of $k_{\nu}$ (\textit{cf.} \eqref{eq6.1}) we have
\begin{equation*}
\sum_{\lambda} \widehat{k}_{\nu}(\lambda) \vert f^{\Gamma}_{\lambda}(x_i) \vert^2 = \vol (X_{\Gamma}) \vert \stab_{\Gamma}(x_i) \vert k_{\nu} (e) \quad (1 \leq i \leq h)
\end{equation*}
with $1 \leq  \vert \stab_{\Gamma}(x_i) \vert \leq  \vert \stab_{\Gamma_{0}}(x_i) \vert$. Inserting Lemma \ref{lem3.7.3}, we may draw the important conclusion that
\begin{equation}\label{eq6.3}
\sum_{\lambda} \widehat{k}_{\nu}(\lambda) \vert f^{\Gamma}_{\lambda}(x_i) \vert^2 \asymp \vol (X_{\Gamma}) \widetilde{\beta_S}(\nu)
\end{equation}
for all $1 \leq i \leq h$. Using Property \ref{property2} of $\widehat{k}_{\nu}$ to drop all terms but those for which ${\Vert \im(\lambda) - \nu \Vert \leq 1}$ and using Property \ref{property3} then yields
\begin{equation} \label{eq6.4}
\sum_{\Vert \im(\lambda) - \nu \Vert \leq 1}  \vert f^{\Gamma}_{\lambda}(x_i) \vert^2 \ll \vol (X_{\Gamma}) \widetilde{\beta_{S}}(\nu) \quad (1 \leq i \leq h).
\end{equation}
On one hand, by dropping all but one term, we get the local bound
\begin{equation*}
\Vert f^{\Gamma}_{\lambda} \Vert_{\infty} \ll \vol (X_{\Gamma})^{\frac{1}{2}} \widetilde{\beta_{S}}(\nu)^{\frac{1}{2}}.
\end{equation*}
On the other hand, given $Q \geq 1$, by Cauchy-Schwartz inequality, we have
\begin{equation*}
\sum_{\Vert \im(\lambda) - \nu \Vert \leq Q} \left\lvert \sum_{i = 1}^{h} w_i f^{\Gamma}_{\lambda}(x_i) \right\rvert^2 \leq \sum_{\Vert \im(\lambda) - \nu \Vert \leq Q} h  \sum_{i = 1}^{h} \vert w_i f^{\Gamma}_{\lambda}(x_i) \vert^2
\end{equation*}
and, upon covering the region $\lbrace \mu \in i \mathfrak{a}^* : \Vert \mu - \nu \Vert \leq Q \rbrace$ by the (finite) union of unit balls centered at points $\mu_n$, we get
\begin{align*}
\sum_{\Vert \im(\lambda) - \nu \Vert \leq Q} \left\lvert \sum_{i = 1}^{h} w_i f^{\Gamma}_{\lambda}(x_i) \right\rvert^2
&\leq h \sum_{i = 1}^{h} \vert w_i \vert^2  \sum_n \sum_{\Vert \im(\lambda) - \mu_n \Vert \leq 1}  \vert f^{\Gamma}_{\lambda}(x_i) \vert^2 \\
&\ll \sum_n \vol(X_{\Gamma}) \widetilde{\beta_S}(\mu_n) \quad (\text{by \eqref{eq6.4}})\\
&\ll \vol(X_{\Gamma}) \widetilde{\beta_S}(\nu) \quad (\text{by \eqref{eq3.4} with $\sigma = \emptyset$}).
\end{align*}
This proves the upper bound in Proposition \ref{prop6.0.1}.
\subsection{Lower bound}
Given $\nu \in i \mathfrak{a}^*$ and $Q \geq 1$, we may write
\begin{align*}
\sum_{\Vert \im(\lambda) - \nu \Vert \leq Q} \widehat{k_{\nu}}(\lambda) \left\lvert \sum_{i = 1}^{h} w_i f^{\Gamma}_{\lambda}(x_i) \right\rvert^2
\end{align*}
as the sum
\begin{align*}
\sum_{\lambda} \widehat{k_{\nu}}(\lambda) \left\lvert \sum_{i = 1}^{h} w_i f^{\Gamma}_{\lambda}(x_i) \right\rvert^2 - \sum_{\Vert \im(\lambda) - \nu \Vert > Q} \widehat{k_{\nu}}(\lambda) \left\lvert \sum_{i = 1}^{h} w_i f^{\Gamma}_{\lambda}(x_i) \right\rvert^2
\end{align*}
of a main term and an error term. In the sequel, we estimate the error term, then the main term and finally deduce the lower bound in Proposition \ref{prop6.0.1}.
\subsubsection{Error term}
\begin{lemma} \label{lem6.3.1}
Let $\nu \in i \mathfrak{a}^*$ and let $Q > 1$, then
\begin{equation*}
\sum_{\Vert \im (\lambda) - \nu \Vert > Q} \widehat{k_{\nu}}(\lambda) \left\lvert \sum_{i = 1}^h w_i f^{\Gamma}_{\lambda}(x_i) \right\rvert^2 \ll_A Q^{-A} \vol (X_{\Gamma}) \widetilde{\beta_S}(\nu).
\end{equation*}
\end{lemma}
\begin{proof}
We cover the region $\lbrace \mu \in i \mathfrak{a}^* : \Vert \mu - \nu \Vert > Q \rbrace$ by the union of unit balls centered at points $\mu_n$ and, applying Property \ref{property1} of $\widehat{k_{\nu}}$, we have
\begin{equation*}
\sum_{\Vert \im (\lambda) - \nu \Vert > Q} \widehat{k_{\nu}}(\lambda) \left\lvert \sum_{i = 1}^h w_i f^{\Gamma}_{\lambda}(x_i) \right\rvert^2 \ll_A \sum_n \Vert \mu_n - \nu \Vert^{-A} \sum_{\Vert \im(\lambda) - \mu_n \Vert \leq 1} \left\lvert \sum_{i = 1}^h w_i f^{\Gamma}_{\lambda}(x_i) \right\rvert^2.
\end{equation*}
By Cauchy-Schwartz inequality, we deduce
\begin{equation*}
\sum_n \Vert \mu_n - \nu \Vert^{-A} \sum_{\Vert \im(\lambda) - \mu_n \Vert \leq 1} \left\lvert \sum_{i = 1}^h w_i f^{\Gamma}_{\lambda}(x_i) \right\rvert^2 \ll \sum_n \Vert \mu_n - \nu \Vert^{-A} \sum_{\Vert \im(\lambda) - \mu_n \Vert \leq 1} \sum_{i = 1}^h \vert f^{\Gamma}_{\lambda}(x_i) \vert^2.
\end{equation*}
Inserting \eqref{eq6.4}, \eqref{eq3.4} (with $\sigma$ the emptyset) and bounding ${\sum_n \Vert \mu_n - \nu \Vert^{-A} \ll Q^{-A}}$, we get the Lemma.
\end{proof}
\subsubsection{Main term}
Let $\nu \in i \mathfrak{a}^*$. We expand the square to write the main term
\begin{equation*}
\sum_{\lambda} \widehat{k_{\nu}}(\lambda) \left\lvert \sum_{i = 1}^{h} w_i f^{\Gamma}_{\lambda}(x_i) \right\rvert^2
\end{equation*}
as the sum
\begin{equation*}
\sum_{\lambda} \widehat{k_{\nu}}(\lambda) \sum_{i = 1}^{h} \vert w_i f^{\Gamma}_{\lambda}(x_i) \vert^2 + \sum_{\lambda} \widehat{k_{\nu}}(\lambda) \sum_{i \neq j} w_i \overline{w_j} f^{\Gamma}_{\lambda}(x_i) \overline{f^{\Gamma}_{\lambda}(x_j)}
\end{equation*}
of a diagonal and an anti-diagonal term. Using \eqref{eq6.3}, we obtain
\begin{equation} \label{eq6.5}
\sum_{\lambda} \widehat{k_{\nu}}(\lambda) \sum_{i = 1}^{h} \vert w_i f^{\Gamma}_{\lambda}(x_i) \vert^2 \asymp \vol (X_{\Gamma}) \widetilde{\beta_S}(\nu).
\end{equation}
Let $B$ be the Killing form on $\mathfrak{g} = \mathfrak{k} \oplus \mathfrak{p}$ and identify $\mathfrak{p}$ with the tangent space at the identity $K$ of $S = G/K$. Then, endow $S$ with the Riemannian metric induced by the restriction of $B$ to $\mathfrak{p}$ and let $d$ be the associated bi-$K$-invariant distance function on $G$. For $1 \leq i \leq h$, we write $x_i = \Gamma g_i K$ ($g_i \in G$). For $1 \leq i, j \leq h$, we set $d_{i,j} = \min_{\gamma \in \hspace*{-1mm} \Gamma} d(g_i, \gamma g_j)$ and define
\begin{equation} \label{eqd}
\mathbf{d} = \min_{i \neq j} d_{i,j}.
\end{equation}
Given $i \neq j$ and $\gamma \in \Gamma$, we may write $x_i^{-1} \gamma x_j = k \exp (a) k'$ for some $k,k' \in K$ and some $a \in \mathfrak{a}$ satisfying $\Vert a \Vert = d(g_i, \gamma g_j)$. From \eqref{eq3.6}, we deduce
\begin{equation*}
	k_{\nu}(x_i^{-1} \gamma x_j) \ll (1 + \Vert \nu \Vert \mathbf{d} )^{-1/2} \widetilde{\beta_S}(\nu) \quad (1 \leq i \neq j \leq h, \gamma \in \Gamma).
\end{equation*}
We can now estimate the anti-diagonal term. By \eqref{eq6.2}, we have
\begin{align*}
\sum_{\lambda} \widehat{k_{\nu}}(\lambda) \sum_{i \neq j} w_i \overline{w_j} f^{\Gamma}_{\lambda}(x_i) \overline{f^{\Gamma}_{\lambda}(x_j)}
&= \sum_{i \neq j} w_i \overline{w_j} \sum_{\lambda} \widehat{k_{\nu}}(\lambda)  f^{\Gamma}_{\lambda}(x_i) \overline{f^{\Gamma}_{\lambda}(x_j)} \\
&=  \sum_{i \neq j} w_i \overline{w_j} \vol (X_{\Gamma}) \sum_{\gamma \in \Gamma} k_{\nu}(x_i^{-1} \gamma x_j).
\end{align*}
Given $i \neq j$, we bound the number of terms appearing in the sum over $\Gamma$ by
\begin{equation*}
\max_{i \neq j} \text{card} \big( \lbrace \gamma \in \Gamma_0 : k_{\nu}(x_i^{-1} \gamma x_j) \neq 0 \rbrace \big).
\end{equation*}
This way, we obtain
\begin{equation} \label{eq6.6}
\sum_{\lambda} \widehat{k_{\nu}}(\lambda) \sum_{i \neq j} w_i \overline{w_j} f^{\Gamma}_{\lambda}(x_i) \overline{f^{\Gamma}_{\lambda}(x_j)} \ll (\Vert \nu \Vert \mathbf{d} )^{-1/2} \vol (X_{\Gamma}) \widetilde{\beta_S}(\nu).
\end{equation}
\subsubsection{Proof of the lower bound}
We are now ready to prove the lower bound in Proposition \ref{prop6.0.1}. Let $Q > 1$, we have
\begin{align*}
\sum_{\Vert \im(\lambda) - \nu \Vert \leq Q} \widehat{k_{\nu}}(\lambda) \left\lvert \sum_{i = 1}^{h} w_i f^{\Gamma}_{\lambda}(x_i) \right\rvert^2
&= \sum_{\lambda} \widehat{k_{\nu}}(\lambda) \sum_{i = 1}^{h} \vert w_i f^{\Gamma}_{\lambda}(x_i) \vert^2
+ \sum_{\lambda} \widehat{k_{\nu}}(\lambda) \sum_{i \neq j} w_i f^{\Gamma}_{\lambda}(x_i) \overline{w_j f^{\Gamma}_{\lambda}(x_j)} \\
&- \sum_{\Vert \im(\lambda) - \nu \Vert > Q} \widehat{k_{\nu}}(\lambda) \left\lvert \sum_{i = 1}^{h} w_i f^{\Gamma}_{\lambda}(x_i) \right\rvert^2.
\end{align*}
Applying \eqref{eq6.5}, \eqref{eq6.6} and Lemma \ref{lem6.3.1}, we deduce
\begin{align*}
\sum_{\Vert \im(\lambda) - \nu \Vert \leq Q} \widehat{k_{\nu}}(\lambda) \left\lvert \sum_{i = 1}^{h} w_i f^{\Gamma}_{\lambda}(x_i) \right\rvert^2
&\asymp \vol(X_{\Gamma}) \widetilde{\beta_S}(\nu) + O \big( (\Vert \nu \Vert \mathbf{d} )^{-1/2}  \vol(X_{\Gamma}) \widetilde{\beta_S}(\nu) \big)  \\
&+ O \big( Q^{-A}  \vol(X_{\Gamma}) \widetilde{\beta_S}(\nu) \big).
\end{align*}
Fixing $A > 1$ and taking $Q > 1$ large enough, and also taking $\nu$ with large enough norm (depending on $\mathbf{d}$), we can make both error terms smaller by a constant factor than the main term. This proves the lower bound.

\section{Weyl type law} \label{chap7}
Let $E$ be $\mathbb{Q}$ or a totally imaginary quadratic extension of $\mathbb{Q}$. Let $v_0$ be an archimedean place of $F$. We fix a split skew-Hermitian $E$-space of dimension $2m$, for some integer $m \geq 1$ and we set $G = \U(W)$ for the associated algebraic $F$-group defined. Note that the adelic quotient $[G] = G(\mathbb{Q}) \backslash G(\mathbb{A})$ is non-compact and has finite volume since the center of $G$ is anisotropic. Let $p$ be a prime number. If $E = \mathbb{Q}$, we let $K_p$ be the stabilizer of a self-dual lattice in $W_p = W \otimes \mathbb{Q}_p$. If $E \neq \mathbb{Q}$ and
\begin{itemize}
\item if $p$ is inert or ramified in $E$, we let $K_p$ be the stabilizer of a self-dual lattice in $W_p$ (see \cite[Chapitre 5, I. Remarque (2)]{MVW});
\item if $p$ splits in $E$, we let $K_p \simeq \GL_{2m}(\mathbb{Z}_p)$.
\end{itemize}
Fix a maximal compact subgroups $K_{\infty}$ of $G(\mathbb{R})$ and define $K_f = \prod_{v \nmid \infty} K_v.$ Let $\mathfrak{g} = \mathfrak{k} \oplus \mathfrak{p}$ be the Cartan decomposition of the Lie algebra $\mathfrak{g}$ of $G(\mathbb{R})$, where $\mathfrak{k}$ is the Lie algebra of $K_{\infty}$, and let $\mathfrak{a} \subset \mathfrak{p}$ be a maximal abelian subspace.

Let $\tau : K_{\infty} \rightarrow \mathbb{C}^{\times}$ be a continuous character, $Q \geq 1$ be a real number, let $\nu \in i \mathfrak{a}^*$ and let $K \subseteq K_f$ be an open compact subgroup. We denote by $\mathcal{A}_{\cusp}(K, \nu, \tau, Q)$ the set of (isomorphism classes of) irreducible cuspidal automorphic unitary representations $\pi = \pi_f \otimes \pi_{\infty}$ of $G$ such that
\begin{itemize}
\item $\pi_f^{K} \neq 0$;
\item $\pi_{\infty}$ is the unique $\tau$-spherical representation of $G(\mathbb{R})$ of spectral parameter $\lambda \in \mathfrak{a}^*_{\mathbb{C}}$ satisfying $\Vert \im(\lambda) - \nu \Vert \leq Q$ (see Corollary \ref{coro3.3.2}).
\end{itemize}
Let $\widetilde{\mathcal{A}}_{\cusp}(K,\nu, \tau, Q)$ denote the set of pairs $(\pi, E)$ where $\pi \in \mathcal{A}_{\cusp}(K, \nu, \tau, Q)$ and $E$ is a morphism from $\pi$ to the space of automorphic forms on $G$. We denote by $m(\pi)$ the multiplicity space of $\pi$, that is, the fiber of the map $\widetilde{\mathcal{A}}_{\cusp} (K, \nu, \tau, Q) \rightarrow \mathcal{A}_{\cusp}(K, \nu, \tau, Q)$ above $\pi$. We denote by $\mathcal{E}_{\cusp}(K,\nu,\tau,Q)$ the associated space of automorphic forms on $G$, namely
\begin{equation*}
	\mathcal{E}_{\cusp}(K,\nu,\tau,Q) = \bigoplus_{\pi \in \mathcal{A}_{\cusp}(K,\nu, \tau,Q)} m(\pi) \otimes \pi_f^{K}.
\end{equation*}
Finally, set
\begin{equation*}
X(K) = G(\mathbb{Q}) \backslash G(\mathbb{A}) / K K_{\infty}
\end{equation*}
and let $\widetilde{\beta_S}$ be the (majorant of the) spectral density function of $S = G(\mathbb{R}) / K_{\infty}$ (see Section \ref{sec3.5}). The main result of this Section is the following Proposition.
\begin{proposition} \label{prop7.0.1}
Let $Q \geq 1$, let $\nu \in i \mathfrak{a}^*$ and let $K \subseteq K_f$ be an open compact subgroup. We have
\begin{equation*}
	\dim \mathcal{E}_{\cusp}(K,\nu,\tau,Q) \ll_{\tau, Q} \vol (X(K)) \widetilde{\beta_S}(\nu) \log \big( (1 + \Vert \nu \Vert) \vol (X(K))^{\frac{1}{2}} \big)^{m}.
\end{equation*}
\end{proposition}
\begin{rmrk}
This logarithmic factor is responsible for the loss of the same quality in the main result.
\end{rmrk}
\begin{rmrk}
We count spectral parameters in spectral windows and not in growing balls. In this regard, the result of \cite{Maiti} is not fine enough for us.
\end{rmrk}
\begin{lemma} \label{lem7.0.4}
	Proposition \ref{prop7.0.1} follows from the case $Q = 1$.
\end{lemma}
\begin{proof}
	We cover the region the region $\lbrace \mu \in i \mathfrak{a}^* : \Vert \mu - \nu \Vert \leq Q \rbrace$ by the (finite) union of unit balls centered at points $\mu_n$ and we have
\begin{equation*}
	\dim \mathcal{E}_{\cusp}(K,Q,\nu,\tau) \leq \sum_n \dim \mathcal{E}_{\cusp}(K,1,\mu_n,\tau).
\end{equation*}
Assuming Proposition \ref{prop7.0.1} holds for $Q = 1$, we deduce
\begin{equation*}
	\dim \mathcal{E}_{\cusp}(K,Q,\nu,\tau) \ll \sum_n \vol (X(K)) \widetilde{\beta_S}(\mu_n) \log \big( (1 + \Vert \mu_n \Vert) \vol (X(K))^{\frac{1}{2}} \big)^{m}
\end{equation*}
and the Lemma now follows from \eqref{eq3.4} (with $\sigma$ the empty set) and ${1 + \Vert \mu_n \Vert \ll_Q 1 + \Vert \nu \Vert}$.
\end{proof}
Thus, we can assume $Q = 1$ and we set $\mathcal{E}_{\cusp}(K, \nu,\tau) =  \mathcal{E}_{\cusp}(K,\nu,\tau, 1)$. The proof follows closely \cite[§6]{BM2} which, in turn, is shaped around \cite[Lemma I.2.10]{MoeglinWald}. The idea is to take advantage of the rapid decay of cusp forms to obtain a sufficiently large truncated Siegel domain $[G]^{\leq T}$, for some sufficiently large parameter $T$ (depending on $K$ and $\nu$), on which cusp forms of spectral parameter about $\nu$ have $L^2$-mass bounded away from zero. We deduce Proposition \ref{prop7.0.1} by integrating the automorphic kernel, associated with (an approximation of) a spectral projector about $\nu$, over $[G]^{\leq T}$.\\
\hspace*{8mm} For reduction theory, we refer to \cite[Chapter 5]{PlatonovRapinchuk} or the article of Springer \cite{Springer}. In the latter, note that the group $G(\mathbb{Q})$ acts on $G(\mathbb{A})$ by right multiplication, changing the inequality defining $A_0(t)$.
\subsection{Reduction theory} \label{sec7.1}
Let $\mathbf{K} = K_f K_{\infty}$. Let $P$ be a minimal parabolic $\mathbb{Q}$-subgroup of $G$, let $N$ be its unipotent radical and let $L$ be a Levi subgroup so that $P = UL$ and $G(\mathbb{A}) = N(\mathbb{A})L(\mathbb{A})\mathbf{K}$. Set
\begin{equation*}
	M = \Big( \bigcup_{\chi \in X^*(L)} \ker \vert \chi \vert \Big)^0
\end{equation*}
where $X^*(L)$ denote the group of $\mathbb{Q}$-characters of $L$. Let $A$ be the maximal $\mathbb{Q}$-split torus of the center of $L$ so that $M(\mathbb{A})$ is the direct product of $M(\mathbb{A}) A(\mathbb{R})^0$. Letting $\mathfrak{a}$ denote the Lie algebra of $A(\mathbb{R})^0$, we can (and do) choose $K_{\infty}$ such that the Lie algebra of $K_{\infty}$ and $\mathfrak{a}$ are orthogonal with respect to the Killing form. We have
\begin{equation*}
	G(\mathbb{A}) = N(\mathbb{A}) M(\mathbb{A}) {A(\mathbb{R})^0} \mathbf{K}.
\end{equation*}
Let $R$ be the root system of $A$ in $G$, $R_+ \subset R$ be the positive roots determined by $N$ and let $\Phi \subset R_+$ denote a subset of simple roots. For any $t > 0$, define
\begin{equation*}
	A_0 (t) = \lbrace a \in A(\mathbb{R})^0 : a^{\alpha}\geq t \text{, } \forall \alpha \in \Phi \rbrace.
\end{equation*}
For any compact set $\omega \subset N(\mathbb{A}) M(\mathbb{A})$ and any $t > 0$, we define the ($\omega, t$)-Siegel subset of $G(\mathbb{A})$ by
\begin{equation*}
	\mathfrak{S}_{\omega, t} = \omega A_0(t) \mathbf{K}.
\end{equation*}
The main result of reduction theory is that, for sufficiently large $\omega$ and sufficiently small $t > 0$, $\mathfrak{S}_{\omega, t}$ is a fundamental set for the action of $G(\mathbb{Q})$ on $G(\mathbb{A})$ in the sense that
\begin{equation*}
	G(\mathbb{Q}) \mathfrak{S}_{\omega, t} = G(\mathbb{A})
\end{equation*}
with $\mathfrak{S}_{\omega, t}^{-1} \mathfrak{S}_{\omega, t} \cap G(\mathbb{Q})$ finite. For $T > t$, we define
\begin{equation*}
	A_0 (t, T) = \lbrace a \in A(\mathbb{R})^0 : T \geq  a^{\alpha} \geq t \text{, } \forall \alpha \in \Phi \rbrace
\end{equation*}
and set $\mathfrak{S}^{\leq T}_{\omega, t} = \omega A_0(t, T) \mathbf{K}$. Finally, we define truncations of the adelic quotient as follows
\begin{itemize}
\item $[G]^{\leq T}$ is the image of $\mathfrak{S}^{\leq T}_{\omega, t} = \omega A_0(t, T) \mathbf{K}$ in $[G]$;
\item $[G]^{> T}$ is the image of $\mathfrak{S}^{> T}_{\omega, t} = \omega A_0^{ > T} \mathbf{K}$ in $[G]$;
\end{itemize}
where
\begin{equation*}
	 A_0^{ > T} =  \lbrace a \in A(\mathbb{R})^0 : a^{\alpha} >  T \text{, } \forall \alpha \in \Phi \rbrace.
\end{equation*}
\subsection{Preliminary Lemmas} \label{sec7.2}
Let $dn$ be a Haar measure on $N(\mathbb{A})$, endow $N(\mathbb{Q})$ with the counting measure and normalize $dn$ so that $\vol (N(\mathbb{Q}) \backslash N(\mathbb{A}) ; dn) = 1$. Let $dk$ be the Haar measure on $\mathbf{K}$ such that $\vol (\mathbf{K} ; dk) = 1$, use the isomorphism $\exp : \mathfrak{a} \rightarrow A(\mathbb{R})^0$ to define a Haar measure on $A(\mathbb{R})^0$ and let $dm$ be a Haar measure on $M(\mathbb{A})$. Letting $\rho = \rho_{R_+}$ denote the half-sum of positive roots, there is Haar measure $dg$ on $G(\mathbb{A})$ such that
\begin{equation} \label{eq7.1}
	\int_{G(\mathbb{A})} \phi(g) dg = \int_{N(\mathbb{A})} \int_{M(\mathbb{A})} \int_{A(\mathbb{R})^0} \int_{\mathbf{K}} \phi (u m a k) a^{-2 \rho} dk da dm dn \quad \big( \phi \in C_c^{\infty} (G(\mathbb{A})) \big).
\end{equation}
We use this measure to define convolution of complex valued functions on $G(\mathbb{A})$, which we denote by $\ast$. If $k \in C^{\infty}_c (G(\mathbb{A}))$ we define the operator $R(k)$ by which $k$ acts on $L^2([G])$. We also define $k^{\vee} \in C^{\infty}_c(G(\mathbb{A}))$ by ${k^{\vee} (g) = \overline{k(g^{-1})}}$. 

The following Lemma gives a pre-trace inequality that will be used in both the proof of the uniform moderate growth estimate (Lemma \ref{lem7.3.1}) and the proof of Proposition \ref{prop7.0.1} in Section \ref{secc7.6}.
\begin{lemma} \label{lem7.2.1}
	Let $k \in C^{\infty}_c (G(\mathbb{A}))$ and let $\lbrace f_i \rbrace_i$ be an orthonormal set in $L^2([G])$, then
\begin{equation*}
	\sum_i \vert R(k) f_i (g) \vert^2 \leq \sum_{\gamma \in G(\mathbb{Q})} k \ast k^{\vee} (g^{-1} \gamma g) \quad (g \in G(\mathbb{A})).
\end{equation*}
\end{lemma}
\begin{proof}
By definition we have
\begin{equation*}
	R(k) f_i (g) = \int_{G(\mathbb{A})} k(x) f_i (gx) dx \quad (g \in G(\mathbb{A}))
\end{equation*}
and thus, by unfolding, we get
\begin{equation*}
	R(k) f_i (g) = \int_{[G]} \sum_{\gamma \in G(\mathbb{Q})} k(\gamma x) f_i (g \gamma x)  dx = \int_{[G]} f_i(x) \sum_{\gamma \in G(\mathbb{Q})} k(g^{-1} \gamma x) dx = \left\langle f_i , \overline{\sum_{\gamma \in G(\mathbb{Q})} k (g^{-1} \gamma \cdot )} \right\rangle
\end{equation*}
where $\langle \cdot , \cdot \rangle$ denotes the inner product in $L^2([G])$. Because the $f_i$'s are orthonormal, we can apply Parseval's Inequality to obtain
\begin{equation*}
\sum_i \vert R(k) f_i (g) \vert^2 \leq \left\lVert \sum_{\gamma \in {G}(\mathbb{Q})} k (g^{-1} \gamma \cdot ) \right\rVert_2^2 = \int_{[G]} \left\lvert \sum_{\gamma \in {G}(\mathbb{Q})} k (g^{-1} \gamma x) \right\rvert^2 dx.
\end{equation*}
Expanding the square gives
\begin{equation*}
\sum_i \vert R(k) f_i (g) \vert^2 \leq \int_{[G]} \sum_{\gamma_1, \gamma_2 \in {G}(\mathbb{Q})} k(g^{-1} \gamma_1 x) \overline{k(x^{-1} \gamma_2 x)} dx \quad (g \in G(\mathbb{A}))
\end{equation*}
and, by unfolding (with respect to $\gamma_2$), we deduce
\begin{align*}
\sum_i \vert R(k) f_i (g) \vert^2 &\leq \sum_{\gamma \in {G}(\mathbb{Q})} \int_{{G}(\mathbb{A})} k(g^{-1} \gamma x) \overline{k(g^{-1} x)} dx \\
&= \sum_{\gamma \in G(\mathbb{Q})} \int_{G(\mathbb{A})} k(g^{-1} \gamma x) k^{\vee} (x^{-1} g) dx \\
&= \sum_{\gamma \in G(\mathbb{Q})} \int_{G(\mathbb{A})} k(x) k^{\vee} (x^{-1} g^{-1} \gamma g) dx \\
&= \sum_{\gamma \in G(\mathbb{Q})} k \ast k^{\vee} (g^{-1} \gamma g).
\end{align*}
\end{proof}
We now establish a bound on the growth of automorphic kernel functions, which will be applied in estimating the uniform moderate growth of cusp forms in Lemma \ref{lem7.3.1}.
\begin{lemma} \label{lem7.2.2}
	Let $\mathcal{K}_f \subset G(\mathbb{A}_f)$ be an open compact subgroup and let $\mathcal{K}_{\infty} \subset G(\mathbb{R})$ be a compact set. For $g = n_g m_g a_g k_g \in \mathfrak{S}_{\omega, t}$, we have
\begin{equation*}
	\sum_{\gamma \in G(\mathbb{Q})} \mathbf{1}_{\mathcal{K}_f \times \mathcal{K}_{\infty}}(g^{-1} \gamma g) \ll a_g^{2 \rho}.
\end{equation*}
\end{lemma}
\begin{proof}
Notice that the function
\begin{equation*}
g_f \in G(\mathbb{A}_f) \mapsto \sum_{\gamma \in G(\mathbb{Q})} \mathbf{1}_{\mathcal{K}_f}(g_f^{-1} \gamma g_f)
\end{equation*}
is left-$G(\mathbb{Q})$-invariant and right-$\mathcal{K}_f$-invariant. In other words, the above sum depends on $g_f$ only through its conjugacy class in $G(\mathbb{Q}) \backslash G(\mathbb{A}_f) / \mathcal{K}_f$. Since this double quotient is finite (see \cite[Theorem 5.1]{BorelFiniteness}), upon fixing representatives $x_1, \ldots, x_h$ of $G(\mathbb{Q}) \backslash G(\mathbb{A}_f) / \mathcal{K}_f$, we are left with estimating the various sums
\begin{equation*}
	\sum_{\gamma \in \Gamma_i} \mathbf{1}_{\mathcal{K}_{\infty}}(g_{\infty}^{-1} \gamma g_{\infty}) \quad (1 \leq i \leq h)
\end{equation*}
where $\Gamma_i = G(\mathbb{Q}) \cap x_i \mathcal{K}_f x_i^{-1}$. As it is now independent of $g_f$, we can quote \cite[Lemma 6.2]{BM2} to deduce the Lemma.
\end{proof}
\subsection{Uniform moderate growth estimate} \label{ss7.33}
The main objective of Sections \ref{ss7.33}, \ref{ssseecc} and \ref{secaeaa} is to obtain control on the $L^2$-mass of cusp forms. More precisely, we show that cusp forms with spectral parameter $\lambda$, satisfying $\Vert \im(\lambda) - \nu \Vert \leq 1$, have a positive proportion of their $L^2$-mass in the truncated Siegel domain $[G]^{\leq T}$, provided $T \gg_{\epsilon} (1 + \Vert \nu \Vert)^{1 + \epsilon} \vol (K)^{- \frac{1 + \epsilon}{2}}$.\\
The corresponding statement is Lemma \ref{lem7.5.2} and we deduce this from a uniform (in both the spectral parameter and the volume) rapid decay estimate for cusp forms, namely Lemma \ref{lem7.4.1}. In turn, the latter estimate follows from a uniform (in both the spectral parameter and the volume) moderate growth estimate for cusp forms, that is, Lemma \ref{lem7.3.1}.

Let $\nu \in i \mathfrak{a}^*$ and $\tau : K_{\infty} \rightarrow \mathbb{C}^{\times}$ be a continuous character. Let $k_{\nu, \tau^{-1}}^0$ be the test function on either $\SU(m,m)$ or $\Sp_{2m}(\mathbb{R})$ appearing in the proof of Proposition \ref{prop3.7.1}; not the $\tau^{-1}$-spherical test function $k_{\nu, \tau^{-1}}$ itself, which is defined by ${k_{\nu, \tau^{-1}} = k_{\nu, \tau^{-1}}^0 \ast k_{\nu, \tau^{-1}}^0}$.\\
As explained in Section \ref{sec6.1} (see also the paragraph following Proposition \ref{prop6.0.1}), in the case $G(\mathbb{R}) = \U(m,m)$, we can extend $k_{\nu, \tau^{-1}}^0$ to a smooth and compactly supported function $k_{\infty}$ on $G(\mathbb{R})$. Let $f \in \mathcal{E}_{\cusp}(K,\nu,\tau)$. Scaling $k_{\infty}$ by a bounded constant, we may assume that $R(k_{\infty}) f = f$ and thus $R(k)f = f$, where ${k = \vol(K)^{-1} \mathbf{1}_K \otimes k_{\infty}}$.

In the following Lemma, we bound $D \cdot f$, for $f \in \mathcal{E}_{\cusp}(K,\nu,\tau)$ and $D$ a left-invariant and right-invariant diﬀerential operator on $G(\mathbb{R})$, by writing it as $R(\phi)f$ for some well-chosen $\phi \in C^{\infty}_c (G(\mathbb{A}))$ depending on $D$, $\tau$ and $\nu$. Then, we apply the pre-trace inequality, that is, Lemma \ref{lem7.2.1}, with $k = \phi$ and bound the right hand side using Lemma \ref{lem7.2.2}.
\begin{lemma} \label{lem7.3.1}
	Let $f \in \mathcal{E}_{\cusp}(K,\nu,\tau)$ be $L^2$-normalized, let $D$ be a left-invariant and right-invariant differential operator of degree $d$ and let ${g \in \mathfrak{S}_{\omega, t}}$, then
	\begin{equation*}
		D \cdot f (g) \ll (1 + \Vert \nu \Vert)^{d} \widetilde{\beta_S}(\nu)^{1/2} \vol(K)^{-1/2} a_g^{\rho} 
	\end{equation*}
where $g = n_g m_g a_g k_g$.
\end{lemma}
\begin{rmrk}
	Here $D \cdot f$ is defined as follows: any smooth function on $G(\mathbb{A})$ can be written as a finite sum of functions $\varphi_f \otimes \varphi_{\infty}$ with $(\varphi_f , \varphi_{\infty}) \in \big(C^{\infty}(G(\mathbb{A}_f)) \times C^{\infty}(G(\mathbb{R})) \big)$. We define $D$ on such $\varphi_f \otimes \varphi_{\infty}$ by
\begin{equation*}
	D \cdot (\varphi_f \otimes \varphi_{\infty}) = \varphi_f \otimes D \cdot \varphi_{\infty}
\end{equation*}
and we extend this action linearly.
\end{rmrk}
\begin{proof}
We have $D \cdot f = D \cdot R(k)f = R(\phi) f$, where ${\phi = D \cdot k = \vol(K)^{-1} \mathbf{1}_K \otimes D \cdot k_{\infty}}$. Then we apply Lemma \ref{lem7.2.1}, to the orthonormal set containing the single element $f$ with test function $\phi$, to obtain
\begin{equation*}
	\vert R(\phi) f(g) \vert^2 \leq \sum_{\gamma \in G(\mathbb{Q})} (\phi \ast \phi^{\vee}) (g^{-1} \gamma g).
\end{equation*}
Given $\gamma \in G(\mathbb{Q})$, we bound
\begin{equation*}
(\vol(K)^{-1} \mathbf{1}_K) \ast (\vol(K)^{-1} \mathbf{1}_K)^{\vee} (g_f^{-1} \gamma g_f) = \vol(K)^{-1} \mathbf{1}_K (g_f^{-1} \gamma g_f) \leq \vol (K)^{-1} \mathbf{1}_{K_f} (g_f^{-1}\gamma g_f)
\end{equation*}
and
\begin{equation*}
(D \cdot k_{\infty}) \ast (D \cdot k_{\infty})^{\vee} (g_{\infty}^{-1} \gamma g_{\infty}) \leq \Vert (D \cdot k_{\infty}) \ast (D \cdot k_{\infty})^{\vee} \Vert_{\infty}.
\end{equation*}
Noticing that the support of the archimedean test function is uniformly bounded, we can apply Lemma \ref{lem7.2.2}, with $\mathcal{K}_f = K_f$ and $\mathcal{K}_{\infty} = \supp \big( (D \cdot k_{\infty}) \ast (D \cdot k_{\infty})^{\vee} \big)$, to deduce
\begin{equation*}
\vert R(\phi) f(g) \vert^2 \ll \vol(K)^{-1} \Vert (D \cdot k_{\infty}) \ast (D \cdot k_{\infty})^{\vee} \Vert_{\infty} \text{ } a_g^{2 \rho}.
\end{equation*}
We now bound $\Vert (D \cdot k_{\infty}) \ast (D \cdot k_{\infty})^{\vee} \Vert_{\infty}$ as follows. We have
\begin{align*}
(D \cdot k_{\infty}) \ast (D \cdot k_{\infty})^{\vee} (g) &= \int_{G(\mathbb{R})} (D \cdot k_{\infty}) (x) \overline{(D \cdot k_{\infty}) (g^{-1} x)} dx \\
&= \int_{G(\mathbb{R})} (D \cdot k_{\infty}) (x) {(D^{\vee} \cdot k_{\infty}) (x^{-1} g)} dx \\
&= D D^{\vee} \cdot (k_{\infty} \ast k_{\infty})(g) = D D^{\vee} \cdot k_{\nu, \tau^{-1}}(g)
\end{align*}
where $D^{\vee}$ is defined on $C^{\infty}(G(\mathbb{R}))$ by $D^{\vee} \cdot \varphi (g) = \overline{D \cdot \varphi (g^{-1})}$, and we quote \cite[Lemma 5.11]{BM2} to bound
\begin{equation*}
	\Vert (D \cdot k_{\infty}) \ast (D \cdot k_{\infty})^{\vee} \Vert_{\infty} = \Vert D D^{\vee} \cdot k_{\nu, \tau^{-1}} \Vert_{\infty}   \ll_{D, \tau} (1 + \Vert \nu \Vert)^{2 d} \widetilde{\beta_S}(\nu).
\end{equation*}
\end{proof}
\begin{rmrk}
In this proof, we bounded $\mathbf{1}_K$ by $\mathbf{1}_{K_f}$. Hence, our estimate is not sharp in $K$. In particular, the height $T$ in Lemma \ref{lem7.5.1} is overestimated and the logarithmic loss in Proposition \ref{prop7.0.1} is not optimal.
\end{rmrk}
\subsection{Uniform Rapid decay estimate} \label{ssseecc}
As explained in Section \ref{ss7.33} we can deduce, from a uniform moderate growth estimate, a uniform rapid decay estimate. The proof is adapted from \cite[Lemma I.2.10]{MoeglinWald}.
\begin{lemma} \label{lem7.4.1}
Let $f \in \mathcal{E}_{\cusp}(K,\nu,\tau)$ be $L^2$-normalized and $\alpha \in \Phi$ (the simple roots). For any $r > 0$ and any $g \in \mathfrak{S}_{\omega, t}$, we have
\begin{equation*}
	f (g) \ll (1 + \Vert \nu \Vert)^{r} \widetilde{\beta_S}(\nu)^{1/2} \vol(K)^{-1/2} a_g^{\rho - r \alpha}
\end{equation*}
where $g = n_g m_g a_g k_g$.
\end{lemma}
\begin{proof}
	We fix a sequence
\begin{equation*}
	\lbrace 0 \rbrace = V_0 \subset V_1 \subset \ldots \subset V_{N-1} \subset V_N = U
\end{equation*}
of $\mathbb{Q}$-subgroups such that, for all $1 \leq i \leq N$, $V_i$ is normal in $U$ and $V_{i - 1} \backslash V_i \simeq \mathbb{G}_a$. Let $j_i : V_i \longrightarrow V_{i - 1} \backslash V_i$ denote the corresponding projection and set
\begin{equation*}
	\Gamma_i = V_i(\mathbb{A}_f) \cap gKg^{-1}.
\end{equation*}
In particular, there is a lattice $L_i \subset \mathbb{R}$ such that
\begin{equation*}
	L_i \backslash \mathbb{R} \simeq \mathbb{Q} \backslash \mathbb{A} / j_i(\Gamma_i)
\end{equation*}
with $\mathbb{Z} \backslash \mathbb{R} \simeq L_i \backslash \mathbb{R} \simeq V_{i-1}(\mathbb{A}) V_i(\mathbb{Q}) \backslash V_i (\mathbb{A}) / \Gamma_i$. More precisely, letting $\mathfrak{v}_i$ denote the Lie algebra of $V_i(\mathbb{R})$ and writting $\mathfrak{v}_i = \mathfrak{v}_{i-1} \oplus X_i \cdot \mathbb{R}$, then $x \mapsto \exp (x X_i)$ induces an isomorphism
\begin{equation*}
\mathbb{Z} \backslash \mathbb{R} \longrightarrow V_{i-1}(\mathbb{A}) V_i(\mathbb{Q}) \backslash V_i (\mathbb{A}) / \Gamma_i.
\end{equation*}
For $i = 0, \ldots, N$, set
\begin{equation*}
	\mathcal{F}_i (g) = \int_{[V_i]} f(ug) du
\end{equation*}	
so that $\mathcal{F}_0 = f$ and $\mathcal{F}_N = 0$ (recall $f$ is a cusp form). Moreover, we have
\begin{equation*}
	f(g) = \mathcal{F}_0(g) - \mathcal{F}_N(g) = \sum_{i = 1}^{N} \mathcal{F}_{i - 1}(g) - \mathcal{F}_{i}(g)
\end{equation*}
and the Lemma follows by proving the corresponding upper bound on each ${\mathcal{F}_{i - 1}(g) - \mathcal{F}_{i }(g)}$. Fix $1 \leq i \leq N$ and note that $u \mapsto \mathcal{F}_{i - 1} (ug)$ is right-$\Gamma_i$-invariant. Identifying
\begin{equation*}
\mathbb{Z} \backslash \mathbb{R} \simeq V_{i-1}(\mathbb{A}) V_i(\mathbb{Q}) \backslash V_i (\mathbb{A}) / \Gamma_i
\end{equation*}
we may consider its Fourier expansion
\begin{equation*}
 \mathcal{F}_{i - 1} (g) = \sum_{\xi \in \mathbb{Z}} \mathcal{F}_{i - 1}^{\xi}(g)
\end{equation*}
with
\begin{equation*}
\mathcal{F}_{i - 1}^{\xi}(g) = \int_{\mathbb{Z} \backslash \mathbb{R}} \mathcal{F}_{i - 1} \big( \exp(x X_i)g \big) \exp(-2i \pi x \xi) dx.
\end{equation*}
This way, using $\mathcal{F}_{i}(g) = \mathcal{F}_{i - 1}^{0}(g)$, we have
\begin{equation*}
\mathcal{F}_{i - 1}(g) - \mathcal{F}_{i }(g) = \sum_{\xi \neq 0} \mathcal{F}_{i - 1}^{\xi}(g)
\end{equation*}
and the Lemma follows by proving the corresponding upper bound on $\mathcal{F}_{i - 1}^{\xi}(g)$, $\xi \neq 0$. We have
\begin{align*}
\mathcal{F}_{i - 1}^{\xi}(g)
&= \int_{\mathbb{Z} \backslash \mathbb{R}} \mathcal{F}_{i - 1} \big( \exp(x X_i)g \big) \exp(-2i \pi x \xi) dx \\
&= \int_{\mathbb{Z} \backslash \mathbb{R}} \mathcal{F}_{i - 1} \big( g \exp(x X_i(g)) \big) \exp(-2i \pi x \xi) dx \\
&= (-2 i \pi \xi)^{-N} \int_{\mathbb{Z} \backslash \mathbb{R}}\mathcal{F}_{i - 1} \big( g \exp(x X_i(g)) \big)  \frac{\partial^N}{\partial x^N}  \exp(-2i \pi x \xi) dx
\end{align*}
where $X_i(g) = \text{Ad}(g^{-1}) X_i$. Then, fixing a non-zero $h \in \mathbb{N}$ and integrating by parts $h$ times yields
\begin{align*}
\mathcal{F}_{i - 1}^{\xi}(g) &= (2 i \pi \xi)^{-h} \int_{\mathbb{Z} \backslash \mathbb{R}} {X_i(g)}^h \cdot \mathcal{F}_{i - 1} \big( g \exp(x X_i(g)) \big) \exp(-2i \pi x \xi) dx \\
&= (2 i \pi \xi)^{-h} \int_{\mathbb{Z} \backslash \mathbb{R}} {X_i(g)}^h \cdot \mathcal{F}_{i - 1} \big( \exp(x X_i) g \big) \exp(-2i \pi x \xi) dx.
\end{align*}
As in the proof of \cite[Lemma I.2.10]{MoeglinWald}, we can show that there is $c_h > 0$ and $D_h$ of degree $h$ such that
\begin{equation*}
	{X_i(g)}^h \cdot \mathcal{F}_{i - 1} \big( \exp(x X_i) g \big) \leq c_h a_g^{- \alpha h} D_h \cdot \mathcal{F}_{i - 1}\big( \exp(x X_i) g \big).
\end{equation*}
Enlarging $\omega$ to ensure $\exp(x X_i) g \in \mathfrak{S}_{\omega, t}$ if necessary, we may apply Lemma \ref{lem7.3.1} to find
\begin{equation*}
	{X_i(g)}^h \cdot \mathcal{F}_{i - 1} \big( \exp(x X_i) g \big) \ll a_g^{- \alpha h} (1 + \Vert \nu \Vert)^{h} \widetilde{\beta_S}(\nu)^{1/2} \vol(K)^{-1/2} a_g^{\rho}
\end{equation*}
and this proves
\begin{equation*}
	f (g) \ll_r (1 + \Vert \nu \Vert)^{r} \widetilde{\beta_S}(\nu)^{1/2} \vol(K)^{-1/2} a_g^{\rho - r \alpha}
\end{equation*}
upon choosing $h \geq r$.
\end{proof}
\subsection{Control of the $L^2$-mass} \label{secaeaa}
We can now prove the statement on the control of the $L^2$-mass mentioned in Section \ref{ss7.33}. We split the proof in two Lemmas. 
\begin{lemma} \label{lem7.5.1}
	Let $f \in \mathcal{E}_{\cusp}(K,\nu,\tau)$ be $L^2$-normalized, let $c >0$ and $\epsilon > 0$. There is $C_{\epsilon} > 0$ such that, if $T > C_{\epsilon} (1 + \Vert \nu \Vert)^{1 + \epsilon} \vol (K)^{-( 1 + \epsilon) /2 }$ and if $g \in [G]^{> T}$, then $\vert f(g) \vert \leq c$.
\end{lemma}
\begin{proof}
	Let $g = n_g m_g a_g k_g \in [G]^{> T}$ and let $\alpha \in \Phi$ such that $a_g^{\alpha} > T$ and
\begin{equation*}	
a_g^{\alpha} \geq a_g^{\beta}
\end{equation*}
for any $\beta \in \Phi$. Let $R_1 > 0$ be such that
\begin{equation*}
a_g^{\rho - R_1 \alpha} \leq 1.
\end{equation*}
Let $R_2 > 1$ be such that
\begin{equation*}
	 (1 + \Vert \nu \Vert)^{R_2 \epsilon} \geq \widetilde{\beta_S}(\nu)^{1/2} (1 + \Vert \nu \Vert)^{R_1}.
\end{equation*}
Let $r = R_1 + R_2$, applying Lemma \ref{lem7.4.1} yields $C_r > 0$ such that
\begin{align*}
	\vert f (g) \vert &\leq C_r (1 + \Vert \nu \Vert)^{r} \widetilde{\beta_S}(\nu)^{1/2} \vol(K)^{-1/2} a_g^{\rho - r \alpha} \\
		&\leq C_r (1 + \Vert \nu \Vert)^{r} \widetilde{\beta_S}(\nu)^{1/2} \vol(K)^{-1/2} a_g^{-(R_2 + R_3) \alpha}.
\end{align*}
Since $a_g^{\alpha} > T$, we have
\begin{equation*}
a_g^{-R_2 \alpha} < T^{-R_2} < C_{\epsilon}^{-R_2} (1 + \Vert \nu \Vert)^{-R_2(1 + \epsilon)} \vol(K)^{R_2(1 + \epsilon)  / 2}.
\end{equation*}
and thus
\begin{align*}
	\vert f (g) \vert &\leq C_r (1 + \Vert \nu \Vert)^{r} \widetilde{\beta_S}(\nu)^{1/2} \vol(K)^{-1/2} a_g^{- R_2 \alpha} \\
	&\leq C_r C_{\epsilon}^{-R_2} \dfrac{(1 + \Vert \nu \Vert)^{R_1 + R_2}}{(1 + \Vert \nu \Vert)^{R_2 \epsilon + R_2}}\widetilde{\beta_S}(\nu)^{1/2} \dfrac{ \vol(K)^{R_2 (1 + \epsilon)/2}}{ \vol(K)^{1/2}}.
\end{align*}
Because $R_2 > 1$, we have $R_2(1+ \epsilon) - 1 > 0$ , hence
\begin{equation*}
\dfrac{ \vol(K)^{R_2 (1 + \epsilon)/2}}{ \vol(K)^{1/2}} = \vol (K)^{\frac{R_2 ( 1 + \epsilon) - 1}{2}} \leq \vol (K_f)^{\frac{R_2 ( 1 + \epsilon) - 1}{2}}
\end{equation*}
and therefore, by assumption on $R_2$, this yields
\begin{equation*}
	\vert f (g) \vert \leq  C_r C^{-R_2}_{\epsilon} \vol (K_f)^{\frac{R_2 ( 1 + \epsilon) - 1}{2}}.
\end{equation*}
Choosing $C_{\epsilon} > 0$ sufficiently large completes the proof.
\end{proof}
\begin{lemma} \label{lem7.5.2}
	Let $f \in \mathcal{E}_{\cusp}(K,\nu,\tau)$ be $L^2$-normalized, let $0 < c < 1$ and let $\epsilon > 0$. There is $C_{\epsilon} > 0$ such that, if $T > C_{\epsilon} (1 + \Vert \nu \Vert)^{1 + \epsilon} \vol (K)^{-( 1 + \epsilon) /2 }$, then
\begin{equation*}
	\int_{[G]^{\leq T}} \vert f(g) \vert^2 dg \geq c.
\end{equation*}
\end{lemma}
\begin{proof}
	We apply Lemma \ref{lem7.5.1} with the given $\epsilon > 0$ and $c = 1$. We obtain $C_{\epsilon} > 0$ such that, for $T > C_{\epsilon} (1 + \Vert \nu \Vert)^{1 + \epsilon} \vol (K)^{-( 1 + \epsilon) /2 }$, then
\begin{equation*}
	\vert f(g) \vert < 1 \quad (g \in [G]^{> T}).
\end{equation*}	
In particular
\begin{equation*}
	\int_{[G]^{> T}} \vert f(g) \vert^2 dg \leq \vol ([G]^{> T})
\end{equation*}
and, because $\vol ([G]^{> T}) \rightarrow 0$ as $T \rightarrow \infty$, making $C_{\epsilon}$ sufficiently large gives the result.
\end{proof}
\subsection{Upper bound} \label{secc7.6}
We are now ready to prove Proposition \ref{prop7.0.1}.

Let $\tau : K_{\infty} \rightarrow \mathbb{C}^{\times}$ be a continuous character. Let $k_{\nu, \tau^{-1}}^0$ be the test function on either $\SU(m,m)$ or $\Sp_{2m}(\mathbb{R})$ appearing in the proof of Proposition \ref{prop3.7.1}; not the $\tau^{-1}$-spherical test function $k_{\nu, \tau^{-1}}$ itself, which is the convolution $k_{\nu, \tau^{-1}}= k_{\nu, \tau^{-1}}^0 \ast k_{\nu, \tau^{-1}}^0$.\\
As explained in Section \ref{sec6.1} (see also the paragraph following Proposition \ref{prop6.0.1}), in the case ${G(\mathbb{R}) = \U(m,m)}$, we can extend $k_{\nu, \tau^{-1}}^0$ to a smooth and compactly supported function $k_{\infty}$ on $G(\mathbb{R})$.

Let $\mathcal{B}_{\cusp}(K,\nu, \tau)$ be an orthonormal basis of $\mathcal{E}_{\cusp}(K,\nu,\tau)$. Lemma \ref{lem7.2.1} applied to the orthonormal set $\lbrace f_{\lambda} \rbrace_{\lambda} = \mathcal{B}_{\cusp}(K,\nu, \tau)$ with test function $k = \vol(K)^{-1} \mathbf{1}_{K} \otimes k_{\infty}$ yields
\begin{equation} \label{eq7.2}
	\sum_{\Vert \im(\lambda) - \nu \Vert \leq 1} \vert R(k) f_{\lambda}(g) \vert^2 \leq \sum_{\gamma \in G(\mathbb{Q})} k \ast k^{\vee}(g^{-1} \gamma g).
\end{equation} 
Scaling $k_{\infty}$ by a bounded constant, we have $\vert R(k) f_{\lambda}(g) \vert^2 = \vert c_{\lambda} f_{\lambda} \vert^2$ with $\vert c_{\lambda} \vert^2 \geq 1$ since $\Vert \im (\lambda) - \nu \Vert \leq 1$ (see Proposition \ref{prop3.7.1}). Arguing as in the proof of Lemma \ref{lem7.3.1} (with $X = 1$ of degree $0$), we can bound
\begin{equation*}
	 \sum_{\gamma \in G(\mathbb{Q})} k \ast k^{\vee}(g^{-1} \gamma g) \ll \vol(K)^{-1} \widetilde{\beta_S}(\nu) a_g^{2 \rho}.
\end{equation*}
Let $\epsilon > 0$, let $c = 1/2$ and apply Lemma \ref{lem7.5.2} to obtain $C_{\epsilon} > 0$ such that
\begin{equation*}
	\int_{[G]^{\leq T}} \vert f_{\lambda} (g) \vert^2 dg \geq 1/2
\end{equation*}
for $T > C_{\epsilon} (1 + \Vert \nu \Vert)^{1 + \epsilon} \vol (K)^{-( 1 + \epsilon) /2 }$. This way, integrating \eqref{eq7.2} over $[G]^{\leq T}$ yields
\begin{equation*}
	\dim \mathcal{E}_{\cusp}(K, \nu, \tau) \ll \vol(K)^{-1} \widetilde{\beta_S}(\nu) \int_{[G]^{\leq T}} a_g^{a \rho} dg.
\end{equation*}
Using Iwasawa coordinates \eqref{eq7.1}, we obtain
\begin{equation*}
\int_{[G]^{\leq T}} a_g^{2 \rho} dg \ll \int_{A_0(t, T)} da \ll_{\epsilon} \log \big( (1 + \Vert \nu \Vert) \vol (K)^{-1/2} \big).
\end{equation*}
and this proves, upon inserting $\vol(K)^{-1} \asymp \vol(X(K))$, Proposition \ref{prop7.0.1} .

\section{Proof of the main result} \label{chap8}
In this Section, we prove the main result of this article, that is, Theorem \ref{theo1.2.2} (and \ref{theo1.2.1}). Let $F$ be a totally real number field and let $E$ be $F$ itself or a totally imaginary quadratic extension of $F$. Let $n > m \geq 1$ be integers, let $V$ be a non-degenerate Hermitian $E$-space of dimension $n + m$ and $W$ be a non-degenerate split skew-Hermitian $E$-space of dimension $2m$. We assume there is an archimedean place $v_0$ of $F$ such that $V$ has signature $(n,m)$ at $v_0$ and is positive definite at every other real places $v \neq v_0$. We set $G = \U(V)$ and $G' = \U(W)$ for the associated unitary groups.

Let $\mathcal{R}_0$ be the finite set of finite places at which either $E/F$ or $V$ is ramified and let $\mathcal{R}$ denote the union of $\mathcal{R}_0$ with the places above $2$. For any $v \in \mathcal{R}$, we choose open compact subgroups $K_v$ of $G(F_v)$. For any archimedean place $v$, we fix maximal open compact subgroups $K_v$ and $K_v'$ of $G(F_v)$ and $G'(F_v)$ respectively. For a given finite place $v \notin \mathcal{R}$, we let $K_v$ (resp. $K_v'$) denote the stabilizer of a self-dual lattice $L_{V,v}$ in $V_v$ (resp. $L_{W,v}$ in $W_v$). We then define
\begin{equation*}
	K_f = \prod_{v \nmid \infty} K_v \quad \text{,} \quad K'_f = \prod_{v \nmid \infty} K'_v
\end{equation*}
and
\begin{equation*}
	K_{\infty} = \prod_{v \mid \infty} K_v \quad \text{,} \quad K'_{\infty} = \prod_{v \mid \infty} K'_v.
\end{equation*}
Let $\mathfrak{n}$ be an ideal of $\mathcal{O}$, the ring of integers of $F$, and assume $\mathfrak{n}$ is prime to $\mathcal{R}$. If $v \mid \mathfrak{n}$, we set
\begin{equation*}
	K_v(\mathfrak{n}) = \lbrace g \in G(F_v) : (g- \id) L_{V,v} \subseteq \mathfrak{n} L_{V,v} \rbrace.
\end{equation*}
We define the principal congruence subgroup of level $\mathfrak{n}$ as the product
\begin{equation*}
	K(\mathfrak{n}) = \prod_{v \mid \mathfrak{n}} K_v(\mathfrak{n}) \prod_{v \nmid \mathfrak{n} \infty} K_v.
\end{equation*}
Similarly, we define
\begin{equation*}
	K'(\mathfrak{n}) = \prod_{v \mid \mathfrak{n}} K'_v(\mathfrak{n}) \prod_{v \nmid \mathfrak{n} \infty} K'_v
\end{equation*}
where $K'_v$, for $v \in \mathcal{R}$, are chosen before Lemma \ref{lemmm}.

As $W$ is split, we can fix a maximal isotropic subspace $X \subset W$ and write $W = X \oplus X^*$ for some $X^* \subset W$ in perfect duality with $X$. Let $q$ be a non-degenerate Hermitian form on $X$ such that $\U(X)(F_{v_0}) = \U(0,m)$ (or $\U(X)(F_{v_0}) = \Ortho(0,m)$) and $\U(X)(F_{v}) = \U(m,0)$ (or $\U(X)(F_{v}) = \Ortho(m,0)$) at every other archimedean places $v \neq v_0$. Let $j : X \hookrightarrow V$ be an embedding of Hermitian (or quadratic) spaces and set
\begin{equation*}
H_j = \U(j(X)) \times \U(j(X)^{\perp}).
\end{equation*}
We now choose and fix a congruence subgroup $K_{H_j}(\mathfrak{n})$ of level $\mathfrak{n}$, that is containing $K(\mathfrak{n})$ and not containing any $K(\mathfrak{d})$ ($\mathfrak{d} \vert \mathfrak{n}$), satisfying
\begin{equation*}
	K_{H_j}(\mathfrak{n}) \cap H_j(\mathbb{A}_f) = K^{H_j}_{f}.
\end{equation*}
Finally, we set
\begin{equation*}
X_{H_j}(\mathfrak{n}) = G(F) \backslash G(\mathbb{A}) / K_{H_j}(\mathfrak{n}) K_{\infty}
\end{equation*}
and
\begin{equation*}
X'(\mathfrak{n}) = G'(F) \backslash G'(\mathbb{A}) / K'(\mathfrak{n}) K'_{\infty}.
\end{equation*}
\subsection{Automorphic forms}
We let $\mathfrak{g} = \mathfrak{k} \oplus \mathfrak{p}$ be the Cartan decomposition of the Lie algebra of $G(F_{v_0})$ with respect to the Lie algebra $\mathfrak{k}$ of $K_{v_0}$ and we fix a maximal abelian subspace $\mathfrak{a} \subseteq \mathfrak{p}$. Let $Q \geq 1$, $\nu \in i \mathfrak{a}^*$ and let $\mathfrak{n} \subset \mathcal{O}$ be an ideal prime to $\mathcal{R}$. We define $\mathcal{A}^G(K_{H_j}(\mathfrak{n}),\nu, Q)$ as the set of (isomorphism classes of) irreducible automorphic unitary representations $\pi = \pi_f \otimes \pi_{\infty}$ of $G$ such that
\begin{itemize}
\item $\pi_f^{K_{H_j}(\mathfrak{n})} \neq 0$;
\item $\pi_{\infty} = \otimes_v \pi_v$ with $\pi_v$ the trivial representation of $G(F_v)$ for all $v \neq v_0$ and $\pi_{v_0}$ the unique spherical representation of $G(F_{v_0})$ of spectral parameter $\lambda \in \mathfrak{a}^*_{\mathbb{C}}$ satisfying $\Vert \im(\lambda) - \nu \Vert \leq Q$ (see Lemma \ref{lem3.3.1} if $E \neq F$ and remark \ref{rmrkOrtho} if $E = F$).
\end{itemize}
Then, we let $\widetilde{\mathcal{A}}^G(K_{H_j}(\mathfrak{n}),\nu, Q)$ denote the set of pairs $(\pi, E)$ where $\pi \in \mathcal{A}^G(K_{H_j}(\mathfrak{n}),\nu, Q)$ and $E$ is a morphism from $\pi$ to the space of automorphic forms on $G$. We denote by $m(\pi,G)$ the multiplicity space of $\pi$ \textit{i.e.} the fiber of the map $\widetilde{\mathcal{A}}^G(K_{H_j}(\mathfrak{n}),\nu, Q) \rightarrow \mathcal{A}^G(K_{H_j}(\mathfrak{n}),\nu, Q)$ above $\pi$. We let $\mathcal{E}^G(K_{H_j}(\mathfrak{n}),\nu,Q)$ denote the associated space of automorphic forms on $G$, namely
\begin{equation*}
	\mathcal{E}^G(K_{H_j}(\mathfrak{n}),\nu,Q) = \bigoplus_{\pi \in \mathcal{A}^G(K_{H_j}(\mathfrak{n}),\nu, Q)} m(\pi, G) \otimes \pi_f^{K_{H_j}(\mathfrak{n})}.
\end{equation*}
Then, set
\begin{equation*}
	\tau = \prod_{v \mid \infty} \tau_v : \prod_{v \mid \infty} K'_v \longrightarrow \mathbb{C}^{\times}
\end{equation*}
for the continuous character of $K'_{\infty}$ given by
\begin{itemize}
\item $\tau_v = \det^{(n+m)/2}$ if $v \neq v_0$ and $E = F$;
\item $\tau_{v_0} = \det^{(n-m)/2}$ if $E = F$;
\item $\tau_v = \det^{(n + m)/2} \otimes \det^{-(n+m)/2}$ if $v \neq v_0$ and $E \neq F$;
\item $\tau_{v_0} = \det^{(n-m)/2} \otimes \det^{(m-n)/2}$ if $E \neq F$.
\end{itemize}
We let $\mathfrak{g}' = \mathfrak{k}' \oplus \mathfrak{p}'$ be the Cartan decomposition of the Lie algebra of $G'(F_{v_0})$ with respect to the Lie algebra $\mathfrak{k}'$ of $K'_{v_0}$ and we fix a maximal abelian subspace $\mathfrak{a}' \subseteq \mathfrak{p}'$. Note that, as $G$ and $G'$ have same rank, we may identify $\mathfrak{a}^*_{\mathbb{C}}$ with ${{\mathfrak{a}'}^*_{\mathbb{C}}}$ and this allows us to view a spectral parameter simultaneously for $G$ and $G'$. 

Let $\mathcal{A}_{\cusp}^{G'}(K'(\mathfrak{n}), \nu, \tau, Q)$ be the set of (isomorphism classes of) irreducible cuspidal automorphic unitary representations $\pi' = \pi'_f \otimes \pi'_{\infty}$ of $G'$ such that
\begin{itemize}
\item ${\pi'_f}^{K'(\mathfrak{n})} \neq 0$;
\item $\pi'_{\infty} = \otimes_v \pi'_v$ with $\pi'_v$ the unique irreducible $\tau_v$-spherical representation of $G'(F_v)$ of spectral parameter $(\frac{n}{2} - 1, \ldots, \frac{n}{2} - m)$ if $E = F$ or $(\frac{3m + n - 1}{2}, \ldots, \frac{n - m + 1}{2})$ if $E \neq F$ for all ${v \neq v_0}$ and $\pi_{v_0}$ the unique $\tau_{v_0}$-spherical representation of $G'(F_{v_0})$ of spectral parameter $\lambda \in {\mathfrak{a}'}^*_{\mathbb{C}}$ satisfying $\Vert \im(\lambda) - \nu \Vert \leq Q$ (see Corollary \ref{coro3.3.2}).
\end{itemize}
We define $m(\pi', G')$ similarly as $m(\pi,G)$ and let $\mathcal{E}_{\cusp}^{G'}(K'(\mathfrak{n}),\nu, \tau,Q)$ denote the associated space of cuspidal automorphic forms on $G'$, that is
\begin{equation*}
	\mathcal{E}_{\cusp}^{G'}(K'(\mathfrak{n}),\nu, \tau,Q) = \bigoplus_{\pi' \in \mathcal{A}_{\cusp}^{G'}(K'(\mathfrak{n}),\nu, \tau, Q)} m(\pi', G') \otimes {\pi_f'}^{K'(\mathfrak{n})}.
\end{equation*}
\subsection{Mean square estimate}
We apply the result of Section \ref{chap6} to show that the sum of the squared $H_j$-periods, over a $L^2$-normalized with respect to the probability measure on $X_{H_j}(\mathfrak{n})$ basis of $\mathcal{E}^G(K_{H_j}(\mathfrak{n}),\nu,Q)$, is $\asymp \vol (X_{H_j}(\mathfrak{n})) \beta_S(\nu)$.

Let $Q > 1$, $\nu \in i \mathfrak{a}^*$ and let $\mathfrak{n} \subset \mathcal{O}$ be an ideal prime to $\mathcal{R}$. We denote by $\mathcal{B}(\mathfrak{n},\nu,Q)$ a basis of $\mathcal{E}^G(K_{H_j}(\mathfrak{n}),\nu,Q)$ of $L^2$-normalized (with respect to the probability measure on $X_{H_j}(\mathfrak{n})$) Maass forms and we define
\begin{equation*}
	\mathcal{M}_{H_j}({\mathfrak{n}},\nu,Q) = \sum_{f_{\lambda} \in \mathcal{B}(\mathfrak{n},\nu,Q)} \left\lvert \mathcal{P}_{H_j} (f_{\lambda}) \right\rvert^2.
\end{equation*}
\begin{lemma} \label{lem8.2.1}
There is $Q > 1$ (independent of $\mathfrak{n}$ and $\nu$) such that, for any sufficiently regular $\nu \in i \mathfrak{a}^*$ (with sufficiently large norm independently of $\mathfrak{n}$) we have
\begin{equation*}
	\mathcal{M}_{H_j}({\mathfrak{n}},\nu,Q) \asymp \vol (X_{H_j}(\mathfrak{n})) \beta_S(\nu).
\end{equation*}
\end{lemma}
\begin{proof}
Consider the finite (see \cite[Theorem 5.1]{BorelFiniteness}) set $\gen_G(\mathfrak{n}) = G(F) \backslash G(\mathbb{A}_f) / K_{H_j}(\mathfrak{n})$ and the identification
\begin{equation*}
	X_{H_j}(\mathfrak{n}) = \bigcup_{g \in \gen_G(\mathfrak{n})} \Gamma^g_{H_j}(\mathfrak{n})  \backslash S
\end{equation*}
where $\Gamma^g_{H_j}(\mathfrak{n}) = G(F) \cap g K_{H_j}(\mathfrak{n}) g^{-1}$. In particular, if
\begin{equation*}
	\mathcal{E} (\Gamma^g_{H_j}(\mathfrak{n}) \backslash S ; \lambda) = \lbrace f \in L^2(\Gamma^g_{H_j}(\mathfrak{n})  \backslash S) : D \cdot f = \chi_{\lambda} (D) f \text{, } \forall D \in \mathcal{D}(S) \rbrace
\end{equation*}
then
\begin{equation*}
	\mathcal{E}^{G}(K_{H_j}(\mathfrak{n}),\nu,Q) = \sum_{g \in \gen_G(\mathfrak{n})} \sum_{\lambda \in \Lambda(\Gamma^g_{H_j}(\mathfrak{n}) , \nu, Q)} \mathcal{E}(\Gamma^g_{H_j}(\mathfrak{n})  \backslash S ; \lambda)
\end{equation*}
where
\begin{equation*}
\Lambda(\Gamma^g_{H_j}(\mathfrak{n}), \nu, Q) = \lbrace \lambda \in \mathfrak{a}^*_{\mathbb{C}} : \dim \mathcal{E}_{\lambda} (\Gamma^g_{H_j}(\mathfrak{n}) \backslash S)  > 0 \text{ and } \Vert \im(\lambda) - \nu \Vert \leq Q \rbrace.
\end{equation*}
This way, we can write the basis $\mathcal{B}(\mathfrak{n},\nu,Q)$ as a union (over $g \in \gen_G(\mathfrak{n})$) of bases $\mathcal{B}_g (\mathfrak{n},\nu,Q) $ of
\begin{equation*}
\sum_{\lambda \in \Lambda(\Gamma^g_{H_j}(\mathfrak{n}), \nu, Q)} \mathcal{E}(\Gamma^g_{H_j}(\mathfrak{n}) \backslash S, \lambda).
\end{equation*}
Hence, $\mathcal{M}_{H_j}({\mathfrak{n}},\nu,Q)$ may be written as a classical (\textit{i.e.} non adelic) sum (see Section \ref{sec5.4})
\begin{equation*}
\sum_{g \in \gen_G(\mathfrak{n})} \sum_{f_{\lambda} \in \mathcal{B}_g(\mathfrak{n},\nu,Q)} \left\lvert \sum_{p \in \mathcal{H}^g} w_p f_{\lambda}(p) \right\rvert^2
\end{equation*}
Note that the number of points in $\mathcal{H}^g$ is independent of both $\mathfrak{n}$ and $\lambda$. Similarly, the weights $w_p$ are independent of $\mathfrak{n}$ and $\lambda$. We can now quote Proposition \ref{prop6.0.1} to obtain that the contribution of each $g \in \gen_G(\mathfrak{n})$ to $\mathcal{M}({\mathfrak{n}},\nu,Q)$ is $\asymp \vol(\Gamma^g_{H_j}(\mathfrak{n}) \backslash S) \beta_S(\nu)$ and hence
\begin{equation*}
\mathcal{M}_{H_j}({\mathfrak{n}},\nu,Q) \asymp \sum_{g \in \gen_G(\mathfrak{n})} \vol(\Gamma^g_{H_j}(\mathfrak{n}) \backslash S) \beta_S(\nu)  = \vol (X_{H_j}(\mathfrak{n})) \beta_S(\nu).
\end{equation*}
\end{proof}
\subsection{Distinction} \label{secc8.2}
We now use the period relation developed in Section \ref{chap5} to determine which forms contribute to $\mathcal{M}_{H_j}({\mathfrak{n}},\nu,Q)$.

Let $Q > 1$ and assume $\nu \in i \mathfrak{a}^*$ is $Q$-regular. For $\pi \in \mathcal{A}^G(K_{H_j}(\mathfrak{n}), \nu, Q)$, we denote by $\Theta(\pi) = \Theta(\pi; W)$ the span of the
\begin{equation*}
	\theta(f ; s) : g' \mapsto \int_{[G]} f(g) \Theta(g,g';s) dg
\end{equation*}
for $f \in \pi$ and $s \in \omega$ (see Section \ref{subsec4.2}). We now choose $K'_v$, for $v \in \mathcal{R}$, small enough as in \cite[Proposition 2.11]{Cossutta}.
\begin{lemma} \label{lemmm}
If $\pi \in \mathcal{A}^G(K_{H_j}(\mathfrak{n}), \nu, Q)$ and if $\Theta(\pi) \neq 0$, then $\Theta(\pi) = \otimes_v \theta(\pi_v)$ and $\Theta(\pi) \in \mathcal{A}_{\cusp}^{G'}(K'(\mathfrak{n}), \nu, \tau, Q)$.
\end{lemma}
\begin{proof}
We quote Corollary \ref{coro6.4.3} to obtain the factorization property (hence the irreducibility of $\Theta(\pi)$) and the fact $\Theta(\pi)$ is cuspidal. Here, we used the assumption $\nu$ is $Q$-regular as follows: if $\Vert \im(\lambda) - \nu \Vert \leq Q$, then $\lambda \in i \mathfrak{a}^*$ and hence $\pi_{v_0}$ is tempered.\\
We then argue place by place to prove $\Theta(\pi) \in \mathcal{A}_{\cusp}^{G'}(K'(\mathfrak{n}), \nu, \tau, Q)$: see Section \ref{sec3.7} for the archimedean places and Section \ref{sec2.5} for the finite places (and also \cite[Proposition 2.11]{Cossutta} for places $v \in \mathcal{R}$).
\end{proof}
Let $C$ denote the compact subset $\Vert \im(\lambda) - \nu \Vert \leq Q$ of $\mathfrak{a}^*_{\mathbb{C}}$ and consider the test function $s_{K_{H_j}(\mathfrak{n})} \otimes s_{\nu, C}$ constructed in the proof of Lemma \ref{lem5.3.1}. Setting
\begin{equation*}
s_{\mathfrak{n}, \nu, Q} =  s_{K_{H_j}(\mathfrak{n})} \otimes s_{\nu, C}
\end{equation*}
and applying Lemma \ref{lem5.3.1}, we obtain
\begin{equation} \label{eqqqqat}
\mathcal{P}_{H_j} ({f_{\lambda}}) \neq 0 \Rightarrow \theta (f_{\lambda} ; s_{\mathfrak{n}, \nu, Q}) \neq 0
\end{equation}
for any $f_{\lambda} \in \mathcal{E}^G(K_{H_j}(\mathfrak{n}), \nu, Q)$. We then define a linear map $\Theta_{\mathfrak{n}, \nu, Q}$ on $\mathcal{E}_{\cusp}^{G'}(K'(\mathfrak{n}), \nu, \tau, Q)$ by
\begin{equation*}
\Theta_{\mathfrak{n}, \nu, Q} (f') : g \mapsto \int_{[G']} f'(g') \overline{\Theta(g,g';s_{\mathfrak{n}, \nu, Q})} dg'.
\end{equation*}
The adjoint property \eqref{eqadj} thus reads
\begin{equation*}
\langle f, \Theta_{\mathfrak{n}, \nu, Q}(f') \rangle_G = \langle \theta(f; s_{\mathfrak{n}, \nu, Q}) , f' \rangle_{G'}.
\end{equation*}
\begin{lemma}
The image of $\Theta_{\mathfrak{n}, \nu, Q}$ lies in $\mathcal{E}^{G}(K(\mathfrak{n}), \nu, Q)$.
\end{lemma}
\begin{proof}
Let $f' \in \pi'$ with $\pi' \in \mathcal{A}_{\cusp}^{G'} (K'(\mathfrak{n}), \nu, \tau, Q)$ and assume $\Theta_{\mathfrak{n}, \nu, Q} (f') \neq 0$; in particular $\Theta(\pi') \neq 0$. Since $G$ is anisotropic, we can quote Theorem \ref{theo6.4.1} to deduce $\Theta(\pi')$ is irreducible and $\Theta(\pi') = \otimes_v \theta(\pi'_v)$. From local considerations (similarly as in Lemma \ref{lemmm}), we deduce $\Theta(\pi') \in \mathcal{A}^{G}(K(\mathfrak{n}), \nu, Q)$ and thus ${\Theta_{\mathfrak{n}, \nu, Q} (f') \in \mathcal{E}^{G}(K(\mathfrak{n}), \nu, Q)}$.
\end{proof}
Since $\mathcal{E}^{G}(K_{H_j}(\mathfrak{n}), \nu, Q) \subset \mathcal{E}^{G}(K(\mathfrak{n}), \nu, Q)$, we can consider the intersection\footnote[1]{This relies on $K(\mathfrak{n}) \subset K_{H_j}(\mathfrak{n})$ but is independent of the assumption $K_{H_j}(\mathfrak{n}) \cap H_j(\mathbb{A}_f) = K^{H_j}_{f}$. We refer to remark  \ref{remarquebrumley} for further comments.} of the image of $\Theta_{\mathfrak{n}, \nu, Q}$ with $\mathcal{E}^{G}(K_{H_j}(\mathfrak{n}), \nu, Q)$. We denote this intersection by $\mathcal{E}^{G, \Theta}(K_{H_j}(\mathfrak{n}), \nu, Q)$.
\begin{lemma} \label{lem8.2.4}
Let $f_{\lambda} \in \mathcal{E}^{G}(K_{H_j}(\mathfrak{n}), \nu, Q)$ and assume $\mathcal{P}_{H_j}(f_{\lambda}) \neq 0$. Then $f_{\lambda}$ lies in $\mathcal{E}^{G, \Theta}(K_{H_j}(\mathfrak{n}), \nu, Q)$.
\end{lemma}
\begin{proof}
From \eqref{eqqqqat} we know that $\mathcal{P}_{H_j}(f_{\lambda}) \neq 0$ ensures ${\theta(f_{\lambda} ; s_{\mathfrak{n}, \nu, Q}) \neq 0}$ for any ${f_{\lambda} \in \mathcal{E}^G(K_{H_j}(\mathfrak{n}), \nu, Q)}$. Moreover, by Lemma \ref{lemmm}, we have $\theta(f_{\lambda} ; s_{\mathfrak{n}, \nu, Q}) \in \mathcal{E}_{\cusp}^{G'}(K'(\mathfrak{n}), \nu, \tau, Q)$. From the adjoint property, we deduce
\begin{equation*}
0 \neq \langle \theta(f_{\lambda} ; s_{\mathfrak{n}, \nu, Q}) , \theta(f_{\lambda} ; s_{\mathfrak{n}, \nu, Q}) \rangle_{G'} = \langle f_{\lambda} , \Theta_{\mathfrak{n}, \nu, Q}( \theta(f_{\lambda} ; s_{\mathfrak{n}, \nu, Q})) \rangle_G
\end{equation*}
and this proves $f_{\lambda} \in \mathcal{E}^{G, \Theta}(K_{H_j}(\mathfrak{n}), \nu, Q)$.
\end{proof}
\subsection{Conclusion}
We are now ready to conclude the proof of Theorem \ref{theo1.2.2} (and \ref{theo1.2.1}).

For $\lambda$ such that $\Vert \im(\lambda) - \nu \Vert \leq Q$, we define
\begin{equation*}
V_{H_j}(\mathfrak{n}, \lambda) = m(\pi, G) \otimes \pi_f^{K_{H_j}(\mathfrak{n})}
\end{equation*}
where $\pi = \pi_f \otimes \pi_{\infty} \in \mathcal{A}^G(K_{H_j}(\mathfrak{n}),\nu, Q)$ with $\pi_{\infty} = \pi_{v_0, \lambda} \otimes \bigotimes_{v \neq v_0} \mathbf{1}_v$ and ${\Vert \im(\lambda) - \nu \Vert \leq Q}$. Similarly, let
\begin{equation*}
V'(\mathfrak{n}, \lambda) = m(\pi', G') \otimes {\pi'_f}^{K'(\mathfrak{n})}
\end{equation*}
where $\pi' = \pi'_f \otimes \pi'_{\infty} \in \mathcal{A}_{\cusp}^{G'}(K'(\mathfrak{n}),\nu, \tau, Q)$ with $\pi'_{\infty} = \pi'_{v_0, \lambda, \tau_{v_0}} \otimes \bigotimes_{v \neq v_0} \pi'_{v, \tau_{v}}$ and ${\Vert \im(\lambda) - \nu \Vert \leq Q}$.

Note the automorphic period $\mathcal{P}_{H_j}$ defines a linear functional on each finite dimensional vector space $V_{H_j}(\mathfrak{n}, \lambda)$. In particular, the kernel of $\mathcal{P}_{H_j \vert V_{H_j}(\mathfrak{n}, \lambda)}$ is of codimension at most $1$.\\
Consider the following composition map
\begin{equation*}
\Theta^{H_j}_{\mathfrak{n}, \nu, Q} : \mathcal{E}_{\cusp}^{G'}(K'(\mathfrak{n}), \nu, \tau, Q) \overset{\Theta_{\mathfrak{n}, \nu, Q}}{\longrightarrow} \mathcal{E}^{G}(K(\mathfrak{n}), \nu, Q) \overset{\text{proj}}{\longrightarrow} \mathcal{E}^{G, \Theta}(K_{H_j}(\mathfrak{n}), \nu, Q).
\end{equation*}
For $\lambda$ such that $\Vert \im(\lambda) - \nu \Vert \leq Q$, if $f' \in V'(\mathfrak{n}, \lambda)$, then $f_{\mathfrak{n}, \lambda} = \Theta^{H_j}_{\mathfrak{n}, \nu, Q}(f')$ (if non-zero) lies in $V_{H_j}(\mathfrak{n}, \lambda)$ and satisfies $\mathcal{P}_{H_j}(f_{\mathfrak{n}, \lambda}) \neq 0$. Normalizing $f_{\mathfrak{n}, \lambda}$ if necessary, $f_{\mathfrak{n}, \lambda}$ is a unit normal vector to the kernel of $\mathcal{P}_{H_j \vert V_{H_j}(\mathfrak{n}, \lambda)}$. We can then complete this set of vectors to obtain an orthonormal basis of $\mathcal{E}^{G, \Theta}(K_{H_j}(\mathfrak{n}), \nu, Q)$ denoted $\mathcal{B}_{H_j}^{\Theta}(\mathfrak{n},\nu,Q)$.

By construction of this basis and by Lemma \ref{lem8.2.4}, we have
\begin{equation*}
	\mathcal{M}_{H_j}({\mathfrak{n}},\nu,Q) = \sum_{f_{\lambda} \in \mathcal{B}_{H_j}^{\Theta}(\mathfrak{n},\nu,Q)} \left\lvert \mathcal{P}_{H_j} (f_{\lambda}) \right\rvert^2 = \sum_{\lambda \text{ : } \Vert \im(\lambda) - \nu \Vert \leq Q} \left\lvert \mathcal{P}_{H_j} (f_{\mathfrak{n}, \lambda}) \right\rvert^2.
\end{equation*}
Inserting Lemma \ref{lem8.2.1}, we obtain
\begin{equation} \label{eq8.2bis}
	\sum_{\lambda \text{ : } \Vert \im(\lambda) - \nu \Vert \leq Q} \left\lvert \mathcal{P}_{H_j} (f_{\mathfrak{n}, \lambda}) \right\rvert^2 \asymp \vol (X_{H_j}(\mathfrak{n})) \beta_S(\nu).
\end{equation}
Now, by definition, we have
\begin{equation*}
\dim \mathcal{E}_{\cusp}^{G'}(K'(\mathfrak{n}) , \nu, Q, \tau) = \sum_{\lambda \text{ : } \Vert \im(\lambda) - \nu \Vert \leq Q} \dim V'(\mathfrak{n}, \lambda)
\end{equation*}
and can quote Proposition \ref{prop7.0.1} (with $\mathbb{Q}$-group $\text{Res}_{F/\mathbb{Q}} G'$) to bound
\begin{equation} \label{eq8.2}
\sum_{\lambda \text{ : } \Vert \im(\lambda) - \nu \Vert \leq Q} \dim V'(\mathfrak{n}, \lambda) \ll  \vol(X'(\mathfrak{n})) {\beta_{S'}}(\nu) \mathcal{L} (\mathfrak{n}, \nu)^{m [F : \mathbb{Q}]}.
\end{equation}

We are now ready to prove the main result of this article, namely Theorem \ref{theo1.2.2} (and \ref{theo1.2.1}). Let $Q$ and $\nu$ be as in lemma \ref{lem8.2.1} and assume moreover $\nu$ is $Q$-regular. Comparing \eqref{eq8.2bis} with \eqref{eq8.2}, there is $\lambda$ such that $\Vert \im(\lambda) - \nu \Vert \leq Q$ and
\begin{equation*}
\vert \mathcal{P}_{H_j} (f_{\mathfrak{n}, \lambda}) \vert \gg \dfrac{\vol (X_{H_j}(\mathfrak{n}))^{\frac{1}{2}}}{\vol (X'(\mathfrak{n}))^{\frac{1}{2}}} \dfrac{\beta_S(\nu)^{\frac{1}{2}}}{\beta_{S'}(\nu)^{\frac{1}{2}}} \dim V'(\mathfrak{n}, \lambda)^{\frac{1}{2}} \mathcal{L} (\mathfrak{n}, \nu)^{-\frac{m [F : \mathbb{Q}]}{2}}.
\end{equation*}
In particular, since $\mathcal{P}_{H_j} (f_{\mathfrak{n}, \lambda})$ is a weighted sum of point evaluations with both the weights and the number of terms independent of both $\mathfrak{n}$ and $\nu$, this implies the same lower bound on $\Vert f_{\mathfrak{n}, \lambda} \Vert_{\infty}$.

\newpage
\appendix
\section{Appendix. Order of finite unitary groups} \label{appendixA}
\subsection*{Symplectic group}
Let $p \geq 3$ be an integer. Let $F$ be a $p$-adic field with ring of integers $\mathcal{O}$, maximal ideal $\mathfrak{p}$ and let $q$ denote the cardinal of the residual field $\mathcal{O} / \mathfrak{p}$.
\begin{lemma} \label{lemA01}
	Let $r, \vartheta \geq 1$ be integers, then
\begin{equation*}
	\vert \Sp_{2r} (\mathcal{O} / \mathfrak{p}^{\vartheta}) \vert = q^{\vartheta (2r^2 + r)} \prod_{i = 1}^r (1 - q^{-2i}).
\end{equation*}
\end{lemma}
\begin{proof}
See \cite[Theorem VII.27]{Newmann}.
\end{proof}
\subsection*{Orthogonal group}
Let $p \geq 3$ be an integer. Let $F$ be a $p$-adic field with ring of integers $\mathcal{O}$, maximal ideal $\mathfrak{p}$ and let $q$ denote the cardinal of the residual field $\mathcal{O} / \mathfrak{p}$.
\begin{lemma} \label{lemA02}
Let $d, \vartheta \geq 1$ be integers, then
\begin{equation*}
\vert \Ortho_d(\mathcal{O} / \mathfrak{p}^{\vartheta}) \vert = q^{\frac{(\vartheta - 1) d(d-1)}{2}}\vert \Ortho_d(\mathcal{O} / \mathfrak{p}) \vert
\end{equation*}
with
\begin{equation*}
\vert \Ortho_d(\mathcal{O} / \mathfrak{p}) \vert = 2 q^{\lfloor \frac{d}{2} \rfloor(\lfloor \frac{d}{2} \rfloor - 1)} q^{\lfloor \frac{d}{2} \rfloor} q^{\lfloor \frac{d - 1}{2} \rfloor (\lfloor \frac{d - 1}{2} \rfloor + 1)} \prod_{i = 1}^{\lfloor \frac{d - 1}{2} \rfloor} (1 - q^{-2i}) (1 - \epsilon q^{- \lfloor \frac{d}{2} \rfloor})
\end{equation*}
for some $\epsilon \in \lbrace -1, 0 , 1\rbrace$. In particular
\begin{equation*}
\vert \Ortho_d(\mathcal{O} / \mathfrak{p}^{\vartheta}) \vert = q^{\vartheta \frac{d(d-1)}{2} + o(1)}.
\end{equation*}
\end{lemma}
\begin{proof}
\cite[Corollary 5.6]{Herman} and \cite[Chapter 11]{TaylorGroups}.
\end{proof}
\subsection*{Unitary group}
Let $p \geq 3$ be an integer. Let $F$ be a $p$-adic field with ring of integers $\mathcal{O}$, maximal ideal $\mathfrak{p}$ and let $q$ denote the cardinal of the residual field $\mathcal{O} / \mathfrak{p}$. Let $E/F$ be an unramified quadratic extension with ring of integer $\mathcal{O}_E$ and maximal ideal $\mathfrak{p}_E$.
\begin{lemma} \label{lemA03}
Let $d, \vartheta \geq 1$ be integers, then
\begin{equation*}
\vert \U_d(\mathcal{O}_E / \mathfrak{p}_E^{\vartheta}) \vert = q^{(\vartheta - 1) d^2}\vert \U_d(\mathcal{O}_E / \mathfrak{p}_E) \vert
\end{equation*}
with
\begin{equation*}
\vert \U_d (\mathcal{O}_E / \mathfrak{p}_E) \vert = q^{d^2} \prod_{i = 1}^d (1 - (-q)^{-i}).
\end{equation*}
\end{lemma}
\begin{proof}
\cite[Corollary 5.6]{Herman} and \cite[Chapter 10]{TaylorGroups}.
\end{proof}
\subsection*{Application}
\begin{lemma}
With notations of Section \ref{sec1.2}, we have
\begin{equation*}
E(\mathfrak{p}^k) = \begin{cases}
      \frac{2m + 1}{2n} & \text{if $E = F$} \\
	\frac{m}{n} & \text{if $E \neq F$} \\
    \end{cases}\,.
\end{equation*}
\end{lemma}
\begin{proof}
We have
\begin{align*}
[K_f : K_H(\mathfrak{p}^k) ] = \frac{[K_f : K(\mathfrak{p}^k) ]}{[K_H(\mathfrak{p}^k) : K(\mathfrak{p}^k)]}
\end{align*}
with
\begin{align*}
[K_H(\mathfrak{p}^k) : K(\mathfrak{p}^k)] &= [ \big(H(\mathbb{A}_f) \cap K_f \big) \cdot K(\mathfrak{p}^k) : K(\mathfrak{p}^k)] \\
&= [H(\mathbb{A}_f) \cap K_f : \big( H(\mathbb{A}_f) \cap K_f \big) \cap K(\mathfrak{p}^k)] \\
&= [H(\mathbb{A}_f) \cap K_f : H(\mathbb{A}_f) \cap K(\mathfrak{p}^k)].
\end{align*}
This way, from Lemma \ref{lemA01}, \ref{lemA02} and \ref{lemA03}, we have
\begin{equation*}
\vol (X_{H}(\mathfrak{p}^k)) = \begin{cases}
       	q^{\vartheta \big( \frac{(n+m)(n+m-1)}{2} - (\frac{n(n-1)}{2} + \frac{m(m-1)}{2}) \big) + o(1)} = q^{\vartheta nm + o(1)} & \text{if $E = F$} \\
	q^{\vartheta \big( (n + m)^2 - (n^2 + m^2) \big) + o(1)} = q^{\vartheta 2nm + o(1)} & \text{if $E \neq F$} \\
    \end{cases}
\end{equation*}
and
\begin{equation*}
\vol (X'(\mathfrak{p}^k)) = \begin{cases}
       q^{\vartheta (2 m^2 + m) + o(1)} &  \text{if $E = F$} \\
	q^{\vartheta (2m)^2 + o(1)} & \text{if $E \neq F$} \\
    \end{cases}\,.
\end{equation*}
Combining these, we find
\begin{align*}
\dfrac{\vol (X_{H}(\mathfrak{p}^k))}{\vol (X'(\mathfrak{p}^k))} &= \begin{cases}
        q^{\vartheta (nm - 2m^2 - m) + o(1)} &  \text{if $E = F$} \\
	q^{\vartheta (2nm - (2m)^2 ) + o(1)} &  \text{if $E \neq F$} \\
    \end{cases} \\
&= \begin{cases}
       q^{\vartheta nm \big( 1 - \frac{2m^2 + m}{nm} \big) + o(1)} &  \text{if $E = F$} \\
	q^{\vartheta 2nm \big( 1 - \frac{(2m)^2}{2nm} \big) + o(1)} &  \text{if $E \neq F$} \\
    \end{cases}\\
&= \begin{cases}
       \vol (X_{H}(\mathfrak{p}^k))^{1 - \frac{2m + 1}{n} + o(1)} &  \text{if $E = F$} \\
	\vol (X_{H}(\mathfrak{p}^k))^{1 - \frac{2m}{n} + o(1)} &  \text{if $E \neq F$} \\
    \end{cases}\,.
\end{align*}
\end{proof}

\newpage
\bibliographystyle{alpha}
\bibliography{biblio}

\end{document}